\documentclass[10pt]{amsart}
\pdfoutput=1
\usepackage{amssymb,amscd,graphicx,enumerate,epsfig,color}
\usepackage[bookmarks=true]{hyperref}

\newcommand{\ClaimN}[2]{\noindent\textbf{Claim #1 }\textit{#2}\\}
\newcommand{\ClaimNN}[2]{\noindent\textbf{Claim #1 }\textit{#2}}
\newtheorem{theorem}{Theorem}[section]

\newtheorem{proposition}[theorem]{Proposition}
\newtheorem{corollary}[theorem]{Corollary}

\newtheorem{lemma}[theorem]{Lemma}

\theoremstyle{definition}
\newtheorem{definition}[theorem]{Definition}

\newcommand{\D}{\mathcal{D}}
\newcommand{\DV}{\mathcal{D}_\mathcal{V}}
\newcommand{\DW}{\mathcal{D}_\mathcal{W}}
\newcommand{\DVW}{\mathcal{D}_{\mathcal{VW}}}
\newcommand{\DVPWP}{\mathcal{D}_{\mathcal{V'W'}}}
\newcommand{\DVWi}{\mathcal{D}_{\mathcal{V}_i\mathcal{W}_i}}

\newenvironment{proofN}[1]{\noindent\textit{Proof of #1.}}{\hfill$\square$\\}

\newcommand{\Case}[1]{\textbf{Case #1.}}
\newcommand{\Step}[1]{\textbf{Step #1.}}

\newcommand{\V}{\mathcal{V}}
\newcommand{\W}{\mathcal{W}}

\begin{document}

\title[Generalized Heegaard splittings and mapping classes]{On the disk complexes of weakly reducible, unstabilized Heegaard splittings of genus three III - Generalized Heegaard splittings and mapping classes}

\author{Jungsoo Kim}
\date{September 1, 2015}

\begin{abstract}
	Let $M$ be an orientable, irreducible $3$-manifold admitting a weakly reducible genus three Heegaard splitting as a minimal genus Heegaard splitting.
	In this article, we prove that if  $[f]$, $[g]\in Mod(M)$ give the same correspondence between two isotopy classes of generalized Heegaard splittings consisting of two Heegaard splittings of genus two, say $[\mathbf{H}]\to[\mathbf{H}']$,  then there exists a representative $h$ of the difference $[h]=[g]\cdot[f]^{-1}$ such that (i) $h$ preserves a suitably chosen embedding of the Heegaard surface $F'$ obtained by amalgamation from $\mathbf{H}'$ which is a representative of $[\mathbf{H}']$ and (ii) $h$ sends a uniquely determined weak reducing pair $(V',W')$ of $F'$ into itself up to isotopy.
Moreover, for every orientation-preserving automorphism $\tilde{h}$ satisfying the previous conditions (i) and (ii), there exist two elements of $Mod(M)$ giving correspondence $[\mathbf{H}]\to[\mathbf{H}']$ such that $\tilde{h}$ belongs to the isotopy class of the difference between them. 
\end{abstract}

\address{\parbox{4in}{
	BK21 PLUS SNU Mathematical Sciences Division,\\ Seoul National University\\ 
	1 Gwanak-ro, Gwanak-Gu, Seoul 08826, Korea\\
}} 
	
\email{pibonazi@gmail.com}

\maketitle
\section{Introduction and Result}
Throughout this paper, all surfaces and 3-manifolds will be taken to be compact, orientable and piecewise-linear.

Let $M$ be an orientable, irreducible $3$-manifold admitting a weakly reducible genus three Heegaard splitting as a minimal genus Heegaard splitting.

Let us consider an element $[f]$ of the group of isotopy classes of orientation-preserving automorphisms of $M$, say $Mod(M)$, and an automorphism $f$ in the isotopy class $[f]$.
Let $[F]$ be the isotopy class of a properly embedded (possibly disconnected) surface $F$ in $M$.
Since we can well-define the image $[f]([F])$ as $[f(F)]$ for an isotopy class $[F]$ and an element $[f]\in Mod(M)$,  if there is a correspondence $[f]([F])=[F']$ between two isotopy classes $[F]$ and $[F']$, then it would contain some information of $[f]$ even though it does not contain all information of $[f]$.
But if $F$ does not divide $M$ into sufficiently small pieces, then one can expect that the correspondence contains not much information and if the genus of $F$ is large, then it would be hard to even just find a correspondence.

Since $M$ admits a weakly reducible Heegaard splitting of genus three, we can get the generalized Heegaard splitting obtained by ``\textit{weak reduction}'', where it consists of two non-trivial Heegaard splittings of genus two.
Conversely, if there is a generalized Heegaard splitting of $M$ consisting of two non-trivial Heegaard splittings of genus two, then the ``\textit{amalgamation}'' is a weakly reducible, genus three Heegaard splitting of $M$.
Hence, we can make use of the correspondences between sets of surfaces in $M$ of genera at most two instead of surfaces of genus three or more.
Since there have been many results about genus two Heegaard splittings, this approach would make sense.

But the question is, how much information of elements of $Mod(F)$ could be contained in a correspondence between two isotopy classes of generalized Heegaard splittings consisting of two Heegaard splittings of genus two?
For $[f],[g]\in Mod(M)$ and a generalized Heegaard splitting $\mathbf{H}$ consisting of two Heegaard splittings of genus two, assume that  $[f]([\mathbf{H}])=[g]([\mathbf{H}])$.
Even though the set of surfaces $f(\mathbf{H})$ is isotopic to $g(\mathbf{H})$, we cannot guarantee that $[f]=[g]$ in $Mod(M)$, i.e. the difference $[h]=[g]\cdot[f]^{-1}$ might not be the identity in $Mod(M)$.
Since two generalized Heegaard splittings $f(\mathbf{H})$ and $g(\mathbf{H})$ are isotopic, we could expect that the amalgamations of them are also isotopic.
Hence, there comes a natural expectation that there would be a representative $h$ of the difference $[h]$ such that $h$ preserves an embedding $F'$ of the amalgamation obtained from $f(\mathbf{H})$.
Hence, there would be the corresponding subset or subgroup of $Mod(M,F')$ containing such representatives of $[h]$ ($Mod(M,F')$ is the group of isotopy classes of orientation-preserving automorphisms of $M$ preserving $F'$) and this subset or subgroup would tell us how much information the correspondence loses for such elements of $Mod(M)$.

First, we will show that ``\textit{whether or not $[f]$ gives a correspondence between two weakly reducible, unstabilized Heegaard surfaces of genus three}'' can be interpreted as ``\textit{whether or not there exists a correspondence between two generalized Heegaard splittings obtained by weak reductions from them by $[f]$}''  in Theorem \ref{theorem-GHS-HS}.
This gives an important motivation to understand $[f]$ as a correspondence between two generalized Heegaard splittings instead of two Heegaard splittings of genus three.

\begin{theorem}[Corollary \ref{corollary-GHS-HS}]\label{theorem-GHS-HS}
Let $(\V,\W;F)$ and $(\V',\W';F')$ be weakly reducible, unstabilized, genus three Heegaard splittings in an orientable, irreducible $3$-manifold $M$ and $f$ an orientation-preserving automorphism of $M$.
Then $f$ sends $F$ into $F'$ up to isotopy if and only if $f$ sends a generalized Heegaard splitting obtained by weak reduction from $(\V,\W;F)$ into a generalized Heegaard splitting obtained by weak reduction from $(\V',\W';F')$ up to isotopy.
\end{theorem}

Let $\widetilde{\mathcal{GHS}}$ be the set of isotopy classes of the generalized Heegaard splittings consisting of two non-trivial Heegaard splittings of genus two  and $\widetilde{\mathcal{GHS}}_{[F]}$ the maximal subset of $\widetilde{\mathcal{GHS}}$ such that every element of $\widetilde{\mathcal{GHS}}_{[F]}$  gives the same isotopy class $[F]$ of amalgamation.

Next, we will prove Theorem \ref{theorem-main-2} which is the main theorem in this article.

\begin{theorem}[Theorem \ref{theorem-main-copy-2}, the Main Theorem] \label{theorem-main-2}
Let $M$ be an orientable, irreducible $3$-manifold having a weakly reducible,  genus three Heegaard splitting as a minimal genus Heegaard splitting.

Suppose that there is a correspondence between (possibly duplicated) two isotopy classes of $\widetilde{\mathcal{GHS}}$ by some elements of $Mod(M)$, say $[\mathbf{H}]\in \widetilde{\mathcal{GHS}}_{[F]}\rightarrow [\mathbf{H}']\in \widetilde{\mathcal{GHS}}_{[F']}$.
If $[f]$, $[g]\in Mod(M)$ give the same correspondence,  then there exists a representative $h$ of the difference $[h]=[g]\cdot[f]^{-1}$ satisfying the follows.

For a suitably chosen representative $F'\in[F']$,
\begin{enumerate}
\item $h$ takes $F'$ into itself and\label{thm12-1}
\item $h$ sends a uniquely determined weak reducing pair $(V',W')$ of $F'$ into itself up to isotopy (i.e. $h(V')$ is isotopic to $V'$ or $W'$ in the relevant compression body and $h(W')$ is isotopic to the other in the relevant compression body),
where $(V',W')$ is determined naturally when we obtain $F'$ by amalgamation from a representative $\mathbf{H}'$ of $[\mathbf{H}']$.\label{thm12-2} 
\end{enumerate}
Moreover, for any orientation-preserving automorphism $\tilde{h}$ of $M$ satisfying (\ref{thm12-1}) and (\ref{thm12-2}), there exist two elements in $Mod(M)$ giving the correspondence $[\mathbf{H}]\to[\mathbf{H}']$ such that $\tilde{h}$ belongs to the isotopy class corresponding to the difference between them.
\end{theorem} 

Hence, the Main Theorem means that the difference between such two elements of $Mod(M)$ comes from the subgroup of $Mod(M,F')$ consisting of  elements preserving the weak reducing pair $(V',W')$, say $Mod(M,F',(V',W'))$.

\section{Preliminaries\label{section2}}

This section introduces basic notations and summarizes the author's results in \cite{JungsooKim2013} \cite{JungsooKim2012} \cite{JungsooKim2014} \cite{JungsooKim2014-2}.

\begin{definition}\label{def-isotopy}
Let $M$ be a manifold.
An \textit{ambient isotopy} taking $N$ into $N'$ is a family of maps $h_t:M\to M$, $t\in I$ such that the associated map $H:M\times I \to M$ given by $H(x,t)=h_t(x)$ is continuous, $h_0$ is the identity, $h_1(N)=N'$, and  $h_t$ is a homeomorphism from $M$ to itself at any time $0\leq t \leq 1$.

In this article, we just say $N$ is \textit{isotopic} to $N'$ in $M$ by an isotopy $h_t$ if there is an ambient isotopy $h_t$ taking $N$ into $N'$. 

An \textit{isotopy} between two homeomorphisms $f, g : X \rightarrow Y$ for two manifolds $X$ and $Y$ is a family of maps $f_t: X\to Y$, $t\in I$ such that the associated map $F:X\times I\to Y$ given by $F(x,t)=f_t(x)$ is continuous, $f_0=f$, $f_1=g$, and $f_t$ is a homeomorphism at any time $0\leq t \leq 1$.

Let $f: X\to Y$  be a homeomorphism such that $f(N)=N_1$ for a submanifold $N\subset X$.
If there is an isotopy $f_t$ such that $f_0=f$ and $f_1(N)=N_2$, then we say that ``\textit{we can isotope $f$ so that $f(N)=N_2$}''.
For example, if $N_1 ~(=f(N))$ itself is isotopic to $N_2$ by an isotopy $h_t$ in $Y$, then we can isotope $f$ so that $f(N)=N_2$ by taking the isotopy $f_t=h_t\circ f$. 
If we can isotope $f$ so that $f(N)=N'$, then we say that ``\textit{$f$ takes (or sends) $N$ into $N'$ up to isotopy}''.
If a homeomorphism $f$ is isotopic to $g$, then we say that $f$ and $g$ belong to the same \textit{isotopy class}, where we will denote the isotopy class of a homeomorphism $f$ as $[f]$.
If we assume $X=Y=M$, then there is the set of isotopy classes of orientation-preserving automorphisms of $M$, say $Mod(M)$.
Then we can well-define the operation $[f]\cdot [g]$ as $[f\circ g]$ and this gives a group structure on $Mod(M)$ with the identity $[\operatorname{id}_M]$ and the inverse $[f]^{-1}=[f^{-1}]$.

Suppose that $f$ is an orientation-preserving automorphism of $M$.
If a submanifold $F_1$ is isotopic to $F_2$ in $M$, i.e. $h_0(F_1)=F_1$ and $h_1(F_1)=F_2$ by an isotopy $h_t$ for $0\leq t \leq 1$, then the image $f(F_1)$ is isotopic to $f(F_2)$ by the isotopy $f\circ{h_t}\circ f^{-1}$ for $0\leq t \leq 1$.
Moreover, if $f$ is isotopic to $f'$ by an isotopy $f_t$ for $0\leq t \leq 1$ for two representatives $f$ and $f'$ of $[f]$, then the isotopy $f_t\circ f^{-1}$ for $0\leq t \leq 1$ sends $f(F)$ into $f'(F)$.
This means that we can well-define the image $[f]([F])$ as $[f(F)]$ for an isotopy class $[F]$ and an element $[f]\in Mod(M)$.
\end{definition}

\begin{definition}\label{def-compression}
A  \textit{compression body} is a $3$-manifold which can be obtained by starting with some closed, orientable, connected surface $F$, forming the product $F\times I$, attaching some number of $2$-handles to $F\times\{1\}$ and capping off all  resulting $2$-sphere boundary components that are not contained in $F\times\{0\}$ with $3$-balls. 
The boundary component $F\times\{0\}$ is referred to as $\partial_+$. 
The rest of the boundary is referred to as $\partial_-$.
If a compression body $\V$ is homeomorphic to $\partial_+\V \times I$, then we call it \textit{trivial} and otherwise we call it \textit{nontrivial}. 
The cores of the $2$-handles defining a compression body $\V$, extended vertically down through $\partial_+\V\times I$, are called a \textit{defining set} of $2$-disks for $\V$.
A defining set for $\V$ is \textit{minimal} if it properly contains no other defining set.

Note that we can define a compression body $\V$ with non-empty minus boundary as a connected $3$-manifold obtained from $F\times I$ for a (possibly disconnected) closed surface $F$ such that each component of $F$ is of genus at least one, followed by $1$-handles attached to $F\times \{1\}$, where $F\times\{0\}$ becomes $\partial_-\V$ and the other boundary of $\V$ becomes $\partial_+\V$.
\end{definition}

\begin{lemma}\label{lemma-defining}
A genus $g\geq 2$ compression body $\V$ with minus boundary having a genus $g-1$ component has a unique minimal defining set up to isotopy and it consists of only one disk.
\end{lemma}

\begin{proof}
If $\partial_-\V$ is connected, i.e. $\partial_-\V$ consists of a genus $g-1$ surface, then there is a unique non-separating disk in $\V$ up to isotopy.
If $\partial_-\V$ is disconnected, i.e. $\partial_-\V$ consists of a genus $g-1$ surface and a torus, then there is a unique compressing disk in $\V$ up to isotopy, where it is separating in $\V$.
Moreover, if we cut $\V$ along the uniquely determined disk, then we get $\partial_-\V\times I$ in any case.
Therefore, we can obtain $\V$ by attaching only one $1$-handle to $\partial_-\V\times I$ corresponding to the disk.
This gives a way to determine $\V$ by attaching only one $2$-handle to $\partial_+\V\times I$ and therefore the relevant defining set is the singleton set consisting of the disk.
Since this defining set consists of only one disk, it is a minimal defining set.
Moreover, if there is a minimal defining set for $\V$, i.e. it consists of a disk, then the disk must be a compressing disk of $\V$ otherwise the resulting compression body would be trivial.
Hence, it must consist of a non-separating disk (if $\partial_-\V$ is connected) or a separating compressing disk (if $\partial_-\V$ is disconnected) by considering the shape of the resulting minus boundary.
Hence, a minimal defining set for $\V$ is uniquely determined up to isotopy by the argument in the start of the proof.

This completes the proof.
\end{proof}

\begin{definition}\label{def-spine}
A \textit{spine} of a compression body $\V$ is a graph $\sigma$ embedded in $\V$ with some valence-one vertices possibly embedded in $\partial_-\V$ such that $\V-\eta(\sigma)$ is homeomorphic to $\partial_+\V \times [0,1]$ where $\eta(\sigma)$ is an open  regular neiborhood of $\sigma$.
A spine $\sigma$ of $\V$ is \textit{minimal} if it is a union of arcs, each of which has both ends on $\partial_-\V$ (or at a single vertex if $\V$ is a handlebody).

A spine $\sigma$ of a compression body $\V$ is \textit{dual} to a defining set $\Delta$ for $\V$ if each edge of $\sigma$ intersects a single disk of $\Delta$ exactly once, each disk of $\Delta$ intersects exactly one edge of $\sigma$, and each ball of $\V-\Delta$ contains exactly one vertex of $\sigma$, and all vertices of $\sigma$ in $\partial_-\V\times I$ component of $\V-\Delta$ are contained in $\partial_-\V$. 
\end{definition}

\begin{definition}
A \textit{Heegaard splitting} of a $3$-manifold $M$ is an expression of $M$ as a union $\V\cup_F \W$, denoted   as $(\V,\W;F)$ (or $(\V,\W)$ simply),  where $\V$ and $\W$ are compression bodies that intersect in a transversally oriented surface $F=\partial_+\V=\partial_+\W$. 
We say $F$ is the \textit{Heegaard surface} of this splitting. 
If $\V$ or $\W$ is homeomorphic to a product, then we say the splitting  is \textit{trivial}. 
If there are compressing disks $V\subset \V$ and $W\subset \W$ such that $V\cap W=\emptyset$, then we say the splitting is \textit{weakly reducible} and call the pair $(V,W)$ a \textit{weak reducing pair}. 
If $(V,W)$ is a weak reducing pair and $\partial V$ is isotopic to $\partial W$ in $F$, then we call $(V,W)$ a \textit{reducing pair}.
If the splitting is not trivial and we cannot take a weak reducing pair, then we call the splitting \textit{strongly irreducible}. 
If there is a pair of compressing disks $(\bar{V},\bar{W})$ such that $\bar{V}$ intersects $\bar{W}$ transversely in a point in $F$, then we call this pair a \textit{canceling pair} and say the splitting is \textit{stabilized}. 
Otherwise, we say the splitting is \textit{unstabilized}.
\end{definition}

\begin{definition}
Let $F$ be a surface of genus at least two in a compact, orientable $3$-manifold $M$. 
Then the \emph{disk complex} $\D(F)$ is defined as follows: 
\begin{enumerate}[(i)]
\item Vertices of $\D(F)$ are isotopy classes of compressing disks for $F$.
\item A set of $m+1$ vertices forms an $m$-simplex if there are representatives for each
that are pairwise disjoint.
\end{enumerate}
Hence, two compressing disks $D_1$ and $D_2$ of $F$ correspond to the same vertex in $\D(F)$ if and only if there exists an isotopy $h_t$ defined on $M$ such that (i) $h_0=\operatorname{id}$, (ii) $h_1(D_1)=D_2$, and  (iii) $h_t(F)=F$ for $0\leq t \leq 1$.
\end{definition}

\begin{definition}\label{def-DVWF}
Consider a Heegaard splitting $(\V,\W;F)$ of an orientable, irreducible $3$-manifold $M$. 
Let $\DV(F)$ and $\DW(F)$ be the subcomplexes of $\D(F)$ spanned by compressing disks in $\V$ and $\W$ respectively. 
We call these subcomplexes \textit{the disk complexes of $\V$ and $\W$}.
Let $\DVW(F)$ be the subset  of $\D(F)$ consisting of the simplices having at least one vertex from $\DV(F)$ and at least one vertex from $\DW(F)$.
We will denote the isotopy class $[V]\in \DV(F)$ as $V\subset \V$ or $V\subset\DV(F)$ for the sake of convenience if there is no confusion.
\end{definition}

\begin{theorem}[McCullough, \cite{8}]
$\DV(F)$ and $\DW(F)$ are contractible.
\end{theorem}

From now on, we will consider only unstabilized Heegaard splittings of an irreducible $3$-manifold. 
If a Heegaard splitting of a compact $3$-manifold is reducible, then the manifold is reducible or the splitting is stabilized (see \cite{SaitoScharlemannSchultens2005}).
Hence, we can exclude the possibilities of reducing pairs among weak reducing pairs.

\begin{definition}
Suppose $W$ is a compressing disk for $F\subset M$. 
Then there is a subset of $M$ that can be identified with $W\times I$ so that $W=W\times\{\frac{1}2\}$ and $F\cap(W\times I)=(\partial W)\times I$. 
We form the surface $F_W$, obtained by \textit{compressing $F$ along $W$}, by removing $(\partial W)\times I$ from $F$ and replacing it with $W\times(\partial I)$. 
We say the two disks $W\times(\partial I)$ in $F_W$ are the $\textit{scars}$ of $W$. 
\end{definition}

\begin{lemma}[Lustig and Moriah, Lemma 1.1 of \cite{7}] \label{lemma-2-8}
Suppose that $M$ is an irreducible $3$-manifold and $(\V,\W;F)$ is an unstabilized Heegaard splitting of $M$. 
If $F'$ is obtained by compressing $F$ along a collection of pairwise disjoint disks, then no $S^2$ component of $F'$ can have scars from disks in both $\V$ and $\W$. 
\end{lemma}

If we add the assumption that the genus of the Heegaard splitting is three, then we get the following important lemma.

\begin{lemma}[J. Kim, Lemma 2.9 of \cite{JungsooKim2013}]\label{lemma-2-9}
Suppose that $M$ is an irreducible $3$-manifold and $(\V,\W;F)$ is an unstabilized, genus three  Heegaard splitting of $M$. 
If there exist three mutually disjoint compressing disks $V$, $V'\subset\V$ and $W\subset \W$, then either $V$ is isotopic to $V'$, or one of $\partial V$ and $\partial V'$ bounds a punctured torus $T$ in $F$ and  the other is a non-separating loop in $T$. 
Moreover, we cannot choose three weak reducing pairs $(V_0, W)$, $(V_1,W)$, and $(V_2,W)$ such that $V_i$ and $V_j$ are mutually disjoint and non-isotopic in $\V$ for $i\neq j$. 
\end{lemma}

Note that ``\textit{one of $\partial V$ and $\partial V'$ bounds a punctured torus $T$ in $F$ and  the other is a non-separating loop in $T$}'' means that one of $V$ and $V'$, say $V'$, cuts off a solid torus from $\V$ and $V$ is a meridian disk of the solid torus and therefore $V'$ is a band sum of two parallel copies of $V$ in $\V$.

\begin{definition}[J. Kim, Definition 2.12 of \cite{JungsooKim2012}]
In a weak reducing pair for a Heegaard splitting $(\V,\W;F)$, if a disk belongs to $\V$, then we call it a \emph{$\V$-disk}.
Otherwise, we call it a \emph{$\W$-disk}.	
We call a $2$-simplex in $\DVW(F)$ represented by two vertices in $\DV(F)$ and one vertex in $\DW(F)$ a \textit{$\V$-face}, and also define a \textit{$\W$-face} symmetrically.
Let us consider a $1$-dimensional graph as follows.
\begin{enumerate}
\item We assign a vertex to each $\V$-face in $\DVW(F)$.
\item If a $\V$-face shares a weak reducing pair with another $\V$-face, then we assign an edge between these two vertices in the graph.
\end{enumerate}
We call this graph the \emph{graph of $\V$-faces}.
If there is a maximal subset $\varepsilon_\V$ of $\V$-faces in $\DVW(F)$ representing a connected component of the graph of $\V$-faces and the component is not an isolated vertex, then we call $\varepsilon_\V$ a \emph{$\V$-facial cluster}.
Similarly, we define the \emph{graph of $\W$-faces} and  a \textit{$\W$-facial cluster}.
In a $\V$-facial cluster, every weak reducing pair gives the common $\W$-disk, and vise versa.
\end{definition}

If we consider an unstabilized, genus three Heegaard splitting of an irreducible $3$-manifold, then we get the following lemmas.

\begin{lemma}[J. Kim, Lemma 2.13 of \cite{JungsooKim2012}]\label{lemma-2-14}
Suppose that $M$ is an irreducible $3$-manifold and $(\V,\W;F)$ is an unstabilized, genus three Heegaard splitting of $M$.
If there are two $\V$-faces $f_1$ represented by $\{V_0,V_1,W\}$ and $f_2$ represented by $\{V_1, V_2, W\}$ sharing a weak reducing pair $(V_1,W)$, then $\partial V_1$ is non-separating, and $\partial V_0$, $\partial V_2$ are separating in $F$.
Therefore, there is a unique weak reducing pair in a $\V$-facial cluster which can belong to two or more faces in the $\V$-facial cluster.
\end{lemma}

\begin{definition}[J. Kim, Definition 2.14 of \cite{JungsooKim2012}]\label{definition-2-15}
By Lemma \ref{lemma-2-14}, there is a unique weak reducing pair in a $\V$-facial cluster belonging to two or more faces in the $\V$-facial cluster.
We call it the \textit{center} of a $\V$-facial cluster.
We call the other weak reducing pairs \textit{hands} of a $\V$-facial cluster.
See Figure \ref{figure3}.
Note that if a $\V$-face in a $\V$-facial cluster is represented by two weak reducing pairs, then one is the center and the other is a hand.
Lemma \ref{lemma-2-14} means that the $\V$-disk in the center of a $\V$-facial cluster is non-separating, and those from hands are all separating.
Moreover, Lemma \ref{lemma-2-9} implies that (i) the $\V$-disk in a hand of a $\V$-facial cluster is a band sum of two parallel copies of that of the center of the $\V$-facial cluster and (ii) the $\V$-disk of a hand of a $\V$-facial cluster determines that of the center of the $\V$-facial cluster by the uniqueness of the meridian disk of the solid torus which the $\V$-disk of the hand cuts off from $\V$.
\end{definition}

\begin{figure}
\includegraphics[width=4.5cm]{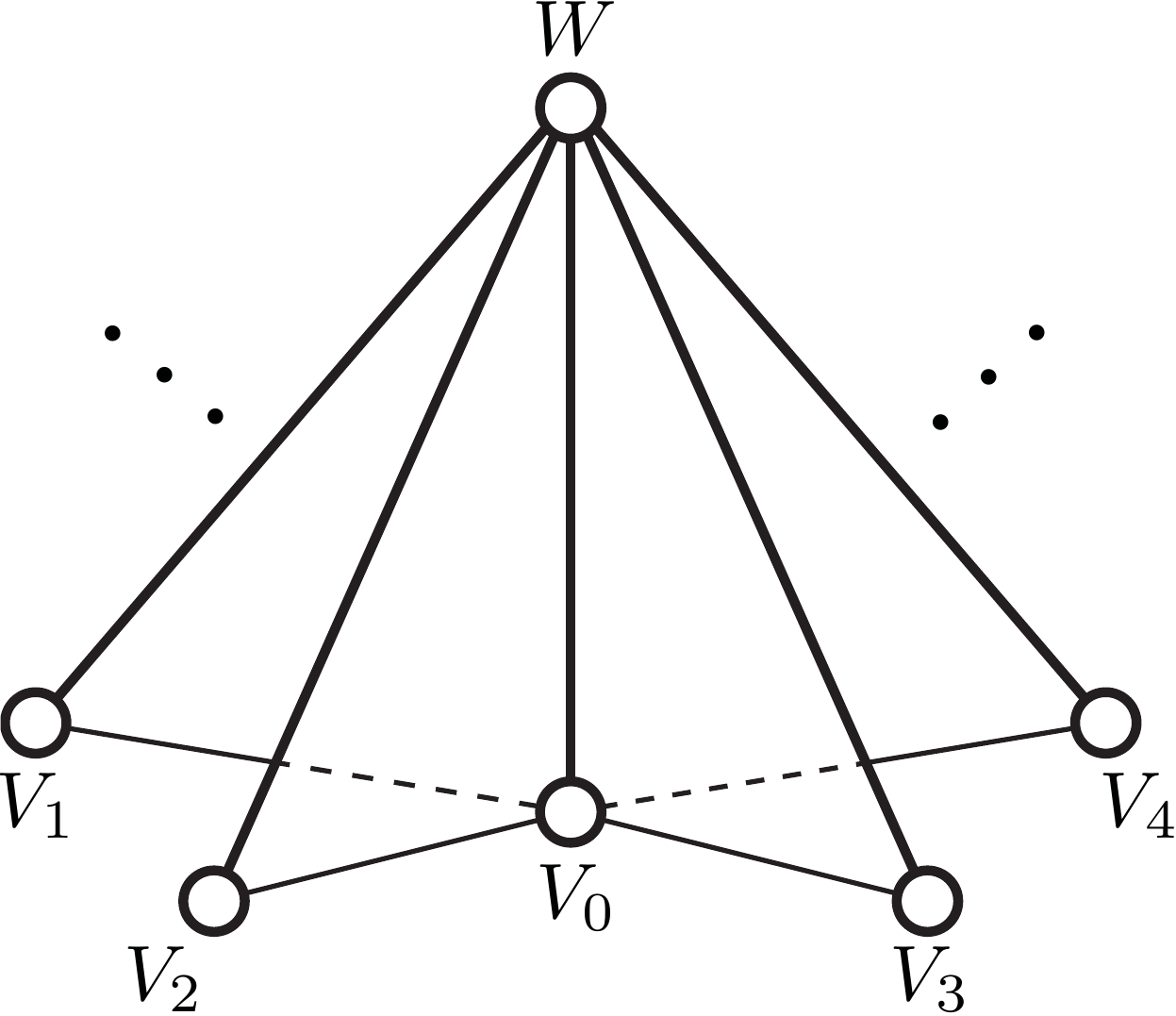}
\caption{An example of a $\V$-facial cluster in $\DVW(F)$. $(V_0, W)$ is the center and the other weak reducing pairs are hands. \label{figure3}}
\end{figure}

\begin{lemma}[J. Kim, Lemma 2.15 of \cite{JungsooKim2012}]\label{lemma-2-16} 
Assume $M$ and $F$ as in Lemma \ref{lemma-2-14}.
Every $\V$-face belongs to some $\V$-facial cluster. 
Moreover, every $\V$-facial cluster has infinitely many hands.
\end{lemma}

The next is the definition of ``\textit{generalized Heegaard splitting}'' originated from \cite{ScharlemannThompson1994}.

\begin{definition}[Definition 4.1 of \cite{Bachman2008}]\label{definition-2-9}
A \textit{generalized Heegaard splitting} (GHS) $\mathbf{H}$ of a $3$-manifold $M$ is a pair of sets of pairwise disjoint, transversally oriented, connected surfaces, $\operatorname{Thick}(\mathbf{H})$ and $\operatorname{Thin}(\mathbf{H})$ (called the \textit{thick levels} and \textit{thin levels}, resp.), which satisfies the following conditions.
\begin{enumerate}
\item Each component $M'$ of $M-\operatorname{Thin}(\mathbf{H})$ meets a unique element $H_+$ of $\operatorname{Thick}(\mathbf{H})$ and $H_+$ is a Heegaard surface in $M'$.
Henceforth we will denote the closure of the component of $M-\operatorname{Thin}(\mathbf{H})$ that contains an element $H_+\in\operatorname{Thick}(\mathbf{H})$ as $M(H_+)$.
\item As each Heegaard surface $H_+\subset M(H_+)$ is transversally oriented, we can consistently talk about the points of $M(H_+)$ that are ``above''  $H_+$ or ``below'' $H_+$.
Suppose $H_-\in\operatorname{Thin}(\mathbf{H})$.
Let $M(H_+)$ and $M(H_+')$ be the submanifolds on each side of $H_-$.
Then $H_-$ is below $H_+$ if and only if it is above $H_+'$.
\item There is a partial ordering on the elements of $\operatorname{Thin}(\mathbf{H})$ which satisfies the following: Suppose $H_+$ is an element of $\operatorname{Thick}(\mathbf{H})$, $H_-$ is a component of $\partial M(H_+)$ above $H_+$ and $H_-'$ is a component of $\partial M(H_+)$ below $H_+$.
Then $H_->H_-'$.
\end{enumerate}
We denote the maximal subset of $\operatorname{Thin}(\mathbf{H})$ consisting of surfaces only in the interior of $M$ as $\overline{\operatorname{Thin}}(\mathbf{H})$ and call it  the \textit{inner thin levels}.
If the corresponding Heegaard splitting of $M(H_+)$ is not trivial for every $H_+\in\operatorname{Thick}(\mathbf{H})$, then we call $\mathbf{H}$ \textit{clean}.
\end{definition}

Note that a GHS in this article is the same as a \textit{pseudo-GHS} in \cite{Bachman2008} since we allow a GHS to have product compression bodies and we do not encounter thin $2$-spheres.

The next is the definition of ``\textit{generalized Heegaard splitting}'' originated from \cite{ScharlemannThompson1994}.

\begin{definition}[Bachman, a restricted version of Definition 5.2, Definition 5.3, and Definition 5.6 of \cite{Bachman2008}]\label{definition-WR}
Let $M$ be an orientable, irreducible 3-manifold.
Let $\mathbf{H}$ be an unstabilized Heegaard splitting of $M$, i.e. $\operatorname{Thick}(\mathbf{H})=\{F\}$ and $\operatorname{Thin}(\mathbf{H})$ consists of $\partial M$.
Let $V$ and $W$ be disjoint compressing disks of $F$ from the opposite sides of $F$ such that ${F}_{VW}$ has no $2$-sphere component. 
(Lemma \ref{lemma-2-8} guarantees that ${F}_{VW}$ does not have a $2$-sphere component.)
Define
$$\operatorname{Thick}(\mathbf{G'})=(\operatorname{Thick}(\mathbf{H})-\{F\})\cup\{{F}_{V}, {F}_{W}\}, \text{ and}$$
$$\operatorname{Thin}(\mathbf{G'})=\operatorname{Thin}(\mathbf{H})\cup\{{F}_{VW}\},$$
where we assume that each element of $\operatorname{Thick}(\mathbf{G'})$ belongs to the interior of $\V$ or $\W$ by slightly pushing off $F_V$ or $F_W$ into the interior of $\V$ or $\W$ respectively and then also assume that they miss $F_{VW}$.
We say the GHS $\mathbf{\mathbf{G'}}=\{\operatorname{Thick}(\mathbf{G'}),\operatorname{Thin}(\mathbf{G'})\}$ is obtained from $\mathbf{H}$ by \textit{pre-weak reduction} along $(V,W)$.
The relative position of the elements of $\operatorname{Thick}(\mathbf{G'})$ and $\operatorname{Thin}(\mathbf{G'})$ follows the order described in Figure \ref{figure2}.
\begin{figure}
\includegraphics[width=10cm]{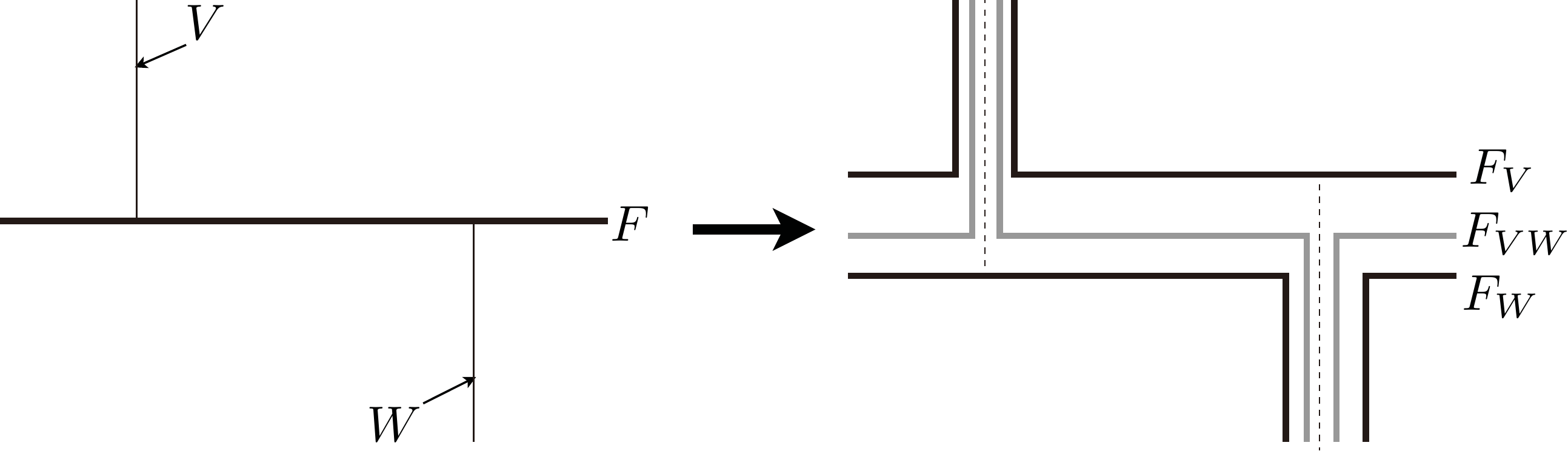}
\caption{pre-weak reduction \label{figure2}}
\end{figure}
If there are elements $S\in \operatorname{Thick}(\mathbf{G'})$ and $s\in \operatorname{Thin}(\mathbf{G'})$ that cobound a product region $P$ of $M$ such that $P\cap \operatorname{Thick}(\mathbf{G'})=S$ and $P\cap \operatorname{Thin}(\mathbf{G'})=s$ then remove $S$ from $\operatorname{Thick}(\mathbf{G'})$ and $s$ from $\operatorname{Thin}(\mathbf{G'})$.
This gives a clean GHS $\mathbf{G}$ of $M$ from the GHS $\mathbf{G'}$ (see Lemma 5.4 of \cite{Bachman2008}) and we say $\mathbf{G}$ is obtained from $\mathbf{G'}$ by \textit{cleaning}.
We say the clean GHS $\mathbf{G}$ of $M$ given by pre-weak reduction along $(V,W)$, followed by cleaning, is obtained from $\mathbf{H}$ by \textit{weak reduction} along $(V,W)$.
\end{definition}

The next is the definition of ``\textit{amalgamation}'' originated from \cite{Schultens1993}.
Since the original definition identifies the product structures near the relevant thin level into the thin level itself, the union of submanifolds after amalgamation is not exactly the same as the union before amalgamation setwisely.
Hence, we need to use another version of amalgamation.

\begin{definition} [The detailed version of ``\textit{partial amalgamation}'' of Section 3 of \cite{Lackenby2008} by using the terms in \cite{Schultens1993}]\label{def-amalgamation}
Let $N$ and $L$ be submanifolds of $M$ such that $N\cap L$ is a (possibly disconnected) closed surface $F$, where $F$ belongs to $\partial N$ and $\partial L$.
Suppose that $N$ and $L$ have non-trivial Heegaard splittings $(\V_1,\V_2;F_N)$ and $(\W_1,\W_2;F_L)$ respectively, where $\partial_-\V_2\cap \partial_-\W_1=F$.
Then we can represent $\V_2$ as the union of $\partial_-\V_2\times I$ and $1$-handles attached to $\partial_-\V_2\times \{1\}$ and the symmetric argument also holds for $\W_1$.
Especially, we can choose the product structures of the submanifolds $N_0= F\times I$ and $L_0= F\times I$ of $\partial_- \V_2\times I$ and $\partial_-\W_1\times I$ respectively (hence $N_0$ and $L_0$ share $F$ as the common $0$-level) such that the projections of attaching disks of the $1$-handles  defining $\V_2$ and $\W_1$ in the $1$-levels of $N_0$ and $L_0$ into $F$ would be mutually disjoint.
Let $\V_2=N_0\cup (\text{the }1\text{-handles})\cup N_1$ and $\W_1=L_0\cup(\text{the }1\text{-handles})\cup L_1$ ($N_1$ or $L_1$ might be empty).
Let $p_{N_0}:N_0\to F$ and $p_{L_0}:L_0\to F$ be the relevant projection functions defined in $N_0$ and $L_0$ respectively. 
Then we can extend the $1$-handles of $\V_2$ until we meet $F$ by using $p_{N_0}$ through $N_0$ and also we can extend those of $\W_1$ until we meet $F$ by using $p_{L_0}$ through $L_0$. 
Let $N_0'$ ($L_0'$ resp.) be the closure of the complement of the extended $1$-handles of $\V_2$ in $N_0$ ($\W_1$ in $L_0$ resp.).
Then we can see that $\V_1\cup N_0'$ is just expanded $\V_1$ vertically down through $N_0'$ and therefore it is a compression body and $\W_2\cup L_0'$ is also a compression body similarly.
If we define $\V=[\V_1\cup N_0']\cup [\text{the (possibly extended)}~~1\text{-handles of }\W_1]\cup L_1$ and $\W=[\W_2\cup L_0']\cup[\text{the (possibly extended)}~~1\text{-handles of }\V_2]\cup N_1$, then $(\V,\W)$ becomes a Heegaard splitting of $M$.
We call $(\V,\W)$ the \textit{amalgamation} of $(\V_1,\V_2)$ and $(\W_1,\W_2)$ along $F$ with respect to the given $1$-handle structures of $\V_2$ and $\W_1$ and the pair $(p_{N_0},p_{L_0})$ (see Figure \ref{fig-amalgamation}).
\begin{figure}
\includegraphics[width=10cm]{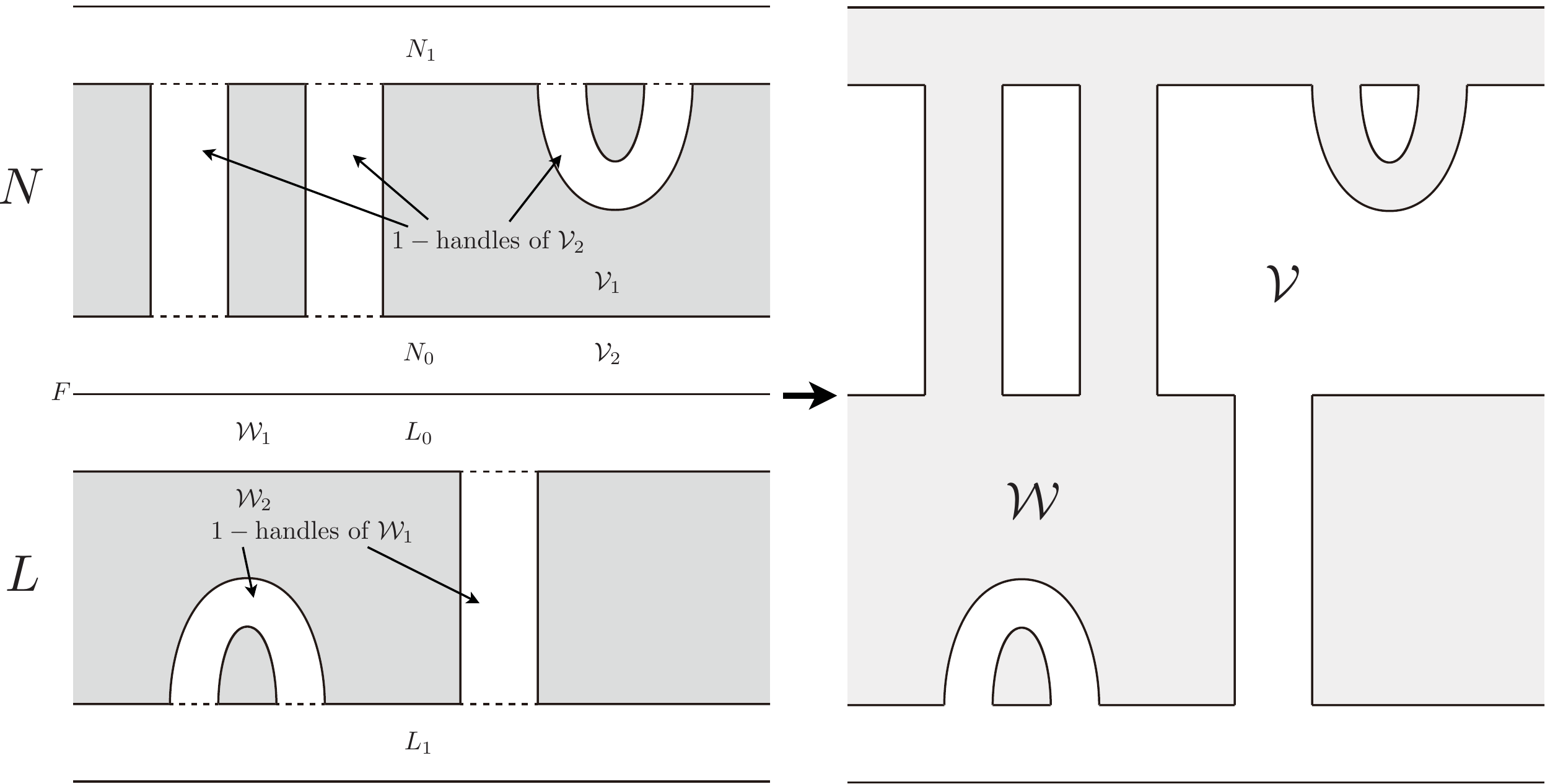}
\caption{the amalgamation of $(\V_1,\V_2)$ and $(\W_1,\W_2)$ along $F$ \label{fig-amalgamation}}
\end{figure}
\end{definition}

\begin{proposition}[Proposition 3.1 of \cite{Lackenby2008}]\label{lemma-Lackenby}
The amalgamation is well-defined up to ambient isotopy.
\end{proposition}

Despite of the existence of Proposition \ref{lemma-Lackenby}, we need the precise definition as in Definition \ref{def-amalgamation} since we will analyze the exact differences between representatives of generalized Heegaard splittings which induce the same amalgamation up to isotopy.\\

The following lemma means that the isotopy class of the generalized Heegaard splitting obtained by weak reduction along a weak reducing pair does not depend on the choice of the weak reducing pair if the weak reducing pair varies in a fixed $\V$- or $\W$-facial cluster.

\begin{lemma}[J. Kim, Lemma 2.17 of \cite{JungsooKim2014}]\label{lemma-2-17}
Assume $M$ and $F$ as in Lemma \ref{lemma-2-14}.
Every weak reducing pair in a $\V$-face gives the same generalized Heegaard splitting after weak reduction up to isotopy.
Therefore, every weak reducing pair in a $\V$-facial cluster gives the same generalized Heegaard splitting after weak reduction up to isotopy.
Moreover, the embedding of the thick level contained in $\V$ or $\W$ does not vary in the relevant compression body up to isotopy. 
\end{lemma}

The next lemma gives an upper bound for the dimension of $\DVW(F)$ and restricts the shape of a $3$-simplex in $\DVW(F)$.

\begin{lemma}[J. Kim, Proposition 2.10 of \cite{JungsooKim2013}]\label{lemma-2-18}
Assume $M$ and $F$ as in Lemma \ref{lemma-2-14}.
Then $\operatorname{dim}(\DVW(F))\leq 3$.
Moreover, if $\operatorname{dim}(\DVW(F))=3$, then every $3$-simplex in $\DVW(F)$ must have the form $\{V_1, V_2, W_1, W_2\}$, where $V_1, V_2\subset \V$ and $W_1,W_2\subset \W$.
Indeed, $V_1$ ($W_1$ resp.) is non-separating in $\V$ (in $\W$ resp.) and $V_2$ ($W_2$ resp.) is a band sum of two parallel copies of $V_1$ in $\V$ ($W_1$ in $\W$ resp.).
\end{lemma}

The next lemma characterizes the possible generalized Heegaard splittings obtained by weak reductions from $(\V,\W;F)$ into five types. 

\begin{lemma}[Lemma 3.1 of \cite{JungsooKim2014-2}]\label{lemma-four-GHSs}
Assume $M$ and $F$ as in Lemma \ref{lemma-2-14}.
Let $(\V_1,\V_2;\bar{F}_V)\cup_{\bar{F}_{VW}}(\W_1,\W_2;\bar{F}_W)$ be the generalized Heegaard splitting obtained by weak reduction along a weak reducing pair $(V,W)$ from the Heegaard splitting $(\V,\W;F)$, where $\partial_-\V_2\cap \partial_-\W_1=\bar{F}_{VW}$.
Then this generalized Heegaard splitting is one of the following five types (see Figure \ref{fig-Heegaard-a}).
\begin{enumerate}[(a)]
\item Each of $\partial_-\V_2$ and $\partial_-\W_1$ consists of a torus, where either\label{GHS-a}
	\begin{enumerate}[(i)]
	\item $V$ and $W$ are non-separating in $\V$ and $\W$ respectively and $\partial V\cup\partial W$ is also non-separating in $F$,
	\item $V$ cuts off a solid torus from $\V$ and $W$ is non-separating in $\W$,
	\item $W$ cuts off a solid torus from $\W$ and $V$ is non-separating in $\V$, or
	\item each of $V$ and $W$ cuts off a solid torus from $\V$ or $\W$.
	\end{enumerate}
	We call it a ``\textit{type (a) GHS}''.
\item One of $\partial_-\V_2$ and $\partial_-\W_1$ consists of a torus and the other consists of two tori, where either\label{lemma-3-1-b}
	\begin{enumerate}[(i)]
	\item $V$ cuts off $(\text{torus})\times I$ from $\V$ and $W$ is non-separating in $\W$,\label{lemma-3-1-b-i}
	\item $V$ cuts off $(\text{torus})\times I$ from $\V$ and $W$ cuts off a solid torus from $\W$,\label{lemma-3-1-b-ii}
	\item $W$ cuts off $(\text{torus})\times I$ from $\W$ and $V$ is non-separating in $\V$, or\label{lemma-3-1-b-iii}
	\item $W$ cuts off $(\text{torus})\times I$ from $\W$ and $V$ cuts off a solid torus from $\V$.\label{lemma-3-1-b-iv}
	\end{enumerate}	
	We call it a ``\textit{type (b)-$\W$ GHS}'' for (\ref{lemma-3-1-b-i}) and (\ref{lemma-3-1-b-ii}) and ``\textit{type (b)-$\V$ GHS}'' for (\ref{lemma-3-1-b-iii}) and (\ref{lemma-3-1-b-iv}).
\item Each of $\partial_-\V_2$ and $\partial_-\W_1$ consists of two tori but $\partial_-\V_2\cap \partial_-\W_1$ is a torus, where each of $V$ and $W$ cuts off $(\text{torus})\times I$ from $\V$ or $\W$.\label{lemma-3-1-c}
We call it a ``\textit{type (c) GHS}''.
\item Each of $\partial_-\V_2$ and $\partial_-\W_1$ consists of two tori and $\partial_-\V_2\cap \partial_-\W_1$ also consists of two tori, where both $V$ and $W$ are non-separating in $\V$ and $\W$ respectively but $\partial V\cup\partial W$ is separating in $F$.\label{lemma-3-1-d}
We call it a ``\textit{type (d) GHS}''.
\end{enumerate}
As the summary of the previous observations, the generalized Heegaard splitting $(\V_1,\V_2;\bar{F}_V)\cup_{\bar{F}_{VW}}(\W_1,\W_2;\bar{F}_W)$ is just a set of three surfaces $\{\bar{F}_V, \bar{F}_{VW}, \bar{F}_W\}$ obtained as the follows.
\begin{enumerate}
\item The thick level $\bar{F}_V$ ($\bar{F}_W$ resp.) is obtained by pushing the genus two component of $F_V$ ($F_W$ resp.) off into the interior of $\V$ (of $\W$ resp.) and \label{lemma-3-1-2nd-1}
\item The inner thin level $\bar{F}_{VW}$ is the union of components of $F_{VW}$ having scars of both $V$ and $W$,\label{lemma-3-1-2nd-2}
where we can see that if $\partial_-\V_2$ ($\partial_-\W_1$ resp.) has another component other than $\bar{F}_{VW}$, then it belongs to $\partial_-\W$ ($\partial_-\V$ resp.). 
\end{enumerate}
\end{lemma}

\begin{figure}
\includegraphics[width=12cm]{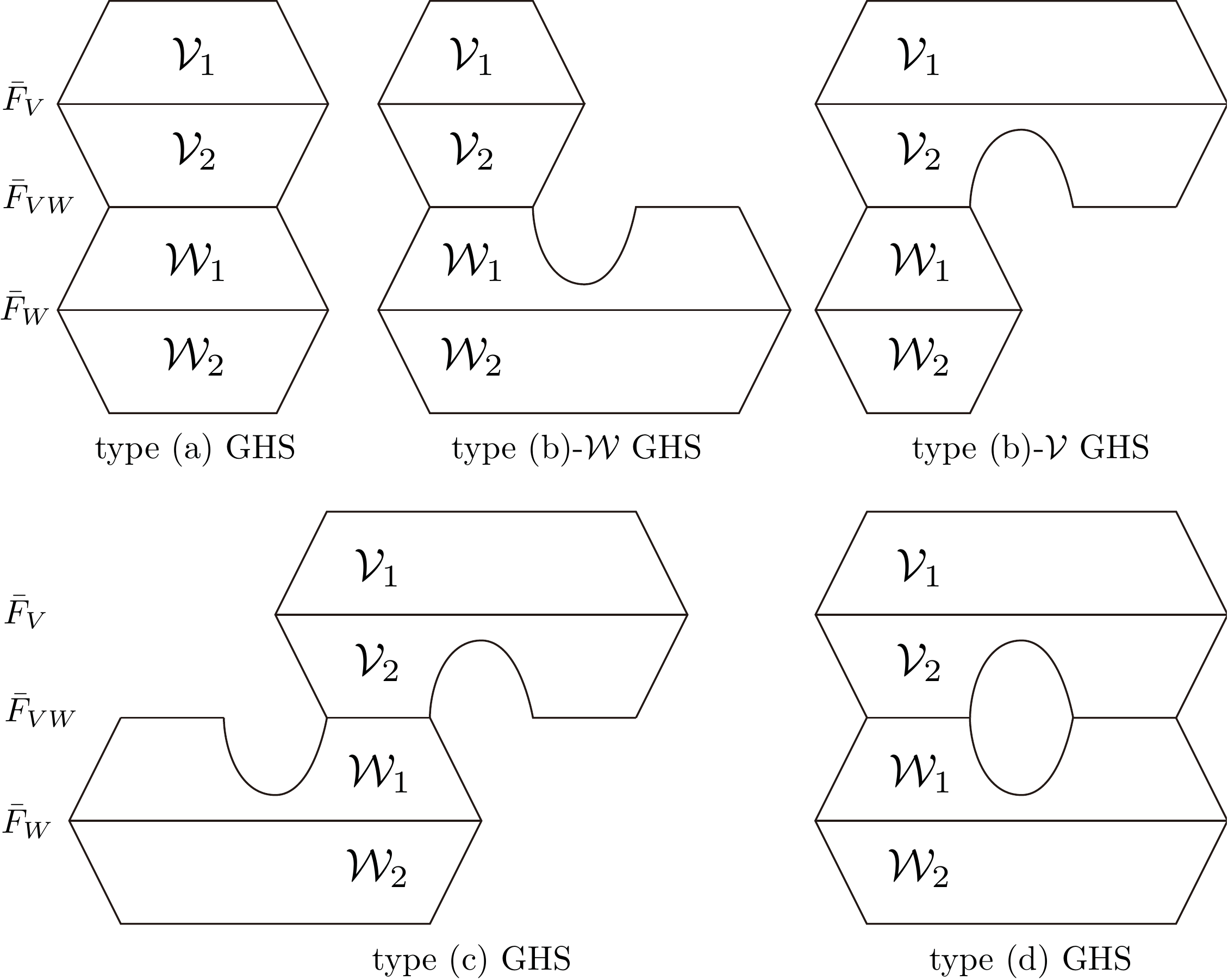}
\caption{the five types of generalized Heegaard splittings \label{fig-Heegaard-a}}
\end{figure}

From now on, we will use the notation $\{\bar{F}_V, \bar{F}_{VW}, \bar{F}_W\}$ as the generalized Heegaard splitting obtained by weak reduction from a weakly reducible, unstabilized Heegaard splitting $(\V,\W;F)$ of genus three along the weak reducing pair $(V,W)$.

Since every weak reducing pair in a $\V$- or $\W$-facial cluster $\varepsilon$ gives a unique generalized Heegaard splitting after weak reduction up to isotopy by Lemma \ref{lemma-2-17}, we can say $\varepsilon$ has a GHS of either type (a), type (b)-$\W$ or type (b)-$\V$  by Lemma \ref{lemma-four-GHSs} (we exclude the possibility that  $\varepsilon$ has a GHS of type (c) or type (d) by Lemma 3.7 of \cite{JungsooKim2014-2}).

In Definition \ref{definition-BB-type-a}, Definition \ref{definition-BB-type-b} and  Definition \ref{definition-BB-type-cd}, we will find a connected portion of $\DVW(F)$, say a ``\textit{building block}'' of $\DVW(F)$, such that every weak reducing pair in a building block gives the same generalized Heegaard splitting obtained by weak reduction up to isotopy.

\begin{definition}[Definition 3.3 of \cite{JungsooKim2014-2}]\label{definition-BB-type-a}
Assume $M$ and $F$ as in Lemma \ref{lemma-2-14}.
Let $\varepsilon_\V$ and  $\varepsilon_\W$ be a $\V$-facial cluster and a $\W$-facial cluster such that they share the common center $(\bar{V},\bar{W})$ (so $\bar{V}$ and $\bar{W}$ are non-separating in $\V$ and $\W$ respectively).
Let $\Sigma$ be the union of all simplices of $\DVW(F)$ spanned by the vertices of $\varepsilon_\V\cup\varepsilon_\W$.
Let $\Sigma_{V'W'}=\{V',\bar{V},\bar{W},W'\}$ be a $3$-simplex of $\DVW(F)$ containing $(\bar{V},\bar{W})$.
Then $\Sigma=\bigcup_{V',W'}\Sigma_{V'W'}$ for all possible $V'$ and $W'$ and therefore every weak reducing pair in $\Sigma$ gives the same generalized Heegaard splitting up to isotopy of type (a).
We call $\Sigma$ and $(\bar{V},\bar{W})$ a \textit{building block of $\DVW(F)$ having a type (a) GHS} and the \textit{center} of $\Sigma$ respectively.
\end{definition}

\begin{definition}[Definition 3.5 of \cite{JungsooKim2014-2}]\label{definition-BB-type-b}$\,$
\begin{enumerate}
\item \textit{A building block of $\DVW(F)$ having a type (b)-$\W$ GHS} is a $\W$-facial cluster having a type (b)-$\W$ GHS.
\item \textit{A building block of $\DVW(F)$ having a type (b)-$\V$ GHS} is a $\V$-facial cluster having a type (b)-$\V$ GHS.
\end{enumerate}
We define the \textit{center} of a building block of $\DVW(F)$ having a type (b)-$\W$ or (b)-$\V$ GHS as the center  of the corresponding $\W$- or $\V$-facial cluster.
\end{definition}

\begin{definition}\label{definition-BB-type-cd}
Assume $M$ and $F$ as in Lemma \ref{lemma-2-14} and let $(V,W)$ be a weak reducing pair.
Suppose that the generalized Heegaard splitting obtained by weak reduction along $(V,W)$ is a type (c) GHS (type (d) GHS resp.).
In this case, we call the weak reducing pair $(V,W)$ itself ``\textit{a building block of $\DVW(F)$ having a type (c) GHS} (\textit{type (d) GHS} resp.)''.
We define the \textit{center} of the building block $(\bar{V},\bar{W})$ as $(V,W)$ itself.
\end{definition}

Note that the embedding of the thick level contained in $\V$ or $\W$ does not vary in the relevant compression body up to isotopy if we do weak reduction along a weak reducing pair contained in a fixed building block by Lemma \ref{lemma-2-17}.

\begin{theorem}[Theorem 3.13 of \cite{JungsooKim2014-2}]\label{lemma-just-BB}
Assume $M$ and $F$ as in Lemma \ref{lemma-2-14}.
Then every component of $\DVW(F)$ is just a building block of $\DVW(F)$.
Hence, we can characterize the components of $\DVW(F)$ into five types. 
Moreover, there is a uniquely determined weak reducing pair in each component of $\DVW(F)$, i.e. the ``\textit{center}'' of the component.
\end{theorem}

By Theorem \ref{lemma-just-BB}, we can say that a component of $\DVW(F)$ has a GHS of either type (a), type (b)-$\W$, type (b)-$\V$, type (c) or type (d).
Moreover, we define the \textit{center} of a component of $\DVW(F)$ as the center of the corresponding building block of $\DVW(F)$.
We can refer to Figure \ref{figure1} for the shapes of the components of $\DVW(F)$.
\begin{figure}
\includegraphics[width=12cm]{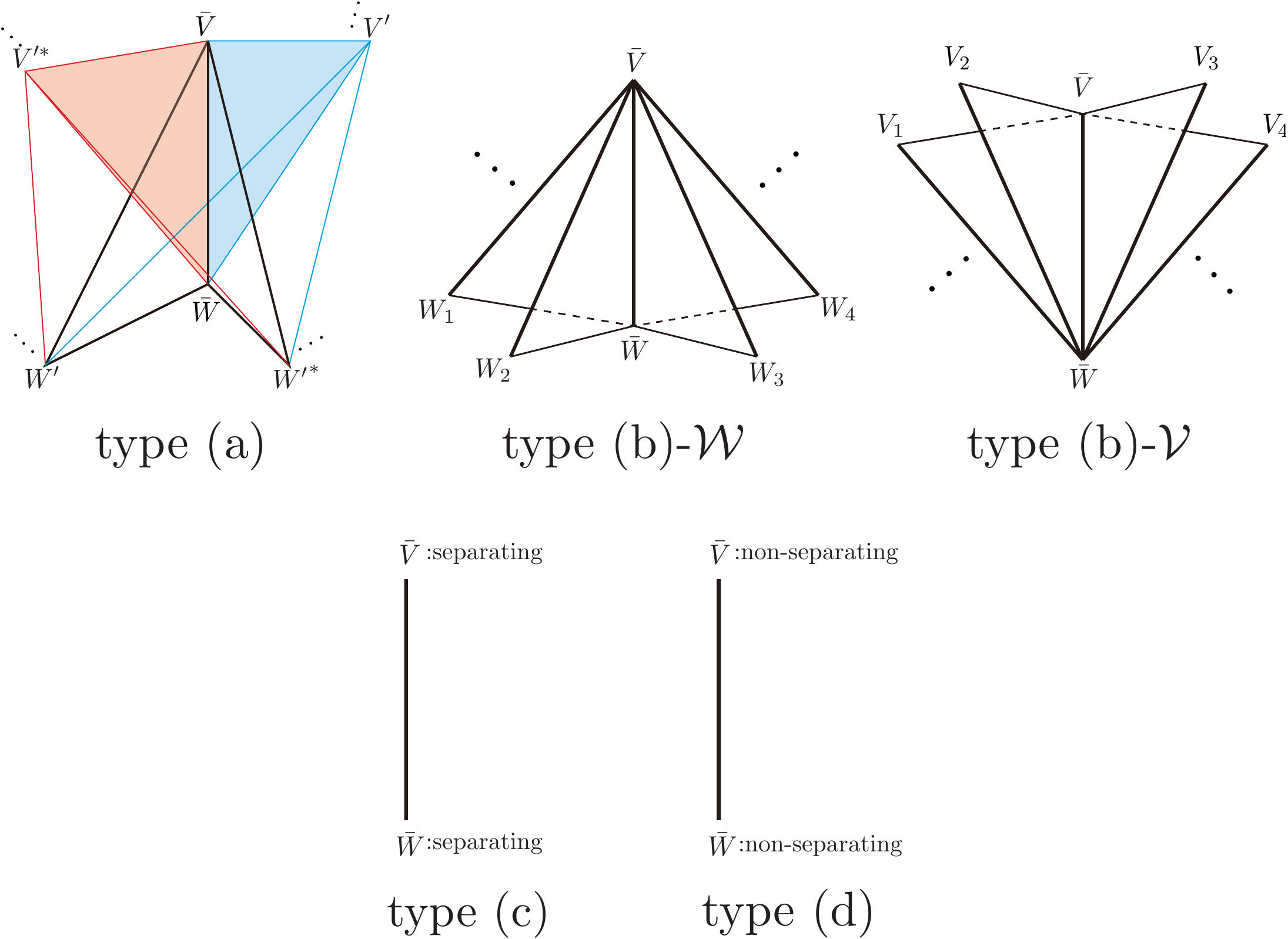}
\caption{the five types of components of $\DVW(F)$\label{figure1}}
\end{figure}

The next lemma determines all centers of components of $\DVW(F)$.

\begin{lemma}[Lemma 3.14 of \cite{JungsooKim2014-2}]\label{lemma-character-BB}
Assume $M$ and $F$ as in Lemma \ref{lemma-2-14}.
A weak reducing pair $(V,W)$ of $(\V,\W;F)$ is the center of a component of $\DVW(F)$ if and only if each of $V$ and $W$ does not cut off a solid torus from the relevant compression body.
Moreover, a compressing disk in a weak reducing pair belongs to the center of a component of $\DVW(F)$ if and only if it does not cut off a solid torus from the relevant compression body.
\end{lemma}

The next theorem means that different components of $\DVW(F)$ give different ways to embed the thick levels of the generalized Heegaard splittings obtained by weak reductions in the relevant compression bodies.

\begin{theorem}[Theorem 1.2 of \cite{JungsooKim2014-2}]\label{theorem-structure}
Let $(\V,\W;F)$ be a weakly reducible, unstabilized, genus three Heegaard splitting in an orientable, irreducible $3$-manifold $M$.
Then there is a function from the components of $\DVW(F)$ to the isotopy classes of the generalized Heegaard splittings obtained by weak reductions from $(\V,\W;F)$.
The number of components of the preimage of an isotopy class of this function is the number of ways to embed the thick level contained in $\V$ into $\V$ (or in $\W$ into $\W$).
This means that if we consider a generalized Heegaard splitting $\mathbf{H}$ obtained by weak reduction from $(\V,\W;F)$, then the way to embed the thick level of $\mathbf{H}$ contained in $\V$ into $\V$ determines the way to embed the thick level of $\mathbf{H}$ contained in $\W$ into $\W$ up to isotopy and vise versa.
\end{theorem}

Let $(\V_i,\W_i;F_i)$ be a weakly reducible, unstabilized Heegaard splitting of genus three in an irreducible $3$-manifold $M$ for $i=1,2$ and $f$ an orientation preserving automorphism of $M$ that takes $F_1$ into $F_2$.
Let $D$ be a compressing disk of $F_1$.
Then we can well-define the map sending the isotopy class $[D]\in \D(F_1)$ into $[f(D)]\in \D(F_2)$ and we can see that this gives a bijection between the set of vertices of $\D(F_1)$ and that of $\D(F_2)$, where we denote this map as $f:\D(F_1)\to \D(F_2)$ by using the same function name $f$ (we will denote this map as $f_\ast$ rigorously in Definition \ref{def-induce}).
The next lemma says that $f$ sends the center of a component of $\D_{\V_1 \W_1}(F_1)$ into the center of a component of $\D_{\V_2 \W_2}(F_2)$ (the proof is essentially the same as that of Lemma 3.1 of \cite{JungsooKim2015}).

\begin{lemma}\label{lemma-center-center}
Suppose that $M$ is an orientable, irreducible $3$-manifold and $(\V_i,\W_i;F_i)$ is a weakly reducible, unstabilized, genus three Heegaard splitting of $M$ for $i=1,2$.
Let $f$ be an orientation preserving automorphism of $M$ that takes $F_1$ into $F_2$.
Then $f$ sends the center of a component of $\D_{\V_1 \W_1}(F_1)$ into the center of a component of $\D_{\V_2 \W_2}(F_2)$.
\end{lemma}

\section{The proof of Theorem \ref{theorem-GHS-HS}\label{section3}}

In this section, we will prove Theorem \ref{theorem-GHS-HS}.

Suppose that there are two generalized Heegaard splittings $\mathbf{H}_1$ and $\mathbf{H}_2$ obtained by weak reductions from weakly reducible, unstabilized, genus three Heegaard splittings  $(\V_1,\W_1;F_1)$ and $(\V_2,\W_2;F_2)$ of an orientable, irreducible $3$-manifold $M$ respectively.
Assume that there is an orientation preserving automorphism $f$ of $M$ that takes $\mathbf{H}_1$ into $\mathbf{H}_2$, i.e. $f$ sends the thick levels of $\mathbf{H}_1$ into those of $\mathbf{H}_2$ and sends the inner thin level of $\mathbf{H}_1$ into that of $\mathbf{H}_2$.
In Theorem \ref{lemma-determine-GHSs}, we will prove that we can isotope $f$ so that (i) $f(F_1)=F_2$ and (ii) $f(\mathbf{H}_1)=\mathbf{H}_2$.

\begin{theorem}\label{lemma-determine-GHSs}
Let $(\V_i,\W_i;F_i)$ be a weakly reducible, unstabilized, genus three Heegaard splitting in an orientable, irreducible $3$-manifold $M$, $\mathcal{B}_i$ a component of $\DVWi(F_i)\subset \D(F_i)$, $(V_i,W_i)$ the center of $\mathcal{B}_i$, and $\mathbf{H}_i$ the generalized Heegaard splitting obtained by weak reduction along $(V_i,W_i)$ from $(\V_i,\W_i;F_i)$ for $i=1,2$.
If $f$ is an orientation preserving automorphism of $M$ sending $\mathbf{H}_1$ into $\mathbf{H}_2$, then there is an isotopy $f_t$ such that $f_0=f$, $f_1(F_1)=F_2$, and $f_t(\mathbf{H}_1)=\mathbf{H}_2$ for  $0\leq t \leq 1$.
\end{theorem}

\begin{proof}
Without loss of generality, assume that $f$ sends the thick level of $\mathbf{H}_1$ contained in $\V_1$ into the thick level of $\mathbf{H}_2$ contained in $\V_2$.

Let $(\V_1^i,\V_2^i;\bar{F_i}_{V_i})\cup_{\bar{F_i}_{V_i W_i}}(\W_1^i,\W_2^i;\bar{F_i}_{W_i})$ be the generalized Heegaard splitting $\mathbf{H}_i$, where $\partial_-\V_2^i\cap\partial_-\W_1^i=\bar{F_i}_{V_i W_i}$.
In this setting, $\V_2^i\cap \partial_- \V_i=\emptyset=\W_1^i\cap\partial_-\W_i$ by Lemma \ref{lemma-four-GHSs}.
By the assumption of $f$, we can see that $f(\V_i^1)=\V_i^2$ and $f(\W_i^1)=\W_i^2$ for $i=1,2$.

We will prove that we can isotope $f$ so that $f(\V_1)=\V_2$ where the isotopy preserves the thick levels and the inner thin level of $f(\mathbf{H}_1)$ during the isotopy.

If we consider the compressing disks $V_i$ and $W_i$ of $\V_i$ and $\W_i$, then they are naturally extended to the compressing disks $\tilde{V}_i$ and $\tilde{W}_i$ of $\W_1^i$ and $\V_2^i$ as follows.
If we consider Lemma \ref{lemma-2-8}, then $\partial V_i$ belongs to the genus two component of ${F_i}_{W_i}$, say ${F_i'}_{W_i}$, and $\partial V_i$ is an essential simple closed curve in ${F_i'}_{W_i}$ (see (a) of Figure \ref{fig-wrc}), where  ``the genus two component of ${F_i}_{W_i}$'' is the one used when we obtain the thick level $\partial_+\W_1^i$ as in the last statement of Lemma \ref{lemma-four-GHSs}.
\begin{figure}
\includegraphics[width=12cm]{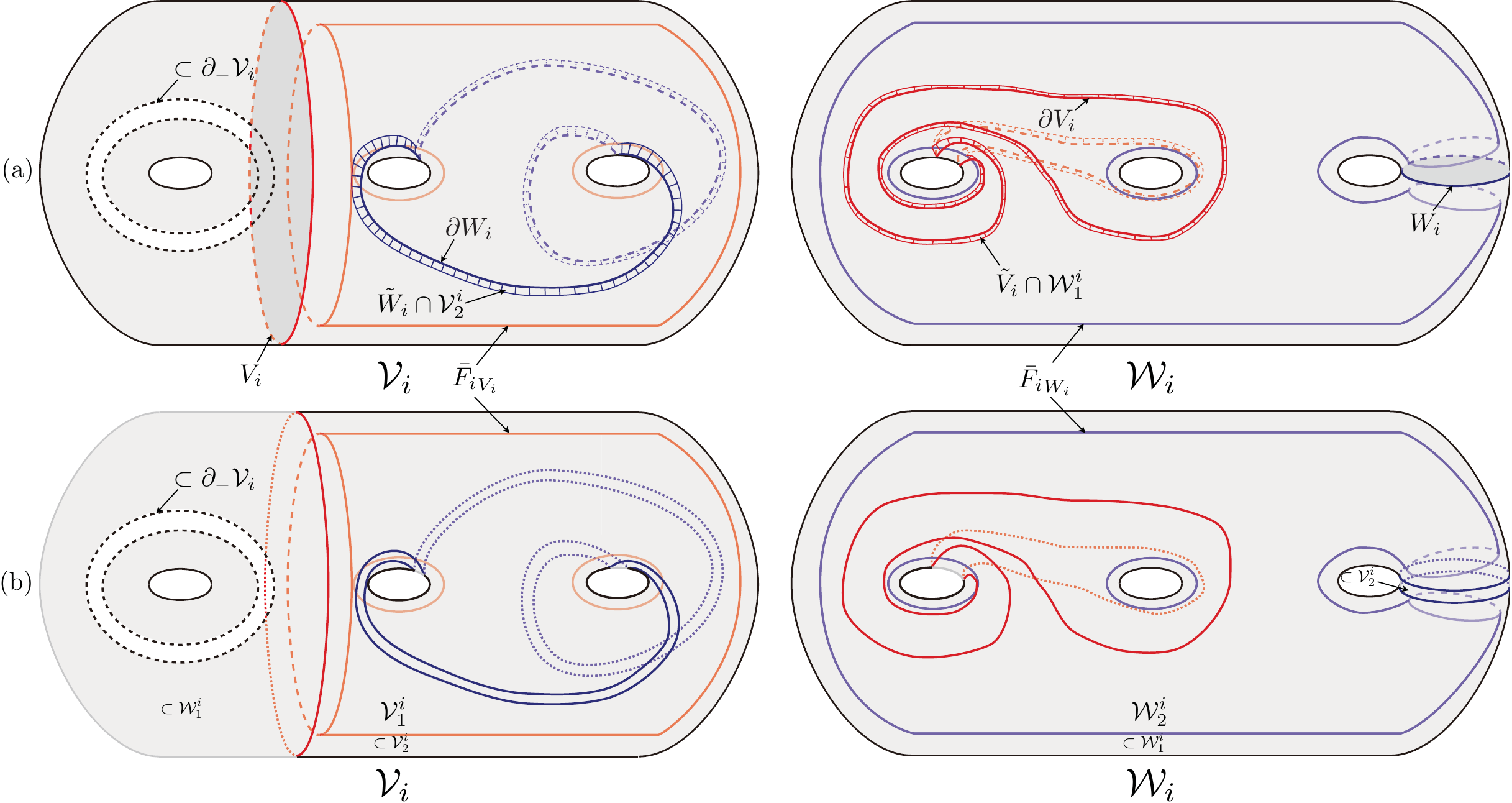}
\caption{(a) We can extend $V_i$ and $W_i$ into the compressing disks of $\W_1^i$ and $\V_2^i$ respectively. (b) the GHS \label{fig-wrc}}
\end{figure}
Here, the region between ${F_i'}_{W_i}$ and $\partial_+\W_1^i$ is homeomorphic to $\partial_+\W_1^i\times I$.
Let $A$ be a properly embedded  incompressible annulus in $\partial_+\W_1^i\times I$ such that $\partial V_i$ is a component of $\partial A$ and the other component of $\partial A$ belongs to $\partial_+\W_1^i$ (such an annulus can be obtained by projecting $\partial V_i$ into $\partial_+\W_1^i$ through a given product structure of $\partial_+\W_1^i\times I$).
Moreover, there is a unique properly embedded incompressible annulus in $\partial_+\W_1^i\times I$ such that it has $\partial V_i$ as a boundary component and the other boundary component belongs to $\partial_+\W_1^i$ up to isotopy constant on ${F_i'}_{W_i}$ (see Lemma 3.4 of \cite{Waldhausen1968}).
Hence, the other component of $\partial A$ other than $\partial V_i$ is uniquely determined  up to isotopy in $\partial_+\W_1^i$.
This means that if we define $\tilde{V}_i$  as $V_i\cup A$, then it becomes a compressing disk of $\W_1^i$ and $\tilde{V}_i$ is well-defined up to isotopy in $\W_1^i$ (if we see (b) of Figure \ref{fig-wrc} or more generally Figure 8, Figure 9, Figure 10 and Figure 11 of \cite{JungsooKim2014}, then we can see that $\tilde{V}_i$ is contained in $\W_1^i$).
The symmetric argument also holds for $\partial W_i$ by considering the product region between the genus two component of ${F_i}_{V_i}$ containing $\partial W_i$, say  ${F_i'}_{V_i}$, and $\partial_+\V_2^i$ and therefore we get the wanted compressing disk $\tilde{W}_i$ of $\V_2^i$ from $W_i$.

Since $(V_i,W_i)$ is the center of $\mathcal{B}_i$ for $i=1,2$, we get the following claim.\\

\ClaimNN{A}{$\V_2^i=(\partial_-\V_2^i\times I)\cup N(\tilde{W}_i)$ and $\W_1^i=(\partial_-\W_1^i\times I)\cup N(\tilde{V}_i)$, where
\begin{enumerate}
	\item $N(\tilde{W}_i)$ and $N(\tilde{V}_i)$ are $1$-handles attached to $(\partial_-\V_2^i\times I)$ and $(\partial_-\W_1^i\times I)$ to complete $\V_2^i$ and $\W_1^i$ respectively and
	\item they are product neighborhoods of $\tilde{W}_i$ and $\tilde{V}_i$ in $\V_2^i$ and $\W_1^i$ respectively.
\end{enumerate}}

\begin{proofN}{Claim A}
Recall that $\V_2^i$ is a genus two compression body with non-empty minus boundary, i.e. there is a unique non-separating disk of $\V_2^i$ if $\partial_-\V_2^i$ is connected or there is a unique compressing disk of $\V_2^i$ if $\partial_-\V_2^i$ is disconnected up to isotopy and the uniquely determined disk cuts off $\V_2^i$ into $\partial_-\V_2^i\times I$ as in the proof of Lemma \ref{lemma-defining}.
Hence, it is sufficient to show that $\tilde{W}_i$ is isotopic to such disk in $\V_2^i$.  

If we consider the case when $\partial_-\V_2^i$ is disconnected, then the proof is trivial by the uniqueness of compressing disk in $\V_2^i$.

Now suppose that $\partial_-\V_2^i$ is connected.
Then we can discard the cases when $\mathbf{H}_i$ is a type (c) or type (d) GHS by Lemma \ref{lemma-four-GHSs}.

If $\partial W_i$ is separating in $F_i$, then $W_i$ does not cut off a solid torus from $\W_i$ by the assumption that $(V_i,W_i)$ is the center of $\mathcal{B}_i$ and Lemma \ref{lemma-character-BB}, i.e. it cuts off $(\text{torus})\times I$ from $\W_i$.
But this means that the generalized Heegaard splitting obtained by weak reduction along $(V_i,W_i)$ is a type (b)-$\V_i$ GHS or type (c) GHS by Lemma \ref{lemma-four-GHSs}, violating  the assumption that $\partial_- \V_2^i$ is connected.
Hence, $\partial W_i$ must be non-separating in $F_i$.
Here, we can see that it is also non-separating in ${F_i'}_{V_i}$ because the case when $\partial W_i$ is non-separating in $F_i$ but it is separating in ${F_i'}_{V_i}$ appears only if $\mathbf{H}_i$ is of type (d) GHS by Lemma \ref{lemma-four-GHSs}.
Therefore the canonical projection of $\partial W_i\subset{F_i'}_{V_i}$ into $\partial_+\V_2^i$ in $\partial_+\V_2^i\times I$ is also non-separating in $\partial_+\V_2^i$, i.e. $\tilde{W}_i$ is a non-separating disk in $\V_2^i$.
The symmetric argument also holds for $\tilde{V}_i$ in $\W_1^i$.

This completes the proof of Claim A.
\end{proofN}

If we consider the assumption that $\tilde{V}_i$ and $\tilde{W}_i$ are naturally extended from $V_i$ and $W_i$ by attaching uniquely determined annuli to them, then we can assume that $N(\tilde{V}_i)\cap \V_i$ and $N(\tilde{W}_i)\cap \W_i$ are also product neighborhoods of $V_i$ and $W_i$ in $\V_i$ and $\W_i$ respectively, say $N(V_i)$ and $N(W_i)$, by choosing $N(\tilde{V}_i)$ and $N(\tilde{W}_i)$ suitably.
Hence, we can consider $N(\tilde{V}_i)$ as a big cylinder and $N(V_i)$ as a vertical small cylinder in the middle of $N(\tilde{V}_i)$ for $i=1,2$ with respect to a given $D^2\times I$ structure  of $N(\tilde{V}_i)$ and the symmetric argument also holds for $N(\tilde{W}_i)$ and $N(W_i)$ for $i=1,2$.
(From now on, we will use the term ``\textit{cylinder}'' to denote a $3$-manifold homeomorphic to  $D^2\times I$.)\\

\ClaimN{B1}{
We can isotope $f$ so that (i) $f(N(\tilde{V}_1))=N(\tilde{V}_2)$ and $f(N(\tilde{V}_1)\cap\partial_+\W_1^1)=N(\tilde{V}_2)\cap\partial_+\W_1^2$, (ii) $f(N(V_1))=N(V_2)$ and $f(N(V_1)\cap\partial_+\V_1)=N(V_2)\cap\partial_+\V_2$, and (iii) the assumption $f(\mathbf{H}_1)=\mathbf{H}_2$ holds at any time during the isotopy.
}

\begin{proofN}{Claim B1}
Since $\W_1^1=(\partial_-\W_1^1\times I)\cup N(\tilde{V}_1)$ by Claim A and $f$ is a homeomorphism, $\W_1^2=f(\W_1^1)=f(\partial_-\W_1^1\times I)\cup f(N(\tilde{V}_1))$, where $f(N(\tilde{V}_1))$ is a $1$-handle attached to $f(\partial_-\W_1^1\times I)=\partial_-\W_1^2\times I$, i.e. $f(\tilde{V}_1)$ is the cocore disk of the $1$-handle.
But Lemma \ref{lemma-defining} implies that there exists a unique such cocore disk in $\W_1^2$ up to isotopy and therefore $f(\tilde{V}_1)$ is isotopic  to $\tilde{V}_2$  in $\W_1^2$ by considering Claim A.
Hence, the existence of the isotopy of $f$ satisfying (i) is obvious (see the procedure from (a) to (b) of Figure \ref{fig-isotopy}).
After the previous isotopy, we can modify the location of the small cylinder $f(N(V_1))$ in the big cylinder $N(\tilde{V}_2)$ by an isotopy to satisfy (ii)  (see the procedure from (b) to (c) of Figure \ref{fig-isotopy}).
Since we can assume that $f(\mathbf{H}_1)=\mathbf{H}_2$ during these isotopies, (iii) holds.
This completes the proof of Claim B1.
\end{proofN}

\begin{figure}
\includegraphics[width=12cm]{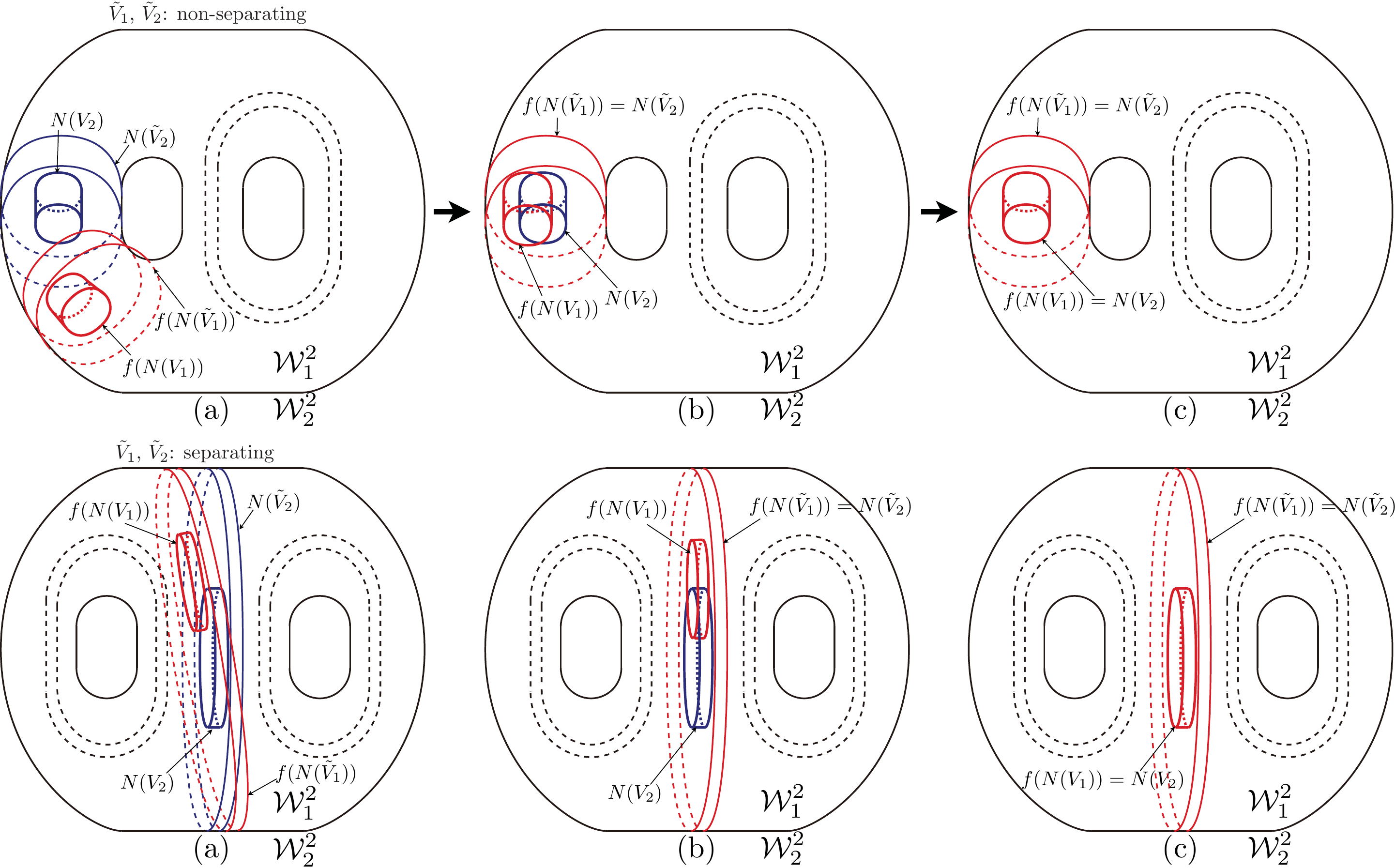}
\caption{$f(N(\tilde{V}_1))=N(\tilde{V}_2)$ and $f(N(V_1))=N(V_2)$ \label{fig-isotopy}}
\end{figure}

Note that the isotopy of Claim B1 affects not only the image $f(\W_1^1)$  but also $f(\W_2^1)$ near $\partial_+\W_2^2$ even though both $f(\W_1^1)$ and $f(\W_2^1)$ are preserved setwisely during the isotopy.
But we can assume that it does not affect the image of the inner thin level and therefore this isotopy does not affect $f(\V_2^1)$.

Hence, we get the following claim similarly.\\

\ClaimN{B2}{
Without changing the result of Claim B1, we can isotope $f$ so that (i) $f(N(\tilde{W}_1))=N(\tilde{W}_2)$ and $f(N(\tilde{W}_1)\cap\partial_+\V_2^1)=N(\tilde{W}_2)\cap\partial_+\V_2^2$, (ii) $f(N(W_1))=N(W_2)$ and $f(N(W_1)\cap\partial_+\W_1)=N(W_2)\cap\partial_+\W_2$, and (iii) the assumption $f(\mathbf{H}_1)=\mathbf{H}_2$ holds at any time during the isotopy.
}

The schematic figure describing this situation is Figure \ref{fig-homeo}.\\
\begin{figure}
\includegraphics[width=12cm]{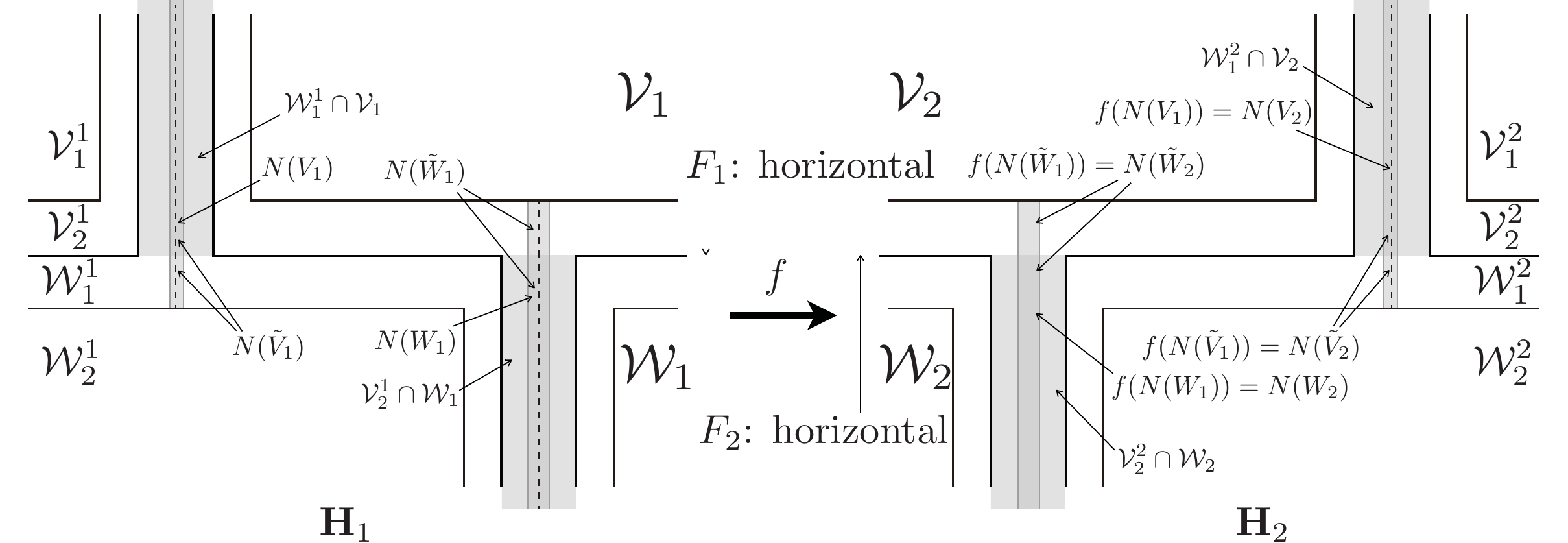}
\caption{after the isotopies of Claim B1 and Claim B2\label{fig-homeo}}
\end{figure}

Next, we can observe the follows, where this observation is the crucial idea of the proof of Theorem \ref{lemma-determine-GHSs}.
Recall that $(V_i, W_i)$ is the center of $\mathcal{B}_i$ for $i=1,2$, i.e. each of $V_i$ and $W_i$ is either non-separating or cuts off $(\text{torus})\times I$ from the relevant compression body by Lemma \ref{lemma-character-BB}.
(Note that we can refer to the top of Figure 8, the top of Figure 9, Figure 10 and Figure 11 in \cite{JungsooKim2014} for all possible cases.)\\
\begin{enumerate}
\item If $V_i$ ($W_i$ resp.) is non-separating in $\V_i$ ($\W_i$ resp.), then $\W_1^i\cap \V_i$ ($\V_2^i\cap \W_i$ resp.) is homeomorphic to $D^2\times I$ in $\W_1^i$ ($\V_2^i$ resp.) intersecting $F_i$ in $\partial D^2 \times I$ such that $D^2\times \{0,1\}$ belongs to the inner thin level $\bar{F_i}_{V_i W_i}$, where both disks are the scars of $V_i$ ($W_i$ resp.), and the other levels belong to the interior of $\W_1^i$ ($\V_2^i$ resp.) for $i=1,2$. 
(If we check (b) of Figure \ref{fig-wrc}, then we can see that $\V_2^i\cap \V_i$ consists of $\partial_+\V_2^i\times I$ and $\V_2^i\cap \W_i$ consists of a product neiborhood of $W_i$ in $\W_i$.)
Moreover, $f(\W_1^1\cap\V_1)$ ($f(\V_2^1\cap \W_1)$ resp.) is also a $D^2\times I$ in $\W_1^2$ ($\V_2^2$ resp.) such that the top and bottom levels belong to the inner thin level $\bar{F_2}_{V_2 W_2}$ and the other levels belong to the interior of $\W_1^2$ ($\V_2^2$ resp.) by the assumption that $f(\W_1^1)=\W_1^2$ ($f(\V_2^1)=\V_2^2$ resp.) and $f$ is a homeomorphism.
\label{papa}\\

\item If $V_i$ ($W_i$ resp.) cuts off $(\text{torus})\times I$ from $\V_i$ ($\W_i$ resp.), then $\W_1^i\cap \V_i$ ($\V_2^i\cap \W_i$ resp.) is homeomorphic to $(\text{torus})\times I$ in $\W_1^i$ ($\V_2^i$ resp.) intersecting $F_i$ in a once-punctured torus such that the top level belongs to $\partial_- \V_i$ ($\partial_-\W_i$ resp.), the bottom level intersects the inner thin level $\bar{F_i}_{V_i W_i}$ in a disk, where this disk is a scar of $V_i$ ($W_i$ resp.), and the other levels belong to the interior of $\W_1^i$ ($\V_2^i$ resp.) for $i=1,2$. 
(If we check (b) of Figure \ref{fig-wrc}, then we can see that $\W_1^i\cap \W_i$ consists of $\partial_+\W_1^i\times I$ and $\W_1^i\cap \V_i$ consists of $(\text{torus})\times I$ in $\V_i$.)
Moreover, $f(\W_1^1\cap\V_1)$ ($f(\V_2^1\cap \W_1)$ resp.) is also a $(\text{torus})\times I$ in $\W_1^2$ ($\V_2^2$ resp.) such that the top level belongs to $\partial_-\V_2$ ($\partial_-\W_2$ resp.), the bottom level intersects the inner thin level $\bar{F_2}_{V_2 W_2}$ in a disk, and the other levels belong to the interior of $\W_1^2$ ($\V_2^2$ resp.)  by the assumption that $f(\W_1^1)=\W_1^2$ ($f(\V_2^1)=\V_2^2$ resp.) and $f$ is a homeomorphism.
\label{papb}\\
\end{enumerate}
Now we will visualize each case.
For the sake of convenience, we will only consider the disk $\tilde{V}_i$ in $\W_1^i$  and use the symmetric arguments for $\tilde{W}_i$ in $\V_2^i$.
We will denote $\operatorname{cl}(\W_1^i-N(\tilde{V}_i))$ as $\partial_-\W_1^i\times I$ by Claim A.
Let $\partial_-\W_1^i\times \{0\}$ be $\partial_-\W_1^i$ itself and $\partial_-\W_1^i\times \{1\}$ be the union of the other components of $\partial \operatorname{cl}(\W_1^i-N(\tilde{V}_i))$ for $i=1,2$.
\begin{enumerate}
\item Case: $V_i$ is non-separating in $\V_i$.\label{aaaaa}

Here, we can see that if we drill a hole in $\W_1^i$ through the cylinder $\W_1^i\cap \V_i$ and take the closure of the resulting one, then $\bar{\W}_1^i=\operatorname{cl}(\W_1^i-(\W_1^i\cap \V_i))=\W_1^i\cap\W_i$ is homeomorphic to $\partial_+\W_1^i\times I$, i.e. the core arc of the cylinder is a spine of $\W_1^i$ by Definition \ref{def-spine}, say $\alpha_i$, because we can consider the closure of $\eta(\alpha_i)$ in $\W_1^i$ as $\W_1^i\cap \V_i$ itself.
In particular, $\operatorname{cl}((\W_1^i\cap \V_i)-N(V_i))$ consists of two components $C_1^i$ and $C_2^i$ such that (i) $\W_1^i\cap \V_i= C_1^i\cup N(V_i)\cup C_2^i$ and (ii) each $C_j^i$ is a cylinder whose top and bottom levels belong to $\partial_-\W_1^i\times \{0,1\}$ and the other levels belong to $\partial_-\W_1^i\times (0,1)$ for $i=1,2$.
Hence, we divide $\alpha_i$ into the three parts, the core arc of $N(V_i)$ and the core arcs of $C_1^i$ and $C_2^i$ which are the extended parts from the core arc of $N(V_i)$ down to $\partial_-\W_1^i$ to complete $\alpha_i$.
If we consider $\W_1^i-\eta(\alpha_i)=\partial_+\W_1^i\times I$, then the incompressible annulus $\tilde{V}_i-\eta(\alpha_i)$ is isotopic to vertical one by Lemma 3.4 of \cite{Waldhausen1968} by an isotopy constant on $\partial_+\W_1^i$.
In other words, we can deform the product structure of $\partial_+\W_1^i\times I$ so that $\tilde{V}_i-\eta(\alpha_i)$ would be vertical and we can assume that $N(\tilde{V}_i)-\eta(\alpha_i)$ is also vertical.
Hence, if we cut $\partial_+\W_1^i\times I$ along $N(\tilde{V}_i)-\eta(\alpha_i)$, take the closure of the resulting one, i.e. $\operatorname{cl}(\partial_+\W_1^i-N(\tilde{V}_i))\times I=(\partial_-\W_1^i \times I)-\eta(\alpha_i)$,  and move the annuli $\operatorname{cl}(\eta(\alpha_i))\cap ((\partial_-\W_1^i \times I)-\eta(\alpha_i))$ into vertical ones, then we get new product structure such that  $(\partial_-\W_1^i \times I)-\eta(\alpha_i)$ is homeomorphic to $((\partial_-\W_1^i\times\{1\})-\eta(\alpha_i))\times I$, where the $1$-level is $(\partial_-\W_1^i\times\{1\})-\eta(\alpha_i)$ itself.
This suggests a product structure of $\partial_-\W_1^i \times I$ such that $\alpha_i\cap (\partial_-\W_1^i \times I)$ is vertical.
Hence, we will represent $\alpha_i\cap (\partial_-\W_1^i \times I)$ and $C_1^i\cup C_2^i$ as vertical ones for the sake of convenience.\\
\begin{figure}
\includegraphics[width=12cm]{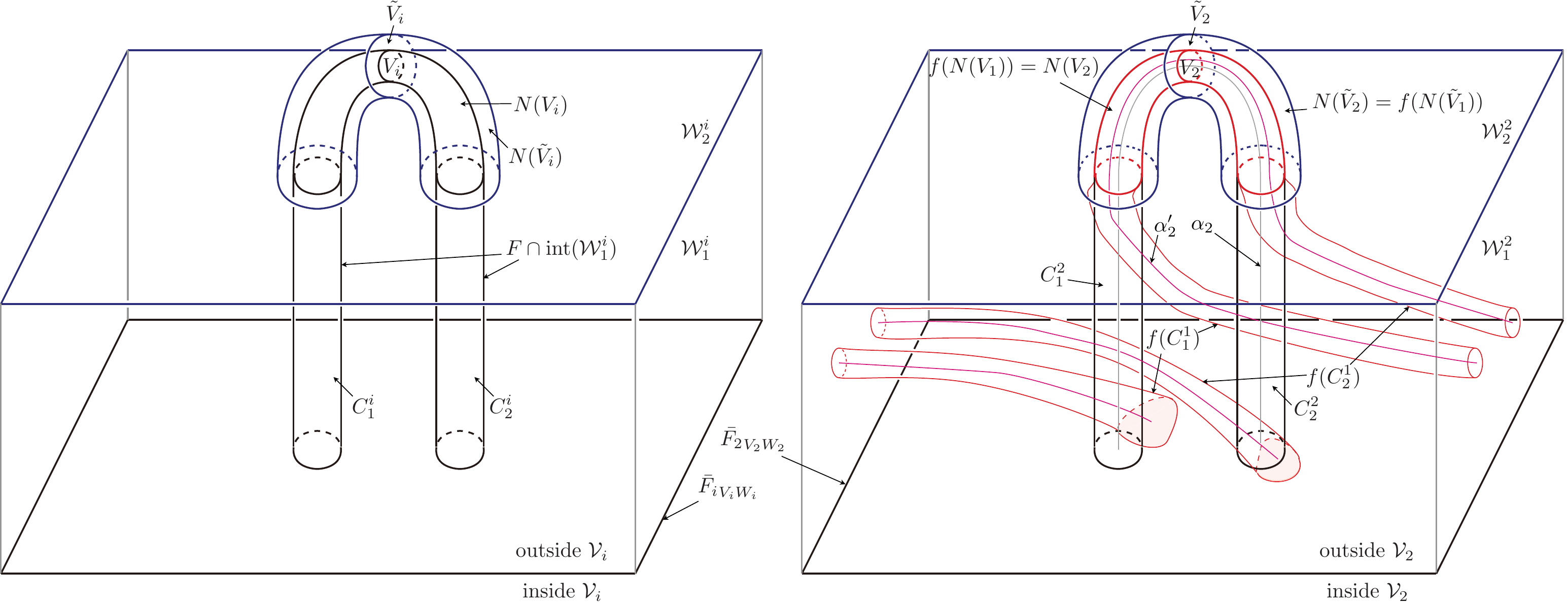}
\caption{the case when $V_i$ is non-separating in $\V_i$. \label{fig-sit-1}}
\end{figure}
Since $f$ is a homeomorphism, $f(\bar{\W}_1^1)$ is homeomorphic to $\partial_+\W_1^1\times I$ and the assumption that $f(\W_1^1)=\W_1^2$ means that it is homeomorphic to $\partial_+\W_1^2\times I$.
Moreover, we can see that $f(\bar{\W}_1^1)=f(\operatorname{cl}(\W_1^1-(\W_1^1\cap \V_1))=\operatorname{cl}(\W_1^2-f(\W_1^1\cap \V_1))$ where $f(\W_1^1\cap \V_1)$ is homeomorphic to $D^2\times I$ such that $D^2\times\{0,1\}$ belongs to $\partial_- \W_1^2$ and $D^2\times (0,1)$ belongs to $\operatorname{int}(\W_1^2)$ as in the previous observation.
Hence, Definition \ref{def-spine} implies that  the core arc of the cylinder $f(\W_1^1\cap \V_1)=f(C_1^1)\cup f(N(V_1))\cup f(C_2^1)$ is also a spine of $\W_1^2$, say $\alpha_2'$.
Here, we can assume that $\alpha_2'\cap f(N(V_1))$ is a parallel copy of $\alpha_2\cap N(V_2)$ by the assumption $f(N(V_1))=N(V_2)$.

Note that the inner thin level $\bar{F_i}_{V_i W_i}$ is either connected or disconnected even though $V_i$ is non-separating in $\V_i$, i.e. it consists of a torus or two tori (see Figure \ref{fig-sit-1} and Figure \ref{fig-sit-3} respectively).\\
\begin{figure}
\includegraphics[width=7cm]{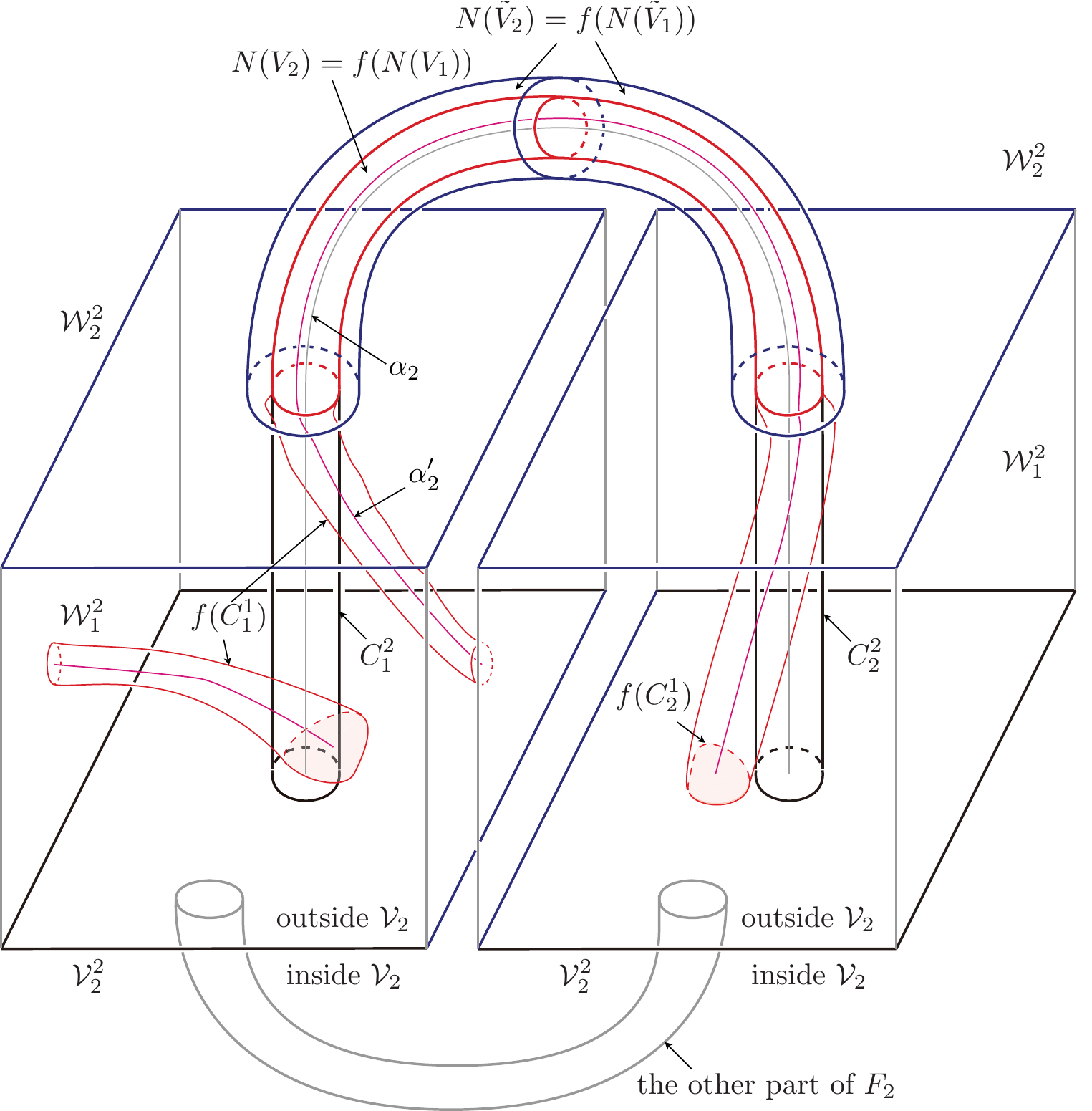}
\caption{another case when $V_i$ is non-separating in $\V_i$. \label{fig-sit-3}}
\end{figure}

\item Case: $V_i$ cuts off $(\text{torus})\times I$ from $\V_i$.

Let $T_i$ be the torus $\partial_-\V_i\cap \partial_- \W_1^i$.
Then we can take $N(T_i)=T_i\times I$ such that (i) $T_i\times \{0\}=T_i$, (ii) $T_i\times \{1\}$ intersects $\partial_+\V_i$ in a once-punctured torus, and (iii) $ \W_1^i\cap \V_i=N(T_i)\cup \bar{N}(V_i)$ where $\bar{N}(V_i)$ is a product neighborhood of $V_i$ in $\V_i$ containing $N(V_i)$ in the middle and $N(T_i)\cap\bar{N}(V_i)$ is a compressing disk in $\V_i$ isotopic to $V_i$ in $\V_i$.
Let us consider the genus two compression body $\tilde{\W}_1^i=\operatorname{cl}(\W_1^i-N(T_i))$ which is a deformation-retraction of $\W_1^i$.
Then $\bar{N}(V_i)$ is a cylinder connecting the two components of $\partial_-\tilde{\W}_1^i$.
Here, we take the product structure of $\partial_-\W_1^i\times I$ so that $T_i\times \{1\}\subset N(T_i)$ would be horizontal.
If we drill a hole in $\tilde{\W}_1^i$ through the cylinder $\bar{N}(V_i)$  and take the closure of the resulting one, say $\bar{\W}_1^i$, then it is equal to $\operatorname{cl}(\W_1^i - (\W_1^i\cap \V_i))$, i.e. $\bar{\W}_1^i$ is homeomorphic to $\partial_+\W_1^i\times I$ (see (b) of Figure \ref{fig-wrc}) and therefore also homeomorphic to $\partial_+\tilde{\W}_1^i\times I$.
This means that the core arc of $\bar{N}(V_i)$ is a spine of $\tilde{\W}_1^i$ by Defnition \ref{def-spine}, say $\alpha_i$.
Let the two $(\text{torus})\times I$ components of $\operatorname{cl}(\tilde{\W}_1^i-N(\tilde{V}_i))$ be $\mathcal{X}_1^i$ and $\mathcal{X}_2^i$, where $\mathcal{X}_1^i\cap\partial_-\V_2^i=\emptyset$ and $\mathcal{X}_2^i\cap\partial_-\V_2^i\neq\emptyset$.
Then $\operatorname{cl}(\bar{N}(V_i)-N(V_i))$ consists of two components $C_1^i$ and $C_2^i$ such that each $C_j^i$ is a cylinder whose top and bottom levels belong to $\partial \mathcal{X}_j^i$ for $j=1,2$ (see the left of Figure \ref{fig-sit-2}).
Here, we can draw $\alpha_i\cap \mathcal{X}_j^i$ and $C_j^i$ as vertical ones in $\mathcal{X}_j^i$ for $j=1,2$ by the assumption that $\alpha_i$ is a spine of $\tilde{\W}_1^i$ similarly as the previous case.

Let us consider the image  $f(\tilde{\W}_1^1)$.
Then it is equal to $\operatorname{cl}(\W_1^2-f(N(T_1)))$ since $f(\W_1^1)=\W_1^2$.
Moreover, we can see that $f(N(T_1))$ is homeomorphic to $T_2\times I$ because $f(T_1)=T_2$.
Hence, we can isotope $f$ so that $f(N(T_1))=N(T_2)$, i.e. we get $f(\tilde{\W}_1^1)=\tilde{\W}_1^2$.
Since $f$ is a homeomorphism, $f(\bar{\W}_1^1)=\operatorname{cl}(f(\tilde{\W}_1^1)-f(\bar{N}(V_1)))=\operatorname{cl}(\tilde{\W}_1^2-f(\bar{N}(V_1)))$ is homeomorphic to $\partial_+\tilde{\W}_1^1\times I$.
Therefore,  it is homeomorphic to $\partial_+\tilde{\W}_1^2\times I$ such that $\partial_+\tilde{\W}_1^2\times \{1\}$ is $\partial_+\tilde{\W}_1^2$ itself, i.e. the core arc of the cylinder $f(\bar{N}(V_1))=f(C_1^1)\cup f(N(V_1))\cup f(C_2^1)$ is also a spine of $\tilde{\W}_1^2$ by Definition \ref{def-spine}, say $\alpha'_i$.
Moreover, we can assume that $\alpha_2'\cap f(N(V_1))$ is a parallel copy of $\alpha_2\cap N(V_2)$. 
See the right of Figure \ref{fig-sit-2}.
\begin{figure}
\includegraphics[width=12cm]{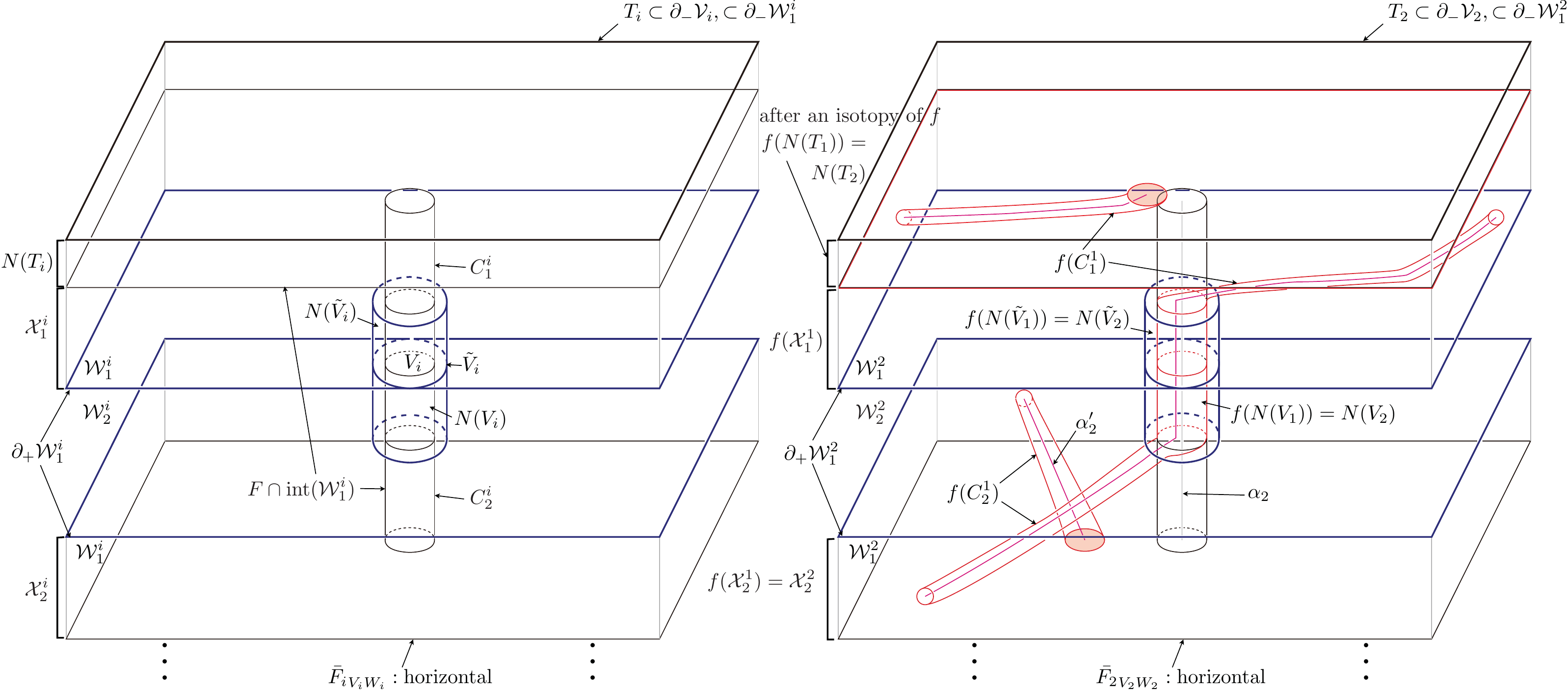}
\caption{the case when $V_i$ cuts off $(\text{torus})\times I$ from $\V_i$. \label{fig-sit-2}}
\end{figure}
\end{enumerate}

In any case, (i) we have represented $\cup_{j=1}^{2} C_j^i$ for $i=1,2$ as vertical cylinders in the relevant product structure  and (ii) we can say that the difference between $\alpha_2$ and $\alpha_2'$ comes from the two subarcs $\alpha_2'\cap f(C_1^1)$ and $\alpha_2'\cap f(C_2^1)$.\\

\ClaimN{C}{
We can isotope $f$ so that $\alpha_2'$ is monotone in the relevant product structure of $\partial_-\W_1^2\times I$ or $\partial_-\tilde{\W}_1^2\times I$.}

\begin{proofN}{Claim C}
We will prove that we can isotope $f$ so that $f(\cup_{j=1}^2 C_j^1)$ intersects each level surface of $\partial_-\W_1^2\times I$ in two disks (i.e. the core arcs of $f(\cup_{j=1}^2 C_j^1)$ are monotone) when $V_i$ is non-separating in $\V_i$ and $\bar{F}_{V_i W_i}$ is connected and the other cases are left as exercise.

Since $f(N(\tilde{V}_1))=N(\tilde{V}_2)$ and $f(\W_1^1)=\W_1^2$, $f(\partial_-\W_1^1\times I)=\partial_-\W_1^2\times I$, where $f(\partial_-\W_1^1\times \{1\})=\partial_-\W_1^2\times \{1\}$. 
Let $\tilde{f}=f|_{\partial_-\W_1^1\times I}$.
Suppose that there is an ambient isotopy $h_t$ defined on $\partial_-\W_1^2 \times I$ such that 
\begin{enumerate}
\item $h_t$ is the identity on $\partial_-\W_1^2\times \{0,1\}$ for $0\leq t \leq 1$ and
\item $h_1(\tilde{f}(\cup_{j=1}^2 C_j^1))$ intersects each level surface of $\partial_-\W_1^2\times I$ in two disks, 
\end{enumerate}
then we can extend it to the ambient isotopy $h'_t$ defined on $M$ such that $h'_t|_{\partial_-\W_1^2\times I}=h_t$ and $h'_t$ is the identity on $M-(\partial_-\W_1^2\times I)$.
Therefore, the argument in Definition \ref{def-isotopy} induces that $f$ can be isotoped so that $f(\cup_{j=1}^2 C_j^1)$ intersects each level surface of $\partial_-\W_1^2\times I$ in two disks.
Hence, it is sufficient to show the existence of such $h_t$.

Let us define a homeomorphism $g:\partial_-\W_1^2 \times I\to\partial_-\W_1^1 \times I$ such that
\begin{enumerate}
\item $g(x,s)=(\bar{g}(x),s)$ for $x\in\partial_-\W_1^2$, $s\in I$ and a homeomorphism $\bar{g}:\partial_-\W_1^2 \to \partial_-\W_1^1$,
\item where $\bar{g}$ satisfies $(\bar{g}(x),1)=\tilde{f}^{-1}(x,1)$ for $x\in\partial_-\W_1^2$ and $1\in I$.
\end{enumerate}
Then we can see that $g(\cup_{j=1}^2 C_j^2)=\cup_{j=1}^2 C_j^1$ because (i) $g(\{x\}\times I)=\{\bar{g}(x)\}\times I$, (ii) each $\cup_{j=1}^2 C_j^i$ is vertical in $\partial_-\W_1^i\times I$ for $i=1,2$, and (iii) $\tilde{f}((\cup_{j=1} C_j^1)\cap (\partial_-\W_1^1\times\{1\})=(\cup_{j=1} C_j^2)\cap (\partial_-\W_1^2\times\{1\})$ by the assumption $f(N(V_1))=N(V_2)$.
Hence, if we consider the composition $h=\tilde{f} \circ g$, then it is an automorphism of $\partial_-\W_1^2 \times I$ such that $h_{\partial_-\W_1^2\times\{1\}}=\operatorname{id}$.
Therefore Lemma 3.5 of \cite{Waldhausen1968} induces that there is an isotopy $h_t^\ast$ defined on $\partial_-\W_1^2\times I$, constant on $\partial_-\W_1^2\times \{0,1\}$, such that $h_0^\ast=h$ and $h_1^\ast$ is a level-preserving homeomorphism.
Since $\cup_{j=1} C_j^2$ intersects each level surface of $\partial_-\W_1^2\times I$ in two disks, so does $h_1^\ast(\cup_{j=1} C_j^2)$ because  $h_1^\ast$ is level-preserving. 

Hence, if we take the isotopy $h_t=h_t^\ast\circ h^{-1}$, then (i) $h_0=\operatorname{id}$ and (ii) $h_1(\tilde{f}(\cup_{j=1}C_j^1))=h_1^\ast(\cup_{j=1} C_j^2)$ intersects each level surface of $\partial_-\W_1^2\times I$ in two disks.
If we consider the assumption that $h_t^\ast$ is constant on $\partial_-\W_1^2\times \{0,1\}$ and $h_0=\operatorname{id}$, then we can see that $h_t$ is the identity on $\partial_-\W_1^2\times \{0,1\}$ for $0\leq t \leq 1$.

This completes the proof of Claim C.
\end{proofN}

We can use the symmetric arguments for $\V_2^i\cap \W_i$ for $i=1,2$ and $f(\V_2^1\cap \W_1)$ to visualize them and therefore we get the spine $\beta_2$ of $\V_2^2$ or $\tilde{\V}_2^2$ corresponding to $D_1^2\cup N(W_2)\cup D_2^2$ and the spine $\beta_2'$ of $\V_2^2=f(\V_2^1)$ or $\tilde{\V}_2^2=f(\tilde{\V}_2^1)$ corresponding to  $f(D_1^1)\cup f(N(W_1))\cup f(D_2^1)$ respectively, where the cylinder $D_j^i$ is obtained similarly as $C_j^i$ for $1\leq i,j\leq 2$.
Moreover, we can assume that $\beta_2'$ is monotone in the relevant product structure of $\partial_-\V_2^2\times I$ or $\partial_-\tilde{\V}_2^2\times I$ where $\beta_2$ is vertical similarly.

Since both $\alpha_2$ and $\alpha_2'$ are spines dual to the same minimal defining set $\{\tilde{V}_2\}$ of $\W_1^2$ or $\tilde{\W}_1^2$,  we can expect that $f(\W_1^1\cap\V_1)$ would be isotopic to $\W_1^2\cap \V_2$.
Moreover, we can expect that $f(\V_2^1\cap \W_1)$ would be isotopic to $\V_2^2\cap \W_2$ similarly.
But we cannot guarantee that the isotopy sending $f(\W_1^1\cap\V_1)$ into $\W_1^2\cap \V_2$ might not affect that sending $f(\V_2^1\cap \W_1)$ into $\V_2^2\cap \W_2$ because they share the common inner thin level.
Hence, we will describe the details of these ``\textit{untying isotopies}'' and find the way how to avoid possible interferences in the proof of Lemma \ref{lemma-spines}.

\begin{lemma}\label{lemma-spines}
We can assume that $f(\W_1^1\cap\V_1)=\W_1^2\cap \V_2$ and $f(\V_2^1\cap \W_1)=\V_2^2\cap \W_2$ simultaneously after a sequence of isotopies of $f$ satisfying $f(\mathbf{H}_1)=\mathbf{H}_2$ at any time of the isotopies.
\end{lemma}

\begin{proof}
From now on, we will describe the ``\textit{untying isotopies}'' of $f(\W_1^1\cap\V_1)$ and $f(\V_2^1\cap \W_1)$ rigorously such that the untying isotopies of $f(\W_1^1\cap\V_1)$ do not affect those of $f(\V_2^1\cap \W_1)$.\\

\ClaimN{D}{
(i) We can take the relevant product structures of $\bar{F_2}_{V_2 W_2}\times Is$ in $\W_1^2$ and $\V_2^2$ such that the projection images of the top and bottom levels of $N(\tilde{V}_2)$ and $N(\tilde{W}_2)$ as $D^2\times I$ into $\bar{F_2}_{V_2 W_2}$ do not intersect each other and (ii) we can isotope $f$ so that $f(\W_1^1\cap\V_1)\cap (\V_2^2\cap \W_2) = \emptyset$ and $f(\V_2^1\cap \W_1)\cap(\W_1^2\cap \V_2)=\emptyset$ satisfying $f(\mathbf{H}_1)=\mathbf{H}_2$ at any time of the isotopies.
}

\begin{proofN}{Claim D}
First, we will find two disks (if ${\bar{F_2}}_{V_2 W_2}$ is connected) or two sets of two disks (if ${\bar{F_2}}_{V_2 W_2}$ is disconnected), say $D_{\W_1^2}$ and $D_{\V_2^2}$, such that $D_{\W_1^2}$ contains the projection image of $N(\tilde{V}_2)\cap(\partial_-\W_1^2\times \{1\})$ into $\partial_-\W_1^2\times \{0\}$, $D_{\V_2^2}$ contains the projection image of $N(\tilde{W}_1)\cap(\partial_-\V_2^2\times \{1\})$ into $\partial_-\V_2^2\times \{0\}$, and $D_{\W_1^2}\cap D_{\V_2^2}=\emptyset$ for the case $\partial_-\W_1^2\cap\partial_-\V_2=\emptyset$ and $\partial_-\V_2^2\cap\partial_-\W_1=\emptyset$, i.e. $\mathbf{H}_2$ is a GHS of type (a) or type (d).
For the other cases, we only consider the relevant product structures in $\W_1^2$ and $\V_2^2$ intersecting the inner thin level ${\bar{F_2}}_{V_2 W_2}$ (for example, $\mathcal{X}_2^2$ in $\W_1^2$) and the details are left as exercise.

If we consider the assumptions that  $\W_1^i\cap\V_2^i$ is the inner thin level $\bar{F_i}_{V_i W_i}$ and the observation that each of $\W_1^i\cap \V_i$ and $\V_2^i\cap \W_i$ intersects $\bar{F_i}_{V_i W_i}$ in the scars of $V_i$ or $W_i$ respectively, then we can see that $(\W_1^i\cap \V_i)\cap (\V_2^i\cap \W_i)=\emptyset$ for $i=1,2$.
Recall that $\W_1^2\cap \V_2$ and $\V_2^2\cap \W_2$ are vertical cylinders in the relevant product structures intersecting $\bar{F_2}_{V_2 W_2}$.
If we change the product structures of $\partial_-\W_1^2\times I$ and $\partial_-\V_2^2\times I$ near $\partial_-\W_1^2\times\{1\}$ and $\partial_-\V_2^2\times\{1\}$ respectively, then we can assume that the projection image of $N(\tilde{V}_2)\cap(\partial_-\W_1^2\times\{1\})$ into $\partial_-\W_1^2\times\{0\}=\bar{F_2}_{V_2 W_2}$, say $p_{\tilde{V}_2}$, misses that of $N(\tilde{W}_2)\cap (\partial_-\V_2^2\times\{1\})$ into $\partial_-\V_2^2\times\{0\}=\bar{F_2}_{V_2 W_2}$, say $p_{\tilde{W}_2}$, without changing the assumption that $\W_1^2\cap \V_2$ and $\V_2^2\cap \W_2$ are vertical in the relevant product structures (see Figure \ref{fig-perturbation}).
Moreover, we can assume that this perturbation does not affect the assumption that $\alpha_2'$ and $\beta_2'$ are monotone in the relevant product structures in $\W_1^2$ and $\V_2^2$ respectively if we deform the product structures sufficiently near $\partial_-\W_1^2\times\{1\}$ and $\partial_-\V_2^2\times\{1\}$.  \begin{figure}
\includegraphics[width=10cm]{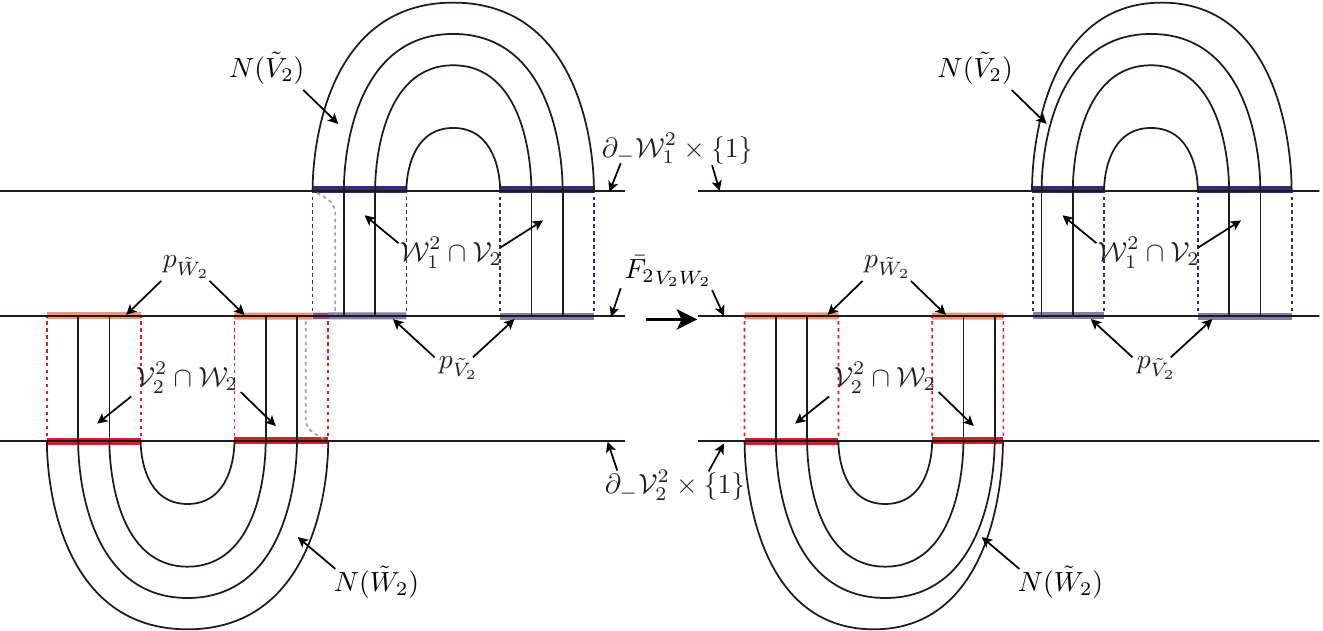}
\caption{the perturbations of two product structures
\label{fig-perturbation}}
\end{figure}
Also we can assume that there is a small neighborhood $p'_{\tilde{V}_2}$ and $p'_{\tilde{W}_2}$ of $p_{\tilde{V}_2}$ and $p_{\tilde{W}_2}$ in $\bar{F_2}_{V_2 W_2}$ respectively such that $p'_{\tilde{V}_2}\cap p'_{\tilde{W}_2}=\emptyset$.
Here, we can see that each of $p_{\tilde{V}_2}$ and $p_{\tilde{W}_2}$ consists of two components.\\

\begin{enumerate}
\item Case: $\bar{F_2}_{V_2 W_2}$ consists of a torus, i.e. $\mathbf{H}_2$ is a type (a) GHS.

We can choose a rectangle $R_{\tilde{V}_2}$ in the four-punctured torus $\operatorname{cl}(\bar{F_2}_{V_2 W_2}-(p'_{\tilde{V}_2}\cup p'_{\tilde{W}_2}))$ such that one edge of $R_{\tilde{V}_2}$ belongs to one component of $\partial p'_{\tilde{V}_2}$, the opposite edge belongs to the other component of $\partial p'_{\tilde{V}_2}$, and the interior of the other two edges belongs to the interior of the four-punctured torus.
Moreover, we can find a rectangle $R_{\tilde{W}_2}$ in the thrice-punctured torus $\operatorname{cl}(\bar{F_2}_{V_2 W_2}-(p'_{\tilde{V}_2}\cup p'_{\tilde{W}_2}\cup R_{\tilde{V}_2}))$ such that two opposite edges of $R_{\tilde{W}_2}$ belong to the two components of $\partial p'_{\tilde{W}_2}$ and the interior of the other two edges belongs to the interior of the thrice-punctured torus similarly.
This gives two disjoint disks $D_{\W_1^2}=p'_{\tilde{V}_2}\cup R_{\tilde{V}_2}$ and $D_{\V_2^2}=p'_{\tilde{W}_2}\cup R_{\tilde{W}_2}$ in $\bar{F_2}_{V_2 W_2}$.\\

\item Case: $\bar{F_2}_{V_2 W_2}$ consists of two tori $\bar{F_2}_{V_2 W_2}^1$ and $\bar{F_2}_{V_2 W_2}^2$, i.e. $\mathbf{H}_2$ is a type (d) GHS.

In this case, one component of $p_{\tilde{V}_2}$ belongs to $\bar{F_2}_{V_2 W_2}^1$ and the other component belongs to $\bar{F_2}_{V_2 W_2}^2$ (see Figure \ref{fig-sit-3}).
The symmetric argument also holds for $p_{\tilde{W}_2}$.
Then we choose $D_{\W_1^2}^j$ as $p_{\tilde{V}_2}'\cap\bar{F_2}_{V_2 W_2}^j$ for $j=1,2$ and $D_{\V_2^2}^j$ as $p_{\tilde{W}_2}'\cap\bar{F_2}_{V_2 W_2}^j$ for $j=1,2$.
In this case, denote $D_{\W_1^2}^1\cup D_{\W_1^2}^2$ and $D_{\V_2^2}^1\cup D_{\V_2^2}^2$ as $D_{\W_1^2}$ and $D_{\V_2^2}$ respectively.\\
\end{enumerate}

Next, we isotope $f$ near $\bar{F_2}_{V_2 W_2}$ so that the disks (or the disk)  $f(\W_1^1\cap\V_1)\cap\bar{F_2}_{V_2 W_2}$ would belong to $D_{\W_1^2}$ and the disks (or the disk) $f(\V_2^1\cap \W_1)\cap\bar{F_2}_{V_2 W_2}$ would belong to $D_{\V_2^2}$ satisfying $f(\mathbf{H}_1)=\mathbf{H}_2$ at any time during the isotopy.
This argument can be generalized to the case when $\bar{F_2}_{V_2 W_2}$ is disconnected.
Then we get $f(\W_1^1\cap\V_1)\cap (\V_2^2\cap \W_2)= \emptyset$ and $f(\V_2^1\cap \W_1)\cap(\W_1^2\cap \V_2)=\emptyset$.
After this isotopy, $D_{\W_1^2}$ and $D_{\V_2^2}$ contain the image of scars $f((\W_1^1\cap\V_1)\cap \bar{F_1}_{V_1 W_1})=f(\W_1^1\cap\V_1)\cap \bar{F_2}_{V_2 W_2}$ and $f((\V_2^1\cap \W_1)\cap \bar{F_1}_{V_1 W_1})=f(\V_2^1\cap \W_1)\cap \bar{F_2}_{V_2 W_2}$ in their interiors as well as the scars $(\W_1^2\cap \V_2)\cap \bar{F_2}_{V_2 W_2}$ and $(\V_2^2\cap \W_2)\cap \bar{F_2}_{V_2 W_2}$ respectively (see Figure \ref{fig-separating}).
\begin{figure}
\includegraphics[width=12cm]{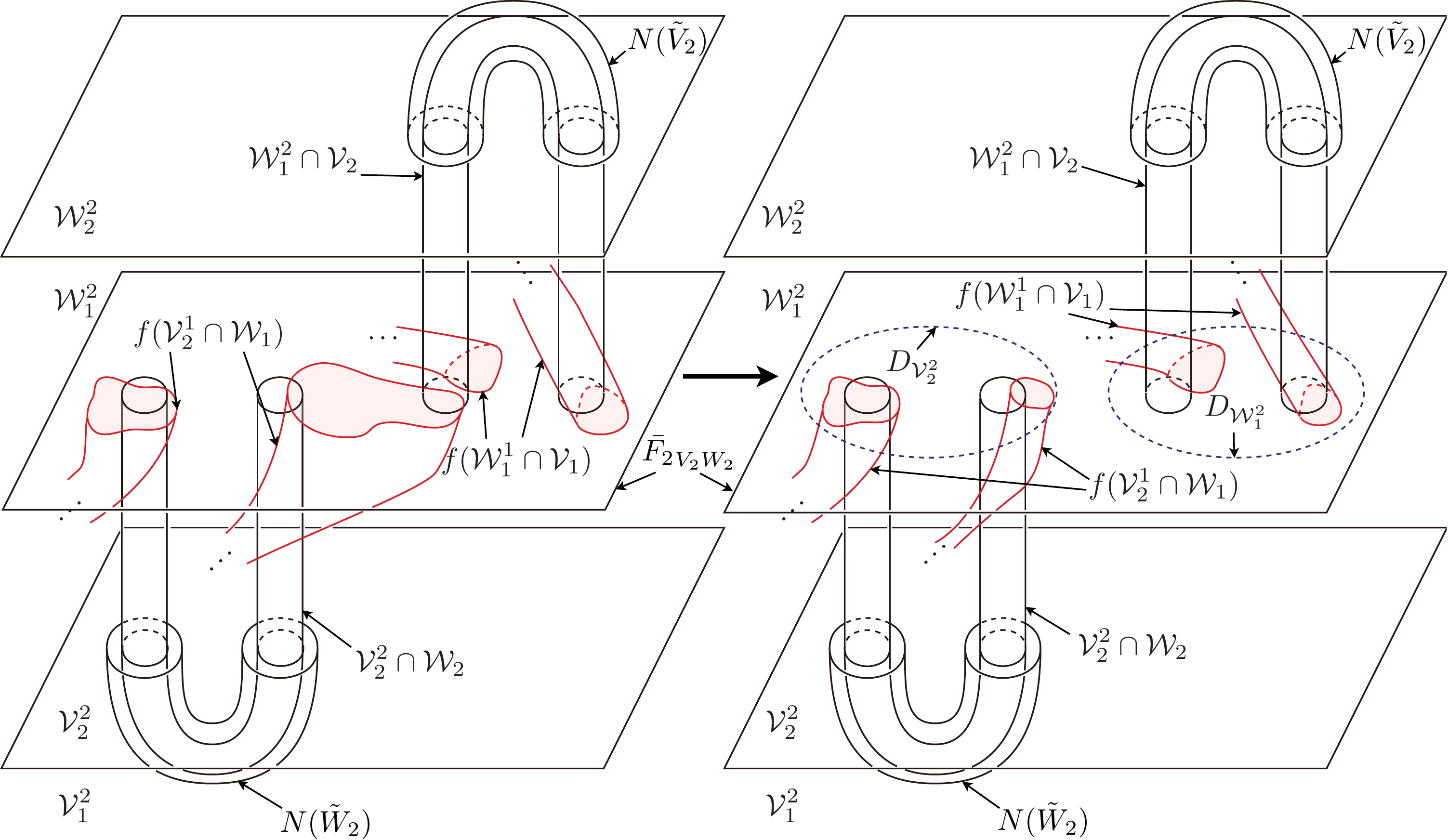}
\caption{We can assume that $f(\W_1^1\cap\V_1)\cap (\V_2^2\cap \W_2) = \emptyset$ and $f(\V_2^1\cap \W_1)\cap(\W_1^2\cap \V_2)=\emptyset$ after an isotopy of $f$. \label{fig-separating}}
\end{figure}
Note that we can assume that the assumption that $\alpha_2'$ and $\beta_2'$ are monotone in the relevant product structures of $\W_1^2$ and $\V_2^2$ are not changed after this step by using a suitable isotopy of $f$.

This completes the proof of Claim D.
\end{proofN}

In the following step, we will realize the untying isotopies of $f(\W_1^1\cap\V_1)$ and  $f(\V_2^1\cap \W_1)$.\\

\Step{A} we will isotope $f(\W_1^1\cap\V_1)$ into $\W_1^2\cap\V_2$ without affecting $f(\V_2^1\cap\W_1)$.\\

\Case{1} $\bar{F_i}_{V_i W_i}$ consists of a torus.\\

\Case{1-A} $\partial_-\W_1^i$ consists of a torus, i.e. a subcase for non-separating $V_i$.\\
In this case, $\partial_-\W_1^i=\bar{F}_{V_i W_i}$.\\

\ClaimN{E.1-A}{
Let $N_{\V_2^2}(\bar{F_2}_{V_2 W_2}-D_{\V_2^2})$ be a small product neighborhood of $\bar{F_2}_{V_2 W_2}-D_{\V_2^2}$  in $\V_2^2$.
After a sequence of isotopies of $f$ in $(\partial_-\W_1^2\times I)\cup(N_{\V_2^2}(\bar{F_2}_{V_2 W_2}-D_{\V_2^2}))$ which are the identity on $(\partial_-\W_1^2\times \{1\})$, $\alpha_2'$ becomes vertical in $\partial_-\W_1^2\times I$.
Moreover, $f(\mathbf{H}_1)=\mathbf{H}_2$ at any time during the isotopies.
This means that this sequence of isotopies does not affect $f(\V_2^1\cap \W_1)$.}

\begin{proofN}{Claim E.1-A}
From now on, we will represent an isotopy of the cylinder $f(C_1^1\cup N(V_1)\cup C_2^1)$ by that of $\alpha_2'$ for the sake of convenience. 

Recall that $\alpha_2'$ is parallel to $\alpha_2$ in $N(V_2)$.
But $\alpha_2'\cap (\partial_-\W_1^2\times I)$ is just a monotone 2-strands in $\partial_-\W_1^2\times I$ even though each component is unknotted (see Figure \ref{fig-4-braid}).\\
\begin{figure}
\includegraphics[width=6cm]{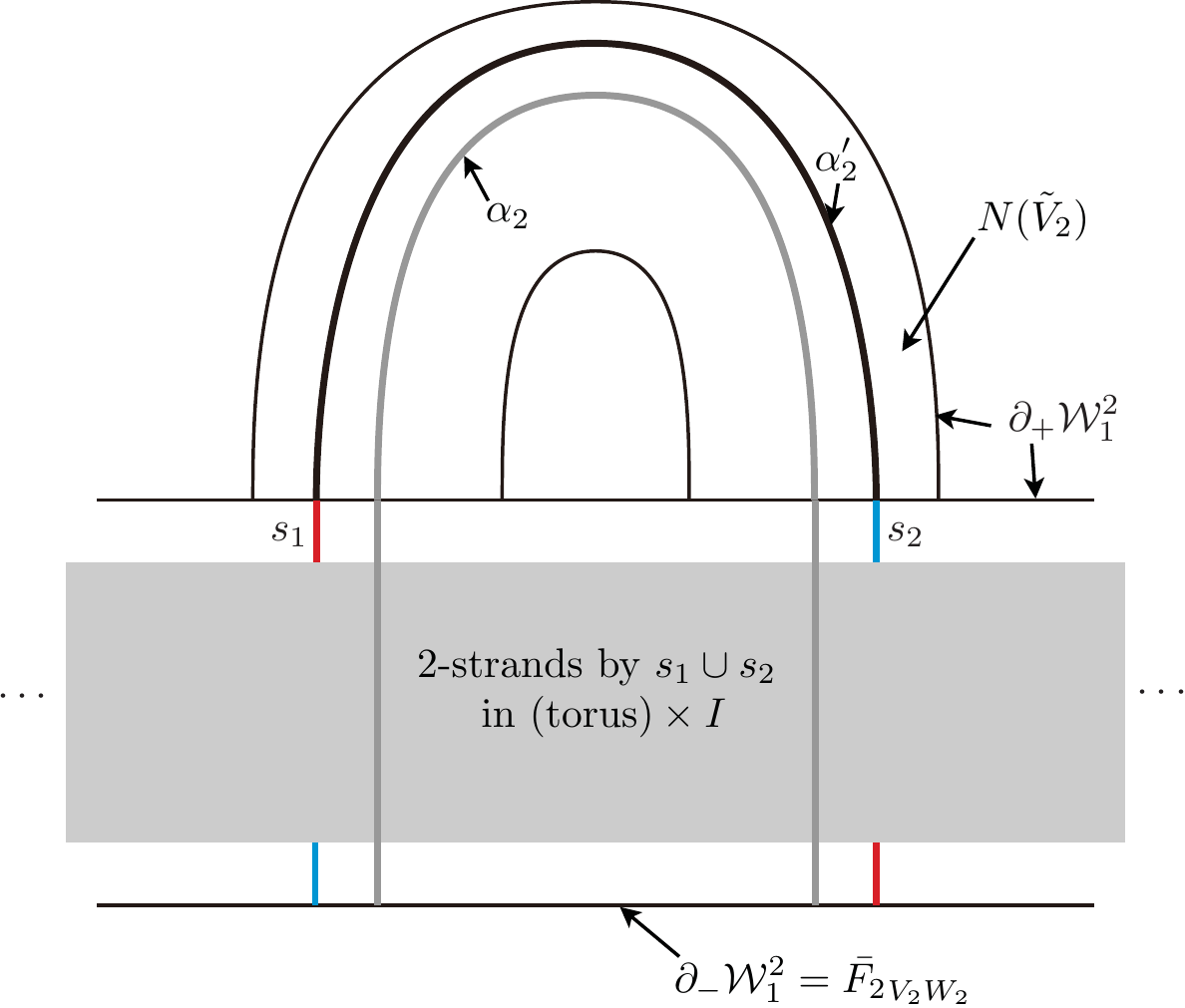}
\caption{the core arcs of $\W_1^2\cap \V_2$ and $f(\W_1^1\cap\V_1)$
 \label{fig-4-braid}}
\end{figure}

\Step{1: Normalize $\alpha_2'$ in $\partial_-\W_1^2\times I$}
In the proof of Step 1, we will denote $D_{\W_1^2}$ as $D$ for the sake of convenience.

Let $s_1$ and $s_2$ be the two strands of $\alpha_2'\cap (\partial_-\W_1^2\times I)$ such that $s_j$ is the core arc of $f(C_j^1)$ for $j=1,2$.
We isotope $f$ near $\bar{F_2}_{V_2 W_2}$ so that the projection of $(s_1\cup s_2)\cap (\partial_-\W_1^2\times\{1\})$ into $\bar{F_2}_{V_2 W_2}$ is equal to $(s_1\cup s_2)\cap \bar{F_2}_{V_2 W_2}$ and we say $p_j=s_j\cap\bar{F_2}_{V_2 W_2}$ for $j=1,2$.
Then we choose sufficiently small $\epsilon>0$ such that $(s_1\cup s_2)\cap(\partial_-\W_1^2\times ([0,\epsilon]\cup [1-\epsilon,1]))\subset \operatorname{int}(D)\times ([0,\epsilon]\cup [1-\epsilon,1])$ and reparametrize  $\partial_-\W_1^2\times I$ so that $\epsilon=\frac{1}{4}$.
Then we isotope $f$ so that each subarc of $s_i$ would be a vertical strand in $\partial_-\W_1^2\times ([0,\frac{1}{4}]\cup[\frac{3}{4},1])$ only affecting $\operatorname{int}(D)\times ([0,\frac{1}{4}+\epsilon']\cup[\frac{3}{4}-\epsilon',1])$ for sufficiently small $\epsilon'>0$.
After the isotopies and the reparametrization, $(p_1\cup p_2)\times ([0,\frac{1}{4}]\cup[\frac{3}{4},1])$ is equal to $\alpha'_2\cap (\partial_-\W_1^2\times ([0,\frac{1}{4}]\cup[\frac{3}{4},1]))$ and we can assume that $\alpha_2'$ remains monotone in $\partial_-\W_1^2\times [\frac{1}{4},\frac{3}{4}]$.

Let $q$ be a point in $\bar{F}_{V_2 W_2}$ missing $D_{\W_1^2}\cup D_{\V_2^2}$.
If we consider $\operatorname{cl}((\bar{F}_{V_2 W_2}-N(q))\times I)$ in $\partial_-\W_1^2\times I$, then it is  a genus two handlebody $\mathcal{H}$, where we assume a neighborhood $N(q)$ in $\bar{F}_{V_2 W_2}$ also misses $D_{\W_1^2}\cup D_{\V_2^2}$.
Here, we choose a meridian $c_1$ and a longitude $c_2$ of $\bar{F}_{V_2 W_2}$  such that (i) $c_1$ intersects $c_2$ transversely in exactly one point $q$, (ii) $c_1\cup c_2$ misses $D_{\W_1^2}\cup D_{\V_2^2}$.
Then $\{D_1=\operatorname{cl}(c_1-N(q))\times I, D_2=\operatorname{cl}(c_2-N(q))\times I\}$ forms a minimal defining set of $\mathcal{H}$. i.e. $\operatorname{cl}(\mathcal{H} - (N(D_1)\cup N(D_2)))$ is a $3$-ball $\mathcal{B}$, where $N(D_i)\cong D^2\times I$ is a product neighborhood of $D_i$ in $\mathcal{H}$ for $i=1,2$ so that each $D^2\times \{t\}$ is vertical in $\partial_-\W_1^2\times I$ for $0\leq t \leq 1$.
If we isotope $f$, then we can assume that (i) $\alpha_2'$ misses $N(q)\times I$, (ii) $\alpha_2'$ intersects $D_1\cup D_2$ transversely, and (iii) two different points of the intersection points $\{q_i\}_{i=1}^r=\alpha_2'\cap (D_1\cup D_2)$ are positioned at two different levels of $\partial_-\W_1^2 \times I$ without changing the assumption that $\alpha_2'$ is monotone in $\partial_-\W_1^2\times I$ (see (a) of Figure \ref{fig-normalize}).
\begin{figure}
\includegraphics[width=13cm]{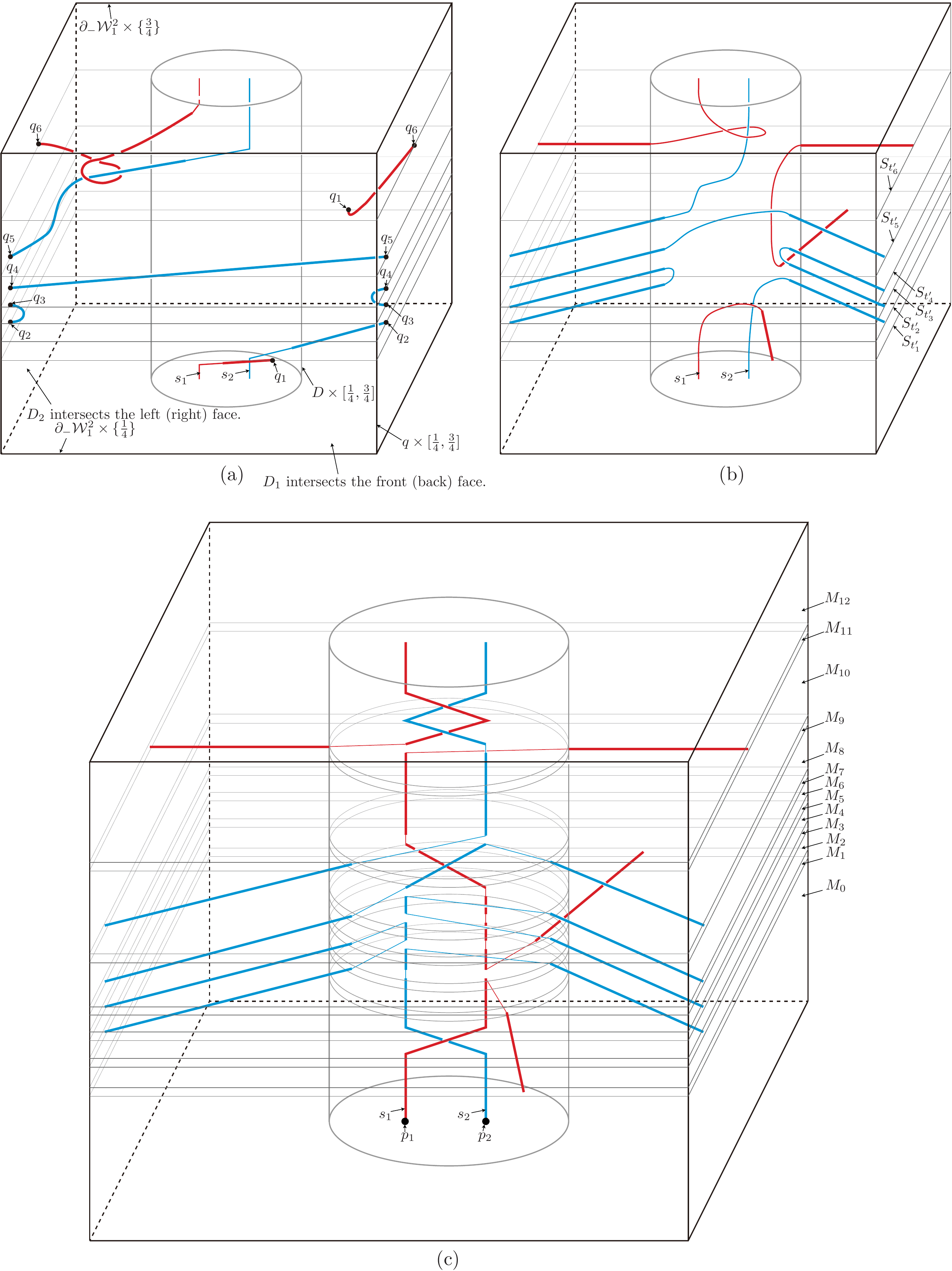}
\caption{the normalization procedure \label{fig-normalize}}
\end{figure}
Here, we assume that the indices of $\{q_i\}_{i=1}^r$ follow the order of levels of $\partial_-\W_1^2\times I$.
Let $\{t'_i\}_{i=1}^r$ ($t'_i\in (\frac{1}{4},\frac{3}{4})$) be the set of levels of $\partial_-\W_1^2 \times I$ corresponding to $\{q_i\}_{i=1}^r$.

Assume that (i) we've chosen  $N(D_1)\cup N(D_2)$ sufficiently thin so that  $N(D_1)\cap N(D_2)=\emptyset$  and (ii) whenever $\alpha_2'$ intersects $N(D_i)\cong D^2\times I$ for $i=1$ or $2$ in a subarc of $\alpha_2'$, it passes through each of  $D^2\times \{0\}$ and $D^2\times \{1\}$ transversely in exactly one point. 
Let $S_{t'_i}$ be the level surface $\partial_-\W_1^2\times \{t'_i\}$ and $\bar{S}_{t'_i}=\operatorname{cl}(S_{t'_i}-(N(q)\times I)-(N(D_1)\cup N(D_2)))$ therefore $\bar{S}_{t'_i}$ is a disk in the level surface $S_{t'_i}$.
Let us isotope $f$ so that this isotopy forces each component of $\alpha_2'\cap (N(D_1)\cup N(D_2))$ to belong to the corresponding $S_{t'_i}$ but monotone elsewhere in $\partial_-\W_1^2\times I$.

Let $\bar{S}$ be the disk $\operatorname{cl}(\bar{F}_{V_2 W_2}-N(q)-N(D_1)- N(D_2))\subset\bar{F}_{V_2 W_2}$.
If we isotope $f$ so that  the $3$-ball $\bar{S}\times [\frac{1}{4}+\epsilon',\frac{3}{4}-\epsilon']$ shrinks into the $3$-ball $D\times  [\frac{1}{4}+\epsilon',\frac{3}{4}-\epsilon']$  for very small $\epsilon'>0$ preserving each level (i.e. the genus two handlebody $((N(q)\times I) \cup N(D_1)\cup N(D_2))\cap(\partial_-\W_1^2\times [\frac{1}{4}+\epsilon',\frac{3}{4}-\epsilon'])$ expands into $\operatorname{cl}(\partial_-\W_1^2-D)\times [\frac{1}{4}+\epsilon',\frac{3}{4}-\epsilon']$, see (a)$\to$(b) in Figure \ref{fig-normalize}), then we can assume that 
\begin{enumerate}
	\item this isotopy only affects $\partial_-\W_1^2\times [\frac{1}{4},\frac{3}{4}]$,
	\item $\alpha_2'$ is horizontal in $\operatorname{cl}((\partial_-\W_1^2\times I)-(D\times I))$, and
	\item  $\alpha_2'$ belongs to $\operatorname{int}(D)\times I$ in the complement of $\{S_{t'_i}\}_{i=1}^r$ in $\partial_-\W_1^2\times I$.
\end{enumerate}
If we cut off $\operatorname{int}(D)\times I$ along the level surfaces $\{S_{t'_i}\}_{i=1}^r$, then we can see that (i) the closure of each component, say $\mathcal{B}_i$, is  a $3$-ball, such that $\mathcal{B}_i$ intersects $S_{t'_i}\cup S_{t'_{i+1}}$ if we say $t'_0=0$ and $t'_{r+1}=1$ and  (ii) each of $s_1$ and $s_2$ intersects $\mathcal{B}_{i}$ in a connected arc whose interior belongs to $\operatorname{int}(\mathcal{B}_{i})$ for every $0\leq i \leq r$.
Since each component of $(s_1\cup s_2)\cap \mathcal{B}_{i}$ is monotone in $\mathcal{B}_i$, $(s_1\cup s_2)\cap \mathcal{B}_i$ forms a $2$-braid in $\mathcal{B}_{i}$.

Hence, we normalize $s_1\cup s_2$ in $\partial_-\W_1^2\times [\frac{1}{4},\frac{3}{4}]$ as follows.

\begin{enumerate}
\item \label{nor0}
 Let $\delta=\frac{1}{4}\min_{i=0}^r(t'_{i+1}-t'_i)$, $\mathcal{B}_0^{\delta}=D\times [0,t'_{1}-\delta]\subset\mathcal{B}_0$, $\mathcal{B}_i^{\delta}=D\times [t'_i+\delta,t'_{i+1}-\delta]\subset\mathcal{B}_i$ for $1\leq i \leq r-1$, and  $\mathcal{B}_r^{\delta}=D\times [t'_r+\delta,1]\subset\mathcal{B}_r$.
 
Let $a_0=0$, $a_1=t'_1-\delta$, $a_2=t'_1+\delta$, $\cdots$, $a_{2r-1}=t'_r-\delta$, $a_{2r}=t'_r+\delta$, and $a_{2r+1}=1=a_n$.
\item \label{nor1}
We isotope $f$ so that the $2$-braid by the subarcs of $s_1\cup s_2$ in $\mathcal{B}_i$ would become a ``\textit{standardly positioned $2$-braid}'' with respect to the vertical direction, i.e. (i) every non-trivial twist of $s_1\cup s_2$ in $\mathcal{B}_i$ belongs to $\mathcal{B}_i^\delta$ and (ii) the image of the projection of the endpoints of this $2$-braid in $\mathcal{B}_i^\delta$ into $\bar{F_2}_{V_2 W_2}$ would be $\{p_1,p_2\}$ for $0\leq i \leq r$ (see (c) of Figure \ref{fig-normalize}).
\item Cut $\partial_-\W_1^2\times I$ along the level surfaces $\{S_{t_i'-\delta},S_{t_i'+\delta}\}_{i=1}^r$ and let $M_j$ be the closure of each component.
Hence, we get the set of submanifolds $\{M_j\}_{j=0}^{2r}$, where the index increase from the bottom level to the top level and therefore $M_0$ intersects $\partial_-\W_1^2\times \{0\}$ and $M_{2r}$ intersects $\partial_-\W_1^2\times \{1\}$.
\item If $j$ is even, then $(s_1\cup s_2)\cap M_j$ belongs to some $\mathcal{B}_k^\delta$.
\item If $j$ is odd, then we can assume that one of $(s_1\cup s_2)\cap M_j$ is vertical in $M_j$ by adding an additional ``\textit{position-changing half-twist}'' to the top of the corresponding $2$-braid in $M_{j-1}$ if we need. (See Figure \ref{fig-position-change}. If we see (c) of Figure \ref{fig-normalize}, then the thin vertical subarcs of $s_1\cup s_2$ in $D\times [\frac{1}{4},\frac{3}{4}]$ denotes the vertical part of $s_1\cup s_2$ in $M_j$ for odd $j$.)
\begin{figure}
\includegraphics[width=10cm]{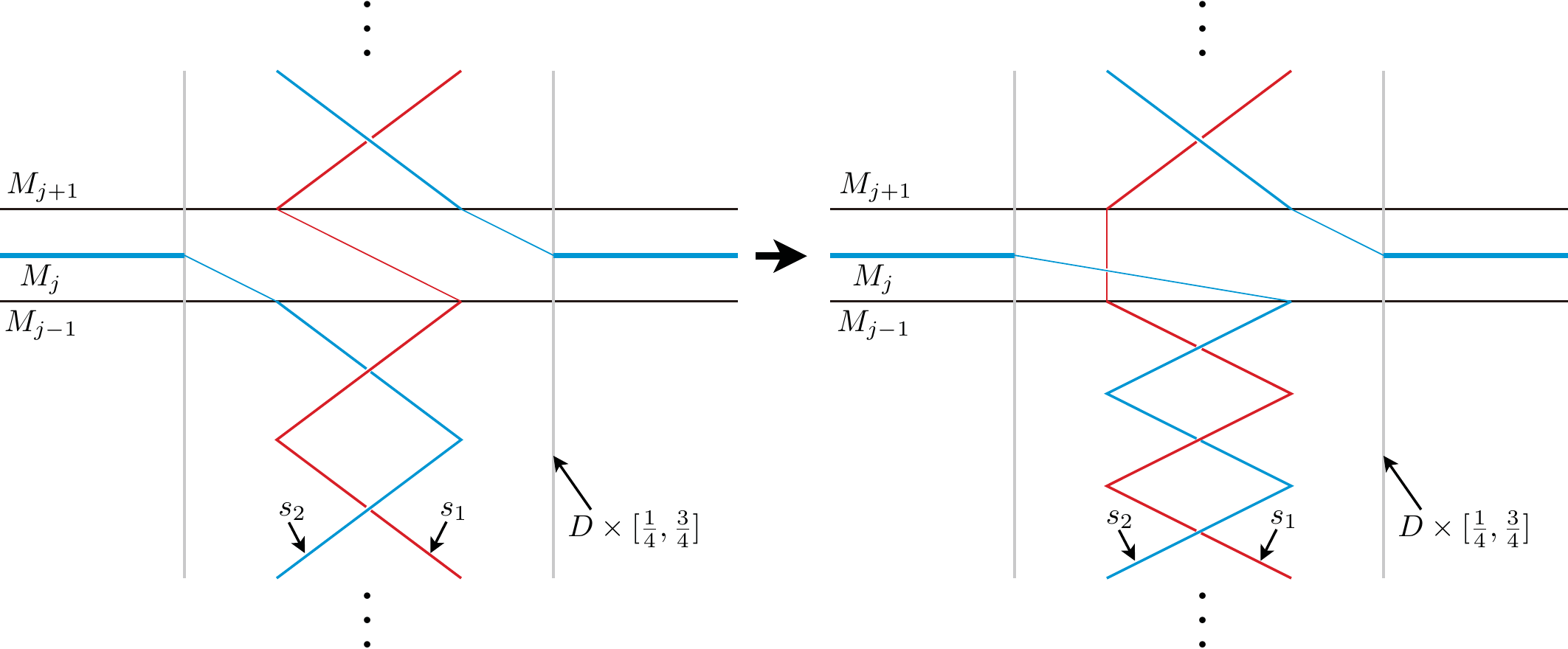}
\caption{the position-changing half-twist\label{fig-position-change}}
\end{figure} 
If the $2$-braid in $M_{j-1}$ is left-handed (right-handed resp.), then we add a left-handed (right-handed resp.) position-changing half-twist for the sake of convenience.
\end{enumerate}

\Step{2: Untying $f(C_1^1)$ and $f(C_1^2)$ into vertical cylinders}

As we did in Step 1, $\partial_-\W_1^2\times [0,1]$ is divided into $n$-submanifolds $\cup_{i=1}^{n}(\partial_-\W_1^2\times[a_{i-1},a_i])$ with respect to $s_1\cup s_2$ satisfying the follows.
\begin{enumerate} 
\item $a_0=0$, $a_n=1$, and $a_{i-1}<a_{i}$,
\item If $s_k$  intersects $(\partial_-\W_1^2-D)\times [a_{i-1},a_i]$ for $k=1$ or $2$, then a subarc of $s_k$ travels horizontally in $(\partial_-\W_1^2-D)\times \{a_{i-1}'\}$ for some $a_{i-1}<a_{i-1}'<a_i$ and $s_l\cap \operatorname{int}(D)\times [a_{i-1},a_i]$ is a vertical subarc for $l\neq k$.\label{step2a}
\item If $(s_1\cup s_2)\cap((\partial_-\W_1^2-D)\times [a_{i-1},a_i])=\emptyset$, then $(s_1\cup s_2)\cap(D\times [a_{i-1},a_i])$ is a standardly positioned $2$-braid in $D\times[a_{i-1},a_i]$ with respect to the vertical direction.\label{step2b}\\
\end{enumerate}

Let us consider $(s_1\cup s_2)\cap (D\times[a_0,a_1])$ which is a $2$-braid consisting of subarcs of $s_1\cup s_2$.
This $2$-braid can be written by $\sigma^k$ for $k\in\mathbb{Z}$, where $\sigma$ is a right-handed half-twist between the corresponding subarcs of $s_1\cup s_2$.
Let $D'$ be a disk in the interior of $D$ such that $D'\times [a_0, a_1]$ contains the two cylinders $(f(C_1^1)\cup f(C_2^1))\cap(D\times[a_0,a_1])$.
Here, we can isotope $f$ so that this makes this $2$-braid into a new $2$-braid with the representation $\sigma^{k'}$ such that $|k'|=|k|-1$ (see Figure \ref{fig-untwisting}) and therefore we 
can repeat such  isotopy over and over again until $s_1\cup s_2$ becomes vertical strands in $D\times [a_0, a_1]$.
Note that we can assume that this isotopy does not affect the outside of $(D'\times [a_0,a_1])$ in $f(\W_1^1)$, does affect only a small product neighborhood of $D'$ in $f(\V_2^1)$, but the compression body $f(\W_1^1)$ is preserved at any time during this isotopy  setwisely.
\begin{figure}
\includegraphics[width=6cm]{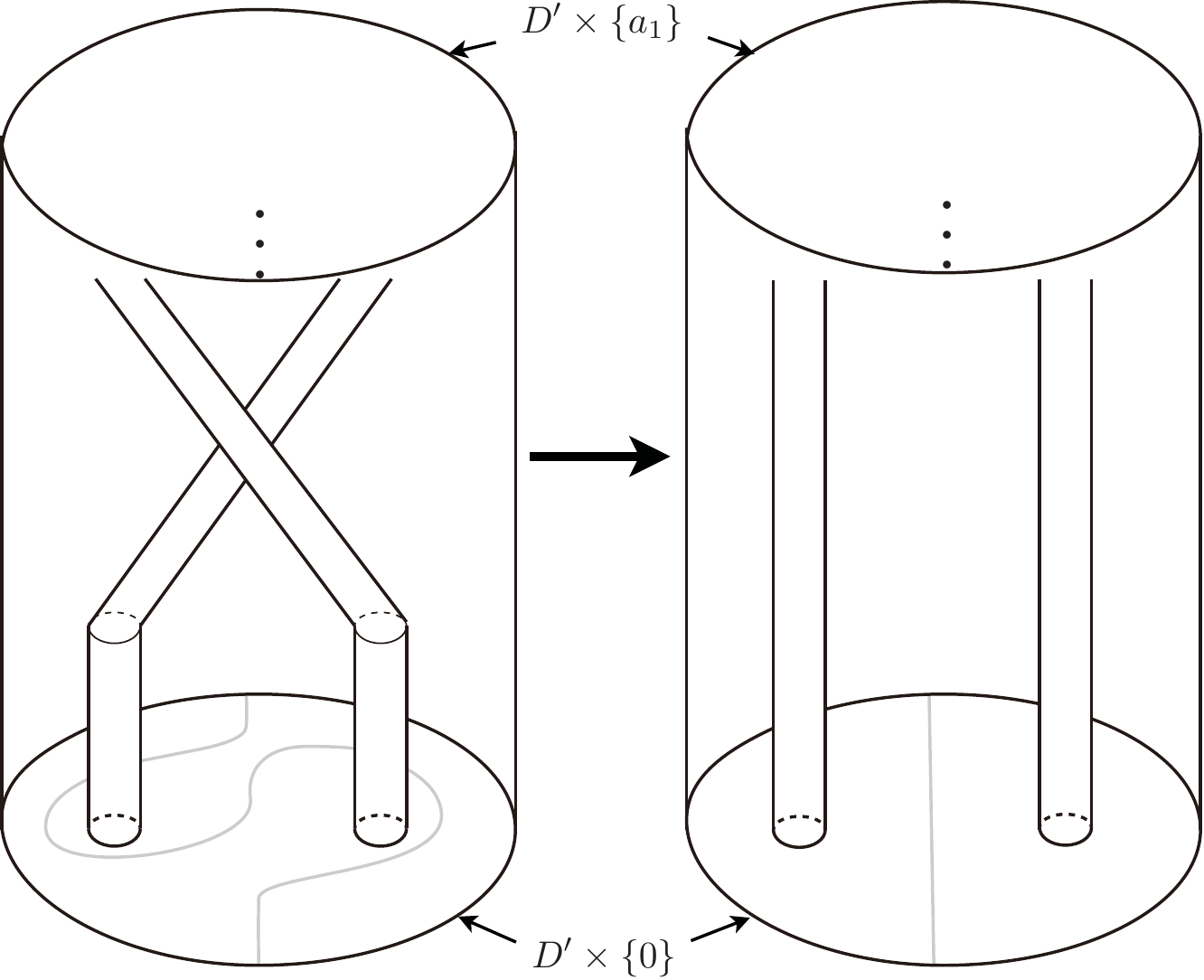}
\caption{untwisting a $2$-braid
 \label{fig-untwisting}}
\end{figure}
After this isotopy, $(s_1\cup s_2)\cap (D\times[a_0,a_1])$ becomes a trivial $2$-braid.
Therefore, if $n=1$, then we've isotoped $f$ so that $\alpha_2'$ became vertical in $\partial_-\W_1^2\times I$ and we've reached the end of the proof of Claim E.1-A.

Hence, we get $n>1$ and therefore $s_i\cap ((\partial_-\W_1^2-D)\times[a_1,a_2])\neq \emptyset$ for some $i=1$ or $2$, say $s_1$.
In this case, $s_2$ is vertical in $D\times[a_1,a_2]$ and a subarc of $s_1$ travels in $(\partial_-\W_1^2-D)\times\{a_1'\}$ for some $a_1< a_1'< a_2$ during the time when it leaves $D\times[a_1,a_2]$.
If we shrink $f(C_1^1)$ into sufficiently thinner one in $\partial_-\W_1^2\times [0,a_2]$ by an isotopy of $f$ and project $C=f(C_1^1)\cap (\partial_-\W_1^2\times[0,a_2])$ into $\bar{F_2}_{V_2 W_2}$, then we get an annulus, say the ``\textit{shadow}'', and denote it as $R$ (see Figure \ref{fig-travel}) and $R\cap D$ is a rectangle which divides $D$ into two pieces.
\begin{figure}
\includegraphics[width=12cm]{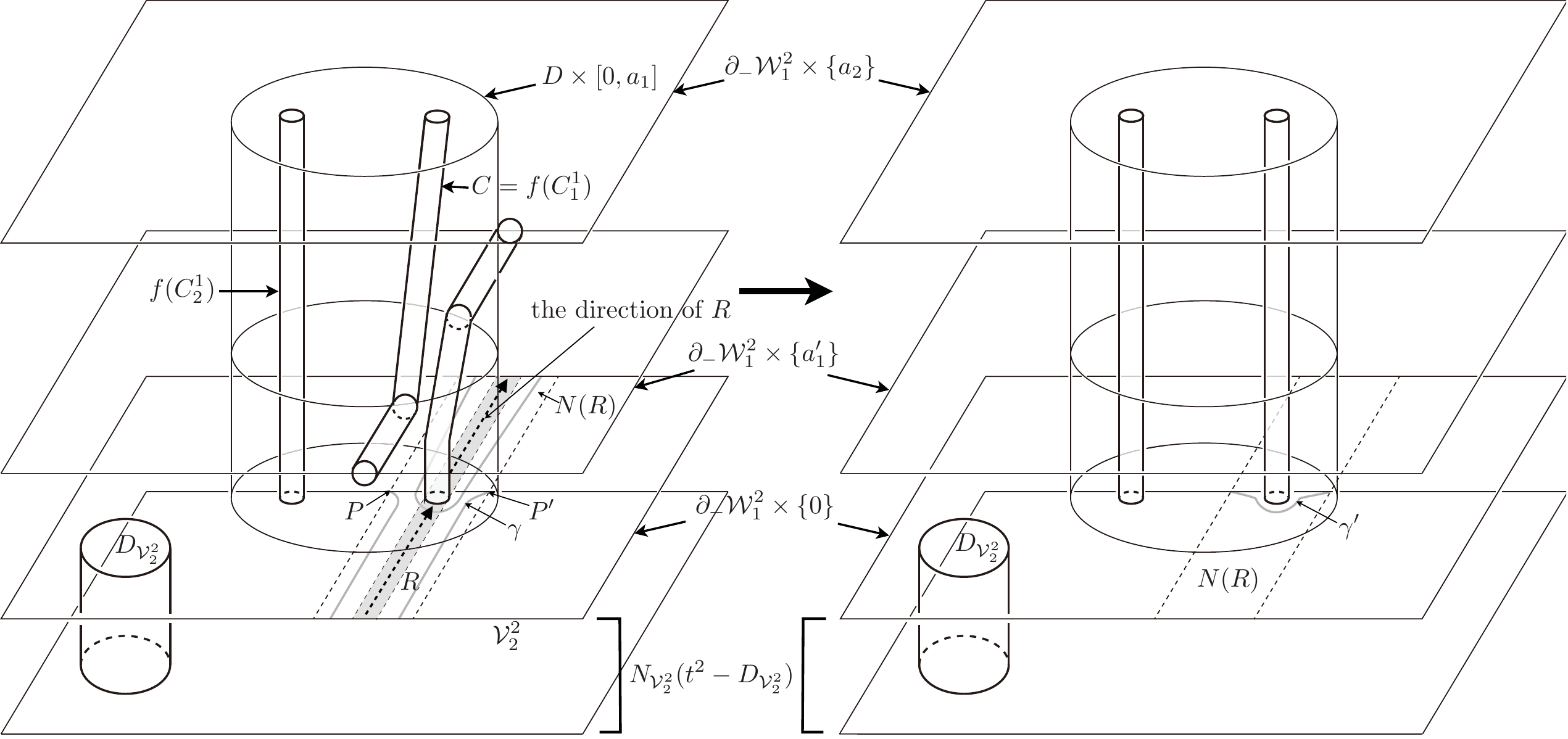}
\caption{$C$ becomes a vertical cylinder.\label{fig-travel}}
\end{figure}
If $R\cap D_{\V_2^2}\neq \emptyset$, then we isotope $f$ so that $R$ would miss $D_{\V_2^2}$ (see the left of Figure \ref{fig-travel}).
Choose a small neighborhood of $R$ in $\bar{F_2}_{V_2 W_2}$, say $N(R)$, so that (i) $\partial N(R)\cap\operatorname{int}(D)$ would consist of two arcs, (ii) $N(R)\cap D_{\V_2^2}= \emptyset$, and (iii) $N(R)\cap f(C_2^1)=\emptyset$.
We can give the canonical direction to the core circle of $R$ such that it follows the direction where the level of $s_1$ increases.
Choose two points $P$ and $P'$ in $\partial N(R)$ such that they are contained in different components of $\partial N(R)\cap\operatorname{int}(D)$.
Let $\pi_1(N(R))=<\alpha>$ where the direction of $\alpha$ is the same as the direction of the core circle of $R$.
Consider a curve $\gamma$ such that $\gamma$ starts from $P$, it travels the interior of $N(R)$ as much as $\alpha^{-1}$, turns around along the half of $\partial N(C\cap\bar{F_2}_{V_2 W_2})$, where $N(C\cap\bar{F_2}_{V_2 W_2})$ means a sufficiently small neighborhood of $C\cap\bar{F_2}_{V_2 W_2}$ in $\bar{F_2}_{V_2 W_2}$, and travels the interior of $N(R)$ as much as $\alpha$ until it ends at $P'$ (see the left of Figure \ref{fig-travel}).
Here, $\gamma$ meets $\partial D$ four times.
If we isotope $f$ so that $\gamma$ shrinks into a curve $\gamma'\subset N(R)$ contained in the interior of $D$, then we can assume that this isotopy make $C$ into a vertical cylinder and it does not affect $f(C_2^1)$ (see the right of Figure \ref{fig-travel}).
Moreover, we can assume that the compression body $f(\W_1^1)$ is preserved at any time during this isotopy setwisely.
(But it affects the image of $f$ in a small product neighborhood of $\bar{F_2}_{V_2 W_2}-D_{\V_2^2}$ in $\V_2^2$, say $N_{\V_2^2}(\bar{F_2}_{V_2 W_2}-D_{\V_2^2})$.)
After this isotopy, we can reduce the $n$-submanifolds $\cup_{i=1}^n(\partial_-\W_1^2\times[a_{i-1},a_i])$ of $\partial_-\W_1^2\times [0,1]$ into $(n-2)$-submanifolds.\\

Therefore, if we repeat the arguments in the previous paragraph, then $f(C_1^1)$ and $f(C_1^2)$ would become vertical cylinders in $\partial_-\W_1^2\times I$  and  $f(\mathbf{H}_1)=\mathbf{H}_2$ at any time during the isotopies of $f$ without affecting $D_{\V_2^2}$.

This completes the proof of Claim E.1-A.
\end{proofN}

\Case{1-B} $\partial_-\W_1^i$ consists of two tori, i.e. $V_i$ cuts off $(\text{torus})\times I$ from $\V_i$.

In this case, $\operatorname{cl}(\tilde{\W}_1^2- N(\tilde{V}_2))$ consists of two $(\text{torus})\times I$s, where $f(C_1^1)$ belongs to one $(\text{torus})\times I$ and $f(C_2^1)$ belongs to the other.
The relevant product structures are $\mathcal{X}_1^2$ and $\mathcal{X}_2^2$ in $\tilde{\W}_1^2$ (see the left of Figure \ref{fig-4-braid-1-B}).
\begin{figure}
\includegraphics[width=11cm]{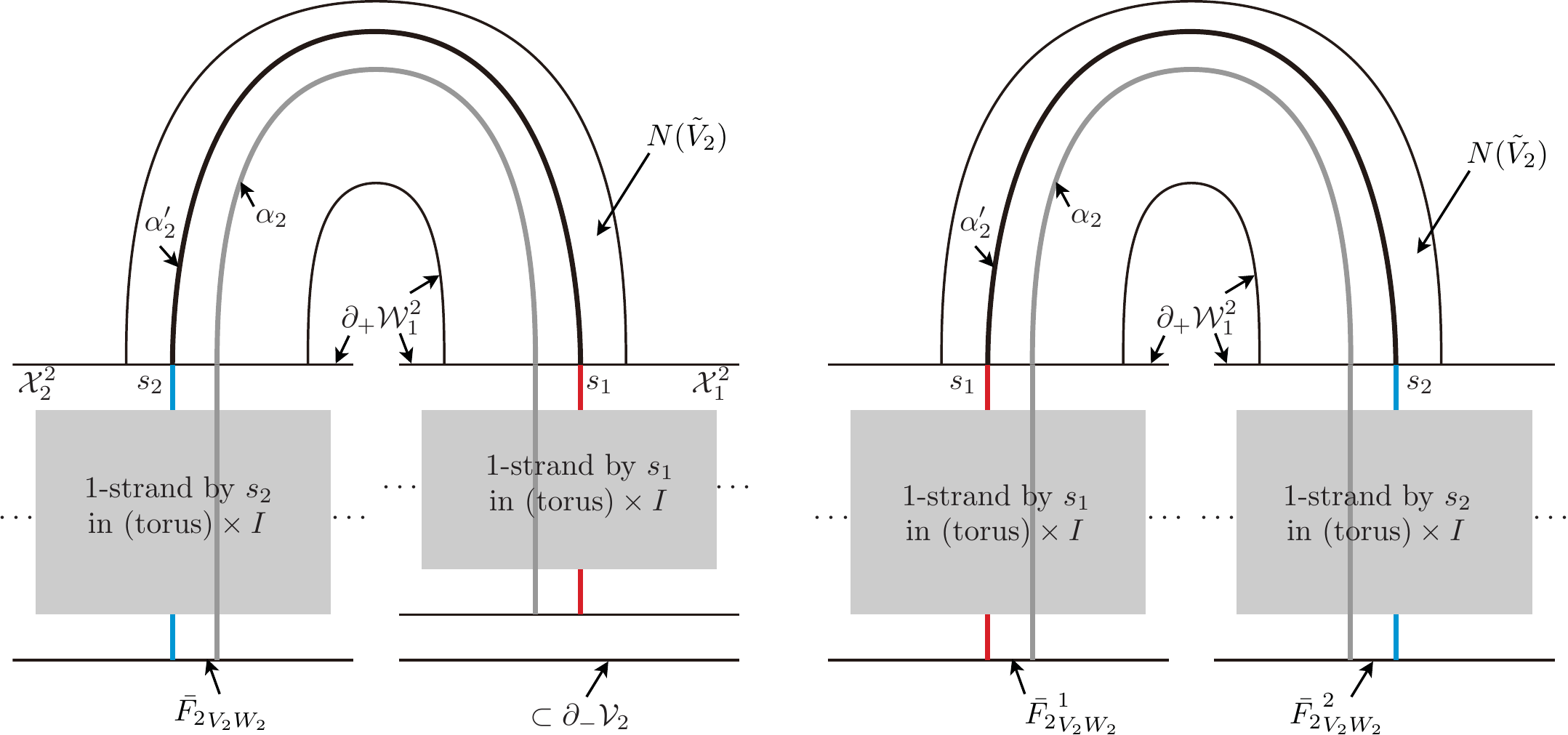}
\caption{the core arcs of $C_1^1\cup N(V_1)\cup C_2^1$ and $f(C_1^1\cup N(V_1)\cup C_2^1)$ when $\partial_-\W_1^2$ is disconnected
 \label{fig-4-braid-1-B}}
\end{figure}
Therefore,  we only need to consider a $1$-strand in each of $\mathcal{X}_1^2$ and $\mathcal{X}_2^2$.
Moreover, $\mathcal{X}_1^2\cap \bar{F_2}_{V_2 W_2}=\emptyset$, i.e. we don't worry about the possibility that $f(C_1^1)$ meets $f(\V_2^1\cap \W_1)$ in the untying procedure in $\mathcal{X}_1^2$.
This means that the untying procedure is more easier than Case 1-A.
Hence, we can isotope $\alpha_2'$ so that it would be vertical in $\partial_-\tilde{\W}_1^2\times I$ without affecting $D_{\V_2^2}$ by the similar arguments as in Case 1-A.
We can assume that this isotopy preserves $f(\tilde{\W}_1^1)$ setwisely as well as $f(\W_1^1)$.\\

\Case{2} $\bar{F_i}_{V_i W_i}$ consists of two tori $\bar{F_i}_{V_i W_i}^1$ and $\bar{F_i}_{V_i W_i}^2$.

If we consider $\operatorname{cl}(\W_1^2- N(\tilde{V}_2))$, then it consists of two $(\text{torus})\times I$s, where $f(C_1^1)$ belongs to one $(\text{torus})\times I$ and $f(C_1^2)$ belongs to the other.
Hence, the untying procedure is essentially the same as Case 1-B except that both $f(C_1^1)$ and $f(C_1^2)$ intersect $\bar{F_2}_{V_2 W_2}$ (see the right of Figure \ref{fig-4-braid-1-B}).
Hence, we can isotope $\alpha_2'$ so that it would be vertical in $\partial_-\W_1^2\times I$ without affecting $D_{\V_2^2}^j$ for $j=1,2$ by the similar arguments as in Case 1-B.\\

After the untying procedure in Case 1-A, Case 1-B or Case 2, $\alpha'$ becomes to be apparentely parallel to $\alpha$ in $\W_1^2$.
Moreover, we can isotope $f$ so that the cylinder $f(C_1^1\cup N(V_1)\cup C_2^1)$ would be moved into $C_1^2\cup N(V_2)\cup C_2^2$ in $\W_1^2$ without affecting $D_{\V_2^2}$.
This means that we have isotoped $f(\W_1^1\cap\V_1)$ into $\W_1^2\cap\V_2$  without affecting $f(\V_2^1\cap\W_1)$ after Claim D.
Moreover, these isotopies satisfy $f(\mathbf{H}_1)=\mathbf{H}_2$ at any time.\\

\Step{B} After Step A, if we use the symmetric arguments in Step A, then we can isotope $f(\V_2^1\cap\W_1)$ into $\V_2^2\cap \W_2$  without affecting $f(\W_1^1\cap \V_1)$ at any time.
Moreover, this isotopy satisfies $f(\mathbf{H}_1)=\mathbf{H}_2$ at any time.\\

After Step A and Step B, $f$ have been isotoped so that $f(\W_1^1\cap\V_1)=\W_1^2\cap\V_2$ and $f(\V_2^1\cap\W_1)=\V_2^2\cap \W_2$ satisfing $f(\mathbf{H}_1)=\mathbf{H}_2$ at any time during the isotopy.

This completes Lemma \ref{lemma-spines}.
\end{proof}

By using the isotopies of Lemma \ref{lemma-spines}, $f$ has been isotoped so that it satisfies the following equation.
\begin{eqnarray*}
f(\V_1)&=&f(\operatorname{cl}(\V_1^1\cup\V_2^1-(\V_2^1\cap\W_1))\cup (\W_1^1\cap\V_1))\\
	&=&\operatorname{cl}(f(\V_1^1)\cup f(\V_2^1)-f(\V_2^1\cap\W_1))\cup f(\W_1^1\cap\V_1)\\
	&=&\operatorname{cl}(\V_1^2 \cup \V_2^2 -(\V_2^2\cap\W_2))\cup (\W_1^2\cap\V_2) = \V_2.
\end{eqnarray*}
This completes the proof of Theorem \ref{lemma-determine-GHSs}.
\end{proof}

\begin{definition}\label{definition-HS-GHS}
Let $F$ be a weakly reducible, unstabilized Heegaard surface of genus three in an orientable, irreducible $3$-manifold $M$.
Let $\mathcal{GHS}_F$ be the set of isotopy classes of the generalized Heegaard splittings obtained by weak reductions from $(\V,\W;F)$.
If there is a generalized Heegaard splitting $\mathbf{H}$ obtained by weak reduction from $(\V,\W;F)$ and its isotopy class is $[\mathbf{H}]\in\mathcal{GHS}_F$, then we call $\mathbf{H}$ a \textit{representative of $[\mathbf{H}]$ coming from weak reduction}.
We will say two representatives $\mathbf{H}_1=\{\bar{F}_{V_1},\bar{F}_{V_1 W_1},\bar{F}_{W_1}\}$ and $\mathbf{H}_2=\{\bar{F}_{V_2},\bar{F}_{V_2 W_2},\bar{F}_{W_2}\}$ of $[\mathbf{H}]\in\mathcal{GHS}_F$ coming from weak reductions are \textit{equivalent} if (i) $\bar{F}_{V_1}$ is isotopic to $\bar{F}_{V_2}$ in $\V$, (ii) $\bar{F}_{W_1}$ is isotopic to $\bar{F}_{W_2}$ in $\W$, and (iii) $\bar{F}_{V_1 W_1}$ is isotopic to $\bar{F}_{V_2 W_2}$ in $M$.

Suppose that  $\bar{F}_{V_1}$ is isotopic to $\bar{F}_{V_2}$ in $\V$ for two representatives $\mathbf{H}_1=\{\bar{F}_{V_1},\bar{F}_{V_1 W_1},\bar{F}_{W_1}\}$ and $\mathbf{H}_2=\{\bar{F}_{V_2},\bar{F}_{V_2 W_2},\bar{F}_{W_2}\}$ of some isotopy classes $[\mathbf{H}_1]$ and $[\mathbf{H}_2]\in\mathcal{GHS}_F$ respectively coming from weak reductions.
If we recall the proof of Theorem \ref{theorem-structure} in \cite{JungsooKim2014-2}, then $(V_1, W_1)$ and $(V_2,W_2)$ belong to the same component of $\DVW(F)$ and therefore $\bar{F}_{W_1}$ is isotopic to $\bar{F}_{W_2}$ in $\W$ and the thin level $\bar{F}_{V_1 W_1}$ is isotopic to $\bar{F}_{V_2 W_2}$ in $M$ by the definitions of building blocks and Theorem \ref{lemma-just-BB}.
Moreover, $\mathbf{H}_1$ is isotopic to $\mathbf{H}_2$ as well as each thick or thin level is isotopic to the relevant thick or thin level, i.e. $[\mathbf{H}_1]=[\mathbf{H}_2]$ in $\mathcal{GHS}_F$.
This means that $\mathbf{H}_1$ is equivalent to $\mathbf{H}_2$.
Therefore, \textit{$\mathbf{H}_1$ is equivalent to $\mathbf{H}_2$  if and only if at least one thick level of one representative is isotopic to that of the other in the relevant compression body}.

Obviously, this gives an equivalent relation to the set of all representatives of the elements of $\mathcal{GHS}_F$ coming from weak reductions.
Let $\overline{\mathcal{GHS}}_F$ be the set of all these equivalent classes and we denote the equivalent class of a representative $\mathbf{H}$ as $(\mathbf{H})$.
If there is a component of $\DVW(F)$, then every weak reducing pair in the component gives the same equivalent class in $\overline{\mathcal{GHS}}_F$ after weak reduction by Theorem \ref{theorem-structure}.
Hence, this defines the function $\Phi_F:\{\text{the components of }\DVW(F)\}\to\overline{\mathcal{GHS}}_F$.\\

\ClaimN{A}{
$\Phi_F$ is bijective.}

\begin{proofN}{Claim A}
If we consider an element of $\overline{\mathcal{GHS}}_F$, then there must be a weak reducing pair in $\DVW(F)$ realizing a representative of the equivalent class by weak reduction.
This gives the component of $\DVW(F)$ containing the weak reducing pair, i.e. $\Phi$ is surjective.

Suppose that $\Phi_F(\mathcal{B}_1)=\Phi_F(\mathcal{B}_2)$ for some components $\mathcal{B}_1$ and $\mathcal{B}_2$ of $\DVW(F)$.
This means that every weak reducing pair in $\mathcal{B}_1\cup\mathcal{B}_2$ gives the same equivalent class in $\overline{\mathcal{GHS}}_F$ by weak reduction, i.e. this gives a uniquely determined isotopy class of the thick level contained in $\V$.
Hence, Theorem \ref{theorem-structure} induces $\mathcal{B}_1=\mathcal{B}_2$ and therefore $\Phi_F$ is injective.

This completes the proof of Claim A.
\end{proofN}

By Claim A, $\Phi_F$ gives a one-to-one correspondence between the components of $\DVW(F)$ and the equivalent classes in $\overline{\mathcal{GHS}}_F$.
\end{definition}

Finally, we reach Corollary \ref{corollary-GHS-HS}.

\begin{corollary}[Theorem \ref{theorem-GHS-HS}]\label{corollary-GHS-HS}
Let $(\V,\W;F)$ and $(\V',\W';F')$ be weakly reducible, unstabilized, genus three Heegaard splittings in an orientable, irreducible $3$-manifold $M$ and $f$ an orientation-preserving automorphism of $M$.
Then $f$ sends $F$ into $F'$ up to isotopy if and only if $f$ sends a representative of an element of $\mathcal{GHS}_F$ coming from weak reduction into a representative of an element of $\mathcal{GHS}_{F'}$ coming from weak reduction up to isotopy.
\end{corollary}

\begin{proof}
$(\Rightarrow)$ Suppose that $f$ sends $F$ into $F'$ up to isotopy.
That is, we can isotope $f$ so that $f(F)=F'$.
Let $[\mathbf{H}]$ be an element of $\mathcal{GHS}_F$.
Then there is a weak reducing pair $(V,W)$ of $(\V,\W;F)$ which gives a representative $\mathbf{H}=\{\bar{F}_{V},\bar{F}_{VW},\bar{F}_{W}\}$ of $[\mathbf{H}]$ coming from weak reduction.
If we consider the weak reducing pair determined by $\{f(V),f(W)\}$ of $F'$, then it gives the generalized Heegaard splitting $\mathbf{H}'=\{\bar{F}_{f(V)}',\bar{F}_{f(V) f(W)}',\bar{F}_{f(W)}'\}$ obtained by weak reduction from $(\V',\W';F')$.\\

\ClaimN{A}{$f(\mathbf{H})$ is a representative of an element of $\mathcal{GHS}_{F'}$ coming from weak reduction.
Moreover, $(\mathbf{H}')=(f(\mathbf{H}))$ in $\overline{\mathcal{GHS}}_{F'}$.}

\begin{proofN}{Claim A}
Recall that $f(F)=F'$.
Without loss of generality, assume that $f(V)\subset \V'$ and $f(W)\subset \W'$, i.e. $f(\bar{F}_V)\subset \V'$ and $f(\bar{F}_W)\subset \W'$.

Let us consider $f(\mathbf{H})=\{f(\bar{F}_{V}), f(\bar{F}_{V W}), f(\bar{F}_W)\}$ and observe the compressing disks $f(V)$ and $f(W)$.
Let $\tilde{\V}$ be the region in $\V$ between the genus two component of $F_V$ and $\bar{F}_{V}$ where ``the genus two component of $F_V$'' is the one used when we obtained the thick level $\bar{F}_{V}$.
Let $N_\V(V)$ be the product neighborhood of $V$ in $\V$ which was used when we compressed $F$ along $V$ to obtain $F_V$.
Then $\tilde{\V}$ is homeomorphic to $(\text{genus two surface})\times I$ whose $0$-level is $\bar{F}_{V}$.
Hence, $f(\tilde{\V})$ is homeomorphic to $(\text{genus two surface})\times I$ whose $0$-level is $f(\bar{F}_V)$.
Moreover, the $1$-level of $f(\tilde{\V})$ is the genus two component of $F_{f(V)}'$ if we compress $F'$ along $f(V)$ by using $f(N_\V(V))$ as the product neighborhood of $f(V)$ in $\V'$.
Therefore, we can easily check the follows.
\begin{enumerate}
\item $f(\bar{F}_{V})$ is obtained by pushing the genus two component of $F'_{f(V)}$ off into the interior of $\V'$.
\item $f(\bar{F}_{W})$ is obtained by pushing the genus two component of $F'_{f(W)}$ off into the interior of $\W'$ similarly.
\item $f(\bar{F}_{V W})$ is the union of components of $F'_{f(V) f(W)}$ having scars of both $f(V)$ and $f(W)$ similarly as $\bar{F}_{VW}$ because the images of the product neighborhoods of $V$ and $W$ in $\V$ and $\W$ which we used when we compressed $F$ along $V$ and $W$ to obtain $\bar{F}_{VW}$ of $f$ are also product neighborhoods of $f(V)$ and $f(W)$ in $\V'$ and $\W'$ respectively.
\end{enumerate}
Hence, $f(\mathbf{H})$ is the  generalized Heegaard splitting obtained by weak reduction from $(\V',\W';F')$ along the weak reducing pair $(f(V),f(W))$ by Lemma \ref{lemma-four-GHSs}.

This completes the proof of Claim A.
\end{proofN}

By Claim A, $(\mathbf{H}')=(f(\mathbf{H}))$ in $\overline{\mathcal{GHS}}_{F'}$, i.e. $f(\mathbf{H})$ is isotopic to $\mathbf{H}'$.
In other words, $f$ sends $\mathbf{H}$ into $\mathbf{H}'$ up to isotopy.\\

$(\Leftarrow)$ Suppose that $f$ sends  a representative $\mathbf{H}=\{\bar{F}_{V},\bar{F}_{VW},\bar{F}_{W}\}$ of an element $[\mathbf{H}]\in \mathcal{GHS}_F$ coming from weak reduction into a representative $\mathbf{H}'=\{\bar{F}_{V'}',\bar{F}_{V'W'}',\bar{F}_{W'}'\}$ of an element $[\mathbf{H}']\in \mathcal{GHS}_{F'}$ coming from weak reduction up to isotopy.
That is, we can isotope $f$ so that $f(\mathbf{H})=\mathbf{H}'$.

Let $(\bar{V},\bar{W})$ and $(\bar{V}',\bar{W}')$ be the centers of the components $\mathcal{B}$ and $\mathcal{B}'$ of $\DVW(F)$ and $\DVPWP(F')$ containing the weak reducing pairs $(V,W)$ and $(V',W')$ respectively.
Then we get two generalized Heegaard splittings $\bar{\mathbf{H}}=\{\bar{F}_{\bar{V}},\bar{F}_{\bar{V}\bar{W}},\bar{F}_{\bar{W}}\}$ and $\bar{\mathbf{H}}'=\{\bar{F}_{\bar{V}'}',\bar{F}_{\bar{V}'\bar{W}'}',\bar{F}_{\bar{W}'}'\}$ obtained by weak reductions from $(\V,\W;F)$ and $(\V',\W';F')$ respectively.
Here, (i) $(\bar{\mathbf{H}})=(\mathbf{H})$ in $\overline{\mathcal{GHS}}_F$ and (ii) $(\bar{\mathbf{H}}')=(\mathbf{H}')$ in $\overline{\mathcal{GHS}}_{F'}$ by considering the functions $\Phi_F$ and $\Phi_{F'}$.
That is, (i) induces that $\mathbf{H}$ is isotopic to $\bar{\mathbf{H}}$ by an isotopy $h_t$ such that $h_0$ is the identity and  $h_1(\mathbf{H})= \bar{\mathbf{H}}$, and therefore $\mathbf{H}'=f(\mathbf{H})$ is isotopic to $f(\bar{\mathbf{H}})$ by the isotopy $f\circ h_t\circ f^{-1}$.
Since $\mathbf{H}'$ is isotopic to $\bar{\mathbf{H}}'$ by (ii), we conclude that $f(\bar{\mathbf{H}})$ is isotopic to $\bar{\mathbf{H}}'$.
This means that we can isotope $f$ so that $f(\bar{\mathbf{H}})=\bar{\mathbf{H}}'$ by using the argument in Definition \ref{def-isotopy}.
Therefore, Theorem \ref{lemma-determine-GHSs} induces that we can isotope $f$ so that $f(F)=F'$.

This completes the proof.
\end{proof}

\section{The proof of Theorem \ref{theorem-main-2}\label{section4}}

In this section, we will prove Theorem \ref{theorem-main-2}.

\begin{definition}
Let $\mathcal{F}$ be the set of isotopy classes of weakly reducible, unstabilized Heegaard surfaces of genus three in $M$.
Now we define $\mathcal{GHS}=\cup_{[F]\in\mathcal{F}}\mathcal{GHS}_F$, where  we take exactly one representative $F$ for each isotopy class $[F]$.
Suppose that $F$ is isotopic to $F'$ in $M$ by an isotopy $h_t$ such that $h_0(F)=\operatorname{id}(F)=F$ and $h_1(F)=F'$.
Then we get a $1$-parameter family of Heegaard splittings $\{(\V_t,\W_t;F_t)\}_{0\leq t\leq 1}$ such that $F_0=F$ and $F_1=F'$ for $\V_t=h_t(\V)$, $\W_t=h_t(\W)$, and $F_t = h_t(F)$.
Let $\mathbf{H}$ be a representative of an element of $\mathcal{GHS}_{F}$ coming from weak reduction along a weak reducing pair $(V,W)$.
If we consider the weak reducing pair $(h_t(V),h_t(W))$ of $F_t$, then it gives the generalized Heegaard splitting $\mathbf{H}_t$ obtained by weak reduction from $(\V_t,\W_t;F_t)$.
Here, Claim A in Corollary \ref{corollary-GHS-HS} induces that $h_t(\mathbf{H})$ is a representative of an element of $\mathcal{GHS}_{F_t}$ coming from weak reduction  and $(\mathbf{H}_t)=(h_t(\mathbf{H}))$  in $\overline{\mathcal{GHS}}_{F_t}$ for $0\leq t\leq 1$.
Hence, we can see that (i) the isotopy $h_t$ sends $\mathbf{H}$ into $h_1(\mathbf{H})$ and (ii) $(\mathbf{H}_1)=(h_1(\mathbf{H}))$ in $\overline{\mathcal{GHS}}_{F'}$, i.e. the isotopy class $[\mathbf{H}]$ is the same as $[\mathbf{H}_1]$ and therefore each element of $\mathcal{GHS}_F$ belongs to $\mathcal{GHS}_{F'}$.
If we consider the isotopy $h_{1-t}\circ h_1^{-1}$ from $F'$ to $F$, then we can see that each element of $\mathcal{GHS}_{F'}$ belongs to $\mathcal{GHS}_{F}$ by the symmetric argument, i.e. $\mathcal{GHS}_{F}=\mathcal{GHS}_{F'}$.
This is why we take only one representative for each element of $\mathcal{F}$ in the union.

Let $\widetilde{\mathcal{GHS}}$ be the set of isotopy classes of the generalized Heegaard splittings consisting of two non-trivial Heegaard splittings of genus two.
Therefore, every representative of $\widetilde{\mathcal{GHS}}$ must be of the form $(\V_1,\V_2;T_1)\cup_t (\W_1,\W_2;T_2)$, where $\partial_-\V_2\cap \partial_-\W_1=t$ ($t$ is a torus or two tori) and the genera of $T_1$ and $T_2$ are both two.
\end{definition}

If we add the assumption that the minimal genus of Heegaard splittings in $M$ is three, then we get the following lemma.

\begin{lemma}\label{lemma-GHS-GHSS}
Let $M$ be an orientable, irreducible $3$-manifold admitting a weakly reducible, unstabilized Heegaard splitting of genus three and assume that the minimal genus of $M$ is three.
Then $\widetilde{\mathcal{GHS}}=\mathcal{GHS}$.
\end{lemma}

\begin{proof}
By Lemma \ref{lemma-four-GHSs}, $\mathcal{GHS}\subset\widetilde{\mathcal{GHS}}$ is obvious.

Suppose that $\mathbf{H}=(\V_1,\V_2;T_1)\cup_t (\W_1,\W_2;T_2)$ is a representative of an element of $\widetilde{\mathcal{GHS}}$ ,where $\partial_-\V_2\cap \partial_-\W_1=t$.
Then we can express $\V_2$ as the union of $\partial_-\V_2\times I$ and a $1$-handle attached to $\partial_-\V_2\times \{1\}$ and the symmetric argument also holds for $\W_1$ since they are genus two compression bodies with non-empty minus boundary.
Hence, we obtain a Heegaard splitting $(\V,\W;F)$ by the amalgamation of $(\V_1,\V_2;T_1)$ and $(\W_1,\W_2;T_2)$ along $t$ with respect to the $1$-handle structures of $\V_2$ and $\W_1$ and a suitable pair of projection functions as in Definition \ref{def-amalgamation}.
Let $D$ and $E$ be the cocore disks of the $1$-handles in the representations of $\W_1$ and $\V_2$ respectively.
Then we can see that $(D,E)$ is a weak reducing pair of $(\V,\W;F)$.
Moreover, if we observe the amalgamation $F$, then we can see the follows.
\begin{enumerate}
\item If both $\partial_-\V_2$ and $\partial_-\W_1$ are connected (so $t$ consists of a torus), then $F$ is the one obtained from $t$ by attaching two tubes corresponding to the $1$-handles of $D$ and $E$ to $t$.
\item If $\partial_-\V_2$ is disconnected and $\partial_-\W_1$ is connected (so $t$ consists of a torus), then $F$ is the one obtained from the union of $t$ and a torus $t'$ parallel to $\partial_-\W$ by attaching the tube corresponding to the $1$-handle of $D$ to $t$ and connecting $t$ and $t'$ by the tube corresponding to the $1$-handle of $E$ (see (b) of Figure \ref{fig-wr-amal-2}).\label{3333b}
\item If $\partial_-\V_2$ is connected and $\partial_-\W_1$ is disconnected (so $t$ consists of a torus), then we get the symmetric result of (\ref{3333b}).
\item If both $\partial_-\V_2$ and $\partial_-\W_1$ are disconnected and $\partial_-\V_2\cap\partial_-\W_1$ is connected (so $t$ consists of a torus), then $F$ is the one obtained from the union of $t$, a torus $t'$ parallel to $\partial_-\W$ and a torus $t''$ parallel to $\partial_-\V$ by connecting $t$ and $t'$ by the tube corresponding to the $1$-handle of $E$ and connecting $t$ and $t''$ by the tube corresponding to the $1$-handle of $D$.
\item If both $\partial_-\V_2$ and $\partial_-\W_1$ are disconnected and $\partial_-\V_2\cap\partial_-\W_1$ is disconnected, i.e. $\partial_-\V_2=\partial_-\W_1$ (so $t$ consists of two tori $t_1$ and $t_2$ ), then $F$ is the one obtained from $t$ attaching two tubes corresponding to the $1$-handles of $D$ and $E$ where each tube connects $t_1$ and $t_2$.
\end{enumerate}
In all cases, we can see that the genus of $F$ is three.
Here, we confirm that $(\V,\W;F)$ is unstabilized by the assumption that the minimal genus of $M$ is three.

By using the above observation, if we compress $F$ along $D$ or $E$ and consider the genus two component, then it is isotopic to $T_1$ or $T_2$ respectively and the union of components of $F_{DE}$ having scars of both $D$ and $E$ is isotopic to $t$ (see (c) and (a) of Figure \ref{fig-wr-amal-2} for type (b)-$\V$ GHS and we can draw similar figures for the other types of GHSs).
\begin{figure}
\includegraphics[width=12cm]{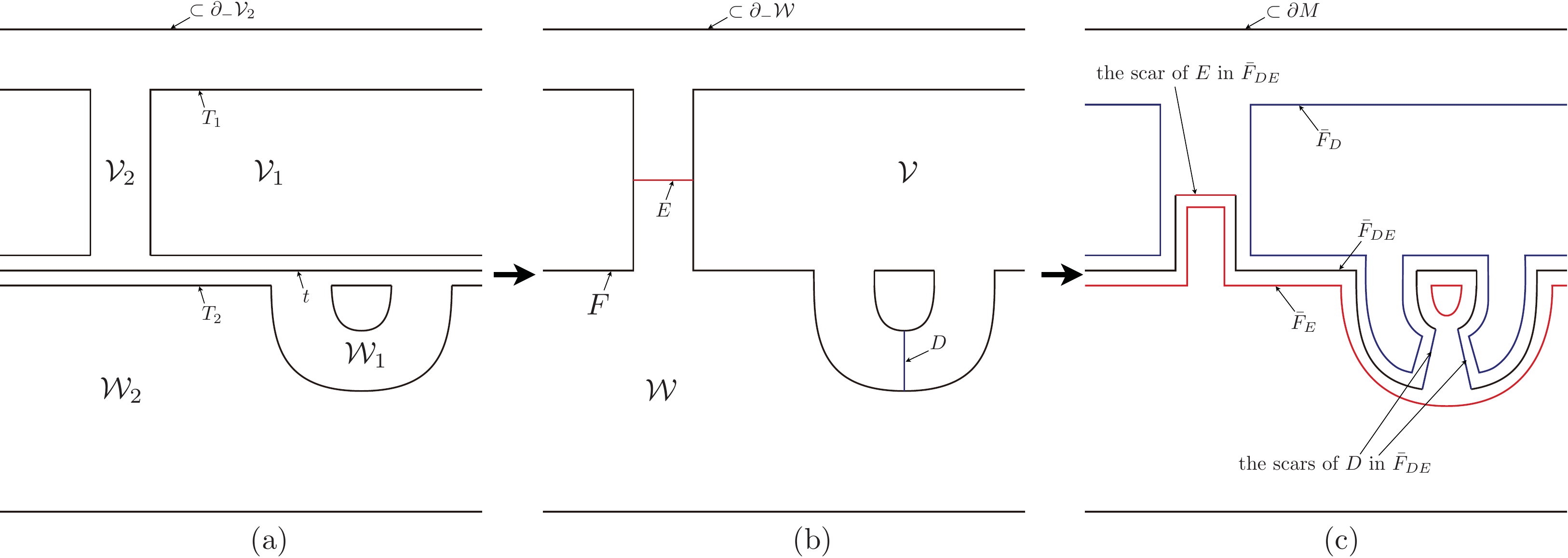}
\caption{(a)$\to$(b) : amalgamation (b)$\to$(c) : weak reduction\label{fig-wr-amal-2}}
\end{figure}
That is, the generalized Heegaard splitting $\{\bar{F}_D,\bar{F}_{DE},\bar{F}_E\}$ obtained by weak reduction from $(\V,\W;F)$ along the weak reducing pair $(D,E)$ is isotopic to $\mathbf{H}$ (refer to the last statement of Lemma \ref{lemma-four-GHSs}).
This means that the isotopy class $[\mathbf{H}]$ belongs to $\mathcal{GHS}_F$, i.e. $\widetilde{\mathcal{GHS}}\subset\mathcal{GHS}$.

This completes the proof.
\end{proof}

\begin{definition}
Let $\mathbf{H}=(\V_1,\V_2;T_1)\cup_t(\W_1,\W_2;T_2)$ ($\partial_-\V_2\cap \partial_-\W_1=t$) be a generalized Heegaard splitting whose isotopy class belongs to $\widetilde{\mathcal{GHS}}$ and $(\V,\W;F)$ be the Heegaard splitting obtained by amalgamation from $\mathbf{H}$ along $t$ with respect to suitable $1$-handle structures of $\V_2$ and $\W_1$ and a pair of projection functions $p_{N_0}$ and $p_{L_0}$ defined on $t\times I$s by using the notations in Definition \ref{def-amalgamation}.

If $\mathbf{H}$ is isotopic to a generalized Heegaard splitting $\mathbf{H}'$ by an isotopy $h_s$ such that $h_0(\mathbf{H})=\operatorname{id}(\mathbf{H})=\mathbf{H}$ and $h_1(\mathbf{H})=\mathbf{H}'$, then $h_s$ gives a $1$-parameter family of generalized Heegaard splittings $\{\mathbf{H}_s\}_{0\leq s \leq 1}$ for $\mathbf{H}_s=h_s(\mathbf{H})=(\V_1^s,\V_2^s;T_1^s)\cup_{t^s}(\W_1^s,\W_2^s;T_2^s)$.
If we consider the images of the $1$-handles of $\V_2$ and $\W_1$ in the relevant $1$-handle structures and the product structures of $t\times I$s determined by the pair $(p_{N_0},p_{L_0})$ of $h_s$, then
there would be the corresponding $1$-handles of $\V_2^s$ and $\W_1^s$ and the pair of projection functions $p^s_{N_0}$ and $p^s_{L_0}$ defined on $t^s\times I$s (see Figure \ref{fig-introduction-2}).
\begin{figure}
\includegraphics[width=10cm]{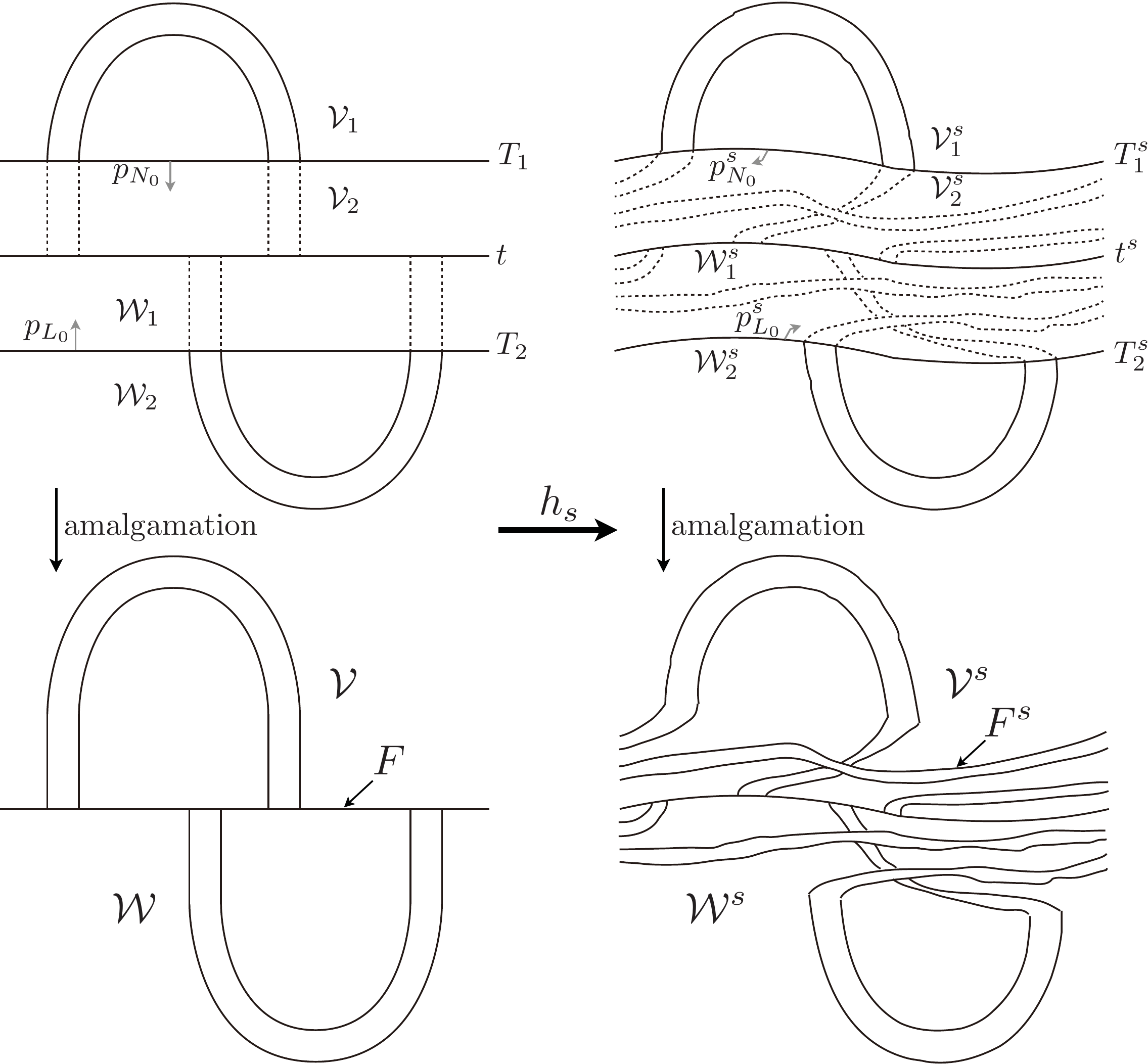}
\caption{$(\V_1^s,\V_2^s;T_1^s)\cup_{t^s}(\W_1^s,\W_2^s;T_2^s)$ and $F^s$\label{fig-introduction-2}}
\end{figure}
Hence, we get the $1$-parameter family of amalgamations $\{(\V_s,\W_s;F^s)\}$ by using these images, where each $(\V_s,\W_s;F^s)$ is obtained from $\mathbf{H}_s$, and we can see that $F^s=h_s(F)$, i.e. it gives an isotopy from $F=F^0$ to $F^1$.
Here, we can see that $F^1$ is obtained by amalgamation from $\mathbf{H}'=\mathbf{H}_1$.
This means that the isotopy class of the amalgamation obtained from $\mathbf{H}$ and that obtained from $\mathbf{H}'$ guaranteed  by Proposition \ref{lemma-Lackenby} are the same, i.e. an isotopy class $[\mathbf{H}]\in\widetilde{\mathcal{GHS}}$ gives a unique isotopy class $[F]$ of amalgamation.

Let  $\widetilde{\mathcal{GHS}}_{[F]}$ be the maximal subset of $\widetilde{\mathcal{GHS}}$ such that every element of $\widetilde{\mathcal{GHS}}_{[F]}$  gives the same isotopy class $[F]$ of amalgamation.
\end{definition}

\begin{definition}\label{def-induce}
Let $f$ be an orientation-preserving automorphism of an irreducible $3$-manifold $M$ that takes a weakly reducible, unstabilized Heegaard surface  $F_1$ of genus three into $F_2$, and $(\V_1,\W_1;F_1)$ and $(\V_2,\W_2;F_2)$ the relevant Heegaard splittings.
Since we can represent a compressing disk in $\V_1$ or $\W_1$ as the boundary curve in $F_1$ and $f$ is a homeomorphism, $f$ would translate the information of the compressing disks of $F_1$ into that of $F_2$.
Let $D_1$ and $D_2$ be compressing disks of $\V_1$.
If $[D_1]=[D_2]$ in $\D(F_1)$, then there is an isotopy $h_t$ defined on $M$ such that (i) $h_0$ is the identity, (ii) $h_1(D_1)=D_2$, and (iii) $h_t(F_1)=F_1$ for $0\leq t \leq 1$.
Without loss of generality, assume that $f(\V_1)= \V_2$ and let us consider the images $f(D_1)$ and $f(D_2)$ in $\V_2$.
Then $f\circ h_t \circ f^{-1}$ is an isotopy sending $f(D_1)$ into $f(D_2)$ and we can see that $f\circ h_t\circ  f^{-1}(F_2)=F_2$ for $0\leq t \leq 1$, i.e. $[f(D_1)]=[f(D_2)]$ in $\D(F_2)$.
Hence, we can well-define the map $f_\ast: \D(F_1)\to\D(F_2)$ by $f_\ast([D])=[f(D)]$.
Moreover, we can induce the follows easily.
\begin{enumerate}
\item $f_\ast$ induces a bijection between the set of vertices of $\D(F_1)$ and that of $\D(F_2)$.
\item $f_\ast$ sends each $k$-simplex in $\D(F_1)$ into the corresponding $k$-simplex in $\D(F_2)$ for $k\geq 0$.
\item $f_\ast$ sends each $k$-simplex in $D_{\V_1 \W_1}(F_1)$ into the corresponding $k$-simplex in $\D_{\V_2 \W_2}(F_2)$ for $k\geq 0$.
\item $f_\ast$ sends each component of $\D_{\V_1 \W_1}(F_1)$ into the corresponding component of $\D_{\V_2 \W_2}(F_2)$  (refer to Lemma \ref{lemma-just-BB}).
\end{enumerate}
Moreover, if $g$ is an orientation-preserving automorphism of $M$ that takes the Heegaard surface $F_2$ into $F_3$, then we can see $(g\circ f)_\ast([D])=[g\circ f(D)]=g_\ast([f(D)])=g_\ast(f_\ast([D]))$, i.e. $(g\circ f)_\ast=g_\ast\circ f_\ast$.
In addition, if we define $(f_\ast)^{-1}$ as $f_\ast^{-1}$ where $f_\ast^{-1}$ is the induced map coming from $f^{-1}$, then we get $(f_\ast)^{-1}\circ f_\ast =(\operatorname{id}_{F_1})_\ast$ and $f_\ast\circ (f_\ast)^{-1} =(\operatorname{id}_{F_2})_\ast$.
\end{definition}

Finally, we reach Theorem \ref{theorem-main-copy-2}.

\begin{theorem}[Theorem \ref{theorem-main-2}]\label{theorem-main-copy-2}
Let $M$ be an orientable, irreducible $3$-manifold having a weakly reducible,  genus three Heegaard splitting as a minimal genus Heegaard splitting.

Suppose that there is a correspondence between (possibly duplicated) two isotopy classes of $\widetilde{\mathcal{GHS}}$ by some elements of $Mod(M)$, say $[\mathbf{H}]\in \widetilde{\mathcal{GHS}}_{[F]}\rightarrow [\mathbf{H}']\in \widetilde{\mathcal{GHS}}_{[F']}$.
If $[f]$, $[g]\in Mod(M)$ give the same correspondence,  then there exists a representative $h$ of the difference $[h]=[g]\cdot[f]^{-1}$ satisfying the follows.

For a suitably chosen representative $F'\in[F']$,
\begin{enumerate}
\item $h$ takes $F'$ into itself and\label{thm123-1}
\item $h$ sends a uniquely determined weak reducing pair $(V',W')$ of $F'$ into itself up to isotopy (i.e. $h(V')$ is isotopic to $V'$ or $W'$ in the relevant compression body and $h(W')$ is isotopic to the other in the relevant compression body),
where $(V',W')$ is determined naturally when we obtain $F'$ by amalgamation from a representative $\mathbf{H}'$ of $[\mathbf{H}']$.\label{thm123-2} 
\end{enumerate}
Moreover, for any orientation-preserving automorphism $\tilde{h}$ of $M$ satisfying (\ref{thm123-1}) and (\ref{thm123-2}), there exist two elements in $Mod(M)$ giving the correspondence $[\mathbf{H}]\to[\mathbf{H}']$ such that $\tilde{h}$ belongs to the isotopy class corresponding to the difference between them.
\end{theorem} 

\begin{proof}
Let $\mathbf{H}$ and $\mathbf{H}'$ be arbitrarily chosen representatives of $[\mathbf{H}]$ and $[\mathbf{H}']$ respectively.
Here, we can represent each compression body of $\mathbf{H}$ intersecting $\overline{\operatorname{Thin}}(\mathbf{H})$ as $\partial_-\times I\cup(\textit{a $1$-handle})$ and the symmetric argument also holds for $\mathbf{H}'$.
With respect to the $1$-handle structures of these compression bodies and suitable pairs of projection functions on $\overline{\operatorname{Thin}}(\mathbf{H})\times I$s and $\overline{\operatorname{Thin}}(\mathbf{H}')\times I$s, we get the weakly reducible, unstabilized Heegaard splittings $(\V,\W;F)$ and $(\V',\W';F')$ of genus three obtained by amalgamations from $\mathbf{H}$ and $\mathbf{H}'$ along $\overline{\operatorname{Thin}}(\mathbf{H})$ and $\overline{\operatorname{Thin}}(\mathbf{H}')$ respectively.
Recall that $\mathbf{H}$ and $\mathbf{H}'$ are just generalized Heegaard splittings such that each consists of two Heegaard splittings of genus two and we only know the isotopy classes of the amalgamations are well-defined by Proposition \ref{lemma-Lackenby}.

If we use the proof of Lemma \ref{lemma-GHS-GHSS}, then $\mathbf{H}$ and $\mathbf{H}'$ are isotopic to the generalized Heegaard splittings obtained by weak reductions from $(\V,\W;F)$ and $(\V',\W';F')$ respectively.
In other words, we can isotope $F$ and $F'$ so that $\mathbf{H}$ and $\mathbf{H}'$ would be the generalized Heegaard splittings obtained by weak reductions from $(\V,\W;F)$ and $(\V',\W';F')$ respectively.
Let us realize these isotopies.
Let $(V,W)$ and $(V',W')$ be the weak reducing pairs coming from the cocore disks of the relevant $1$-handles used when we obtained the amalgamations $(\V,\W;F)$ and $(\V',\W';F')$ respectively.
If we thin the $1$-handle parts of $F$ and push $F$ off slightly to miss the thick levels of $\mathbf{H}$ if we need, then we can see that $\mathbf{H}$ itself is a generalized Heegaard splitting obtained by weak reduction from $(\V,\W;F)$ along $(V,W)$. (Refer to the last statement of Lemma \ref{lemma-four-GHSs} and see Figure \ref{fig-wr-amal-3}. 
We can draw the similar figures for the other cases among the five cases of amalgamations in the proof of Lemma \ref{lemma-GHS-GHSS}).
\begin{figure}
\includegraphics[width=10cm]{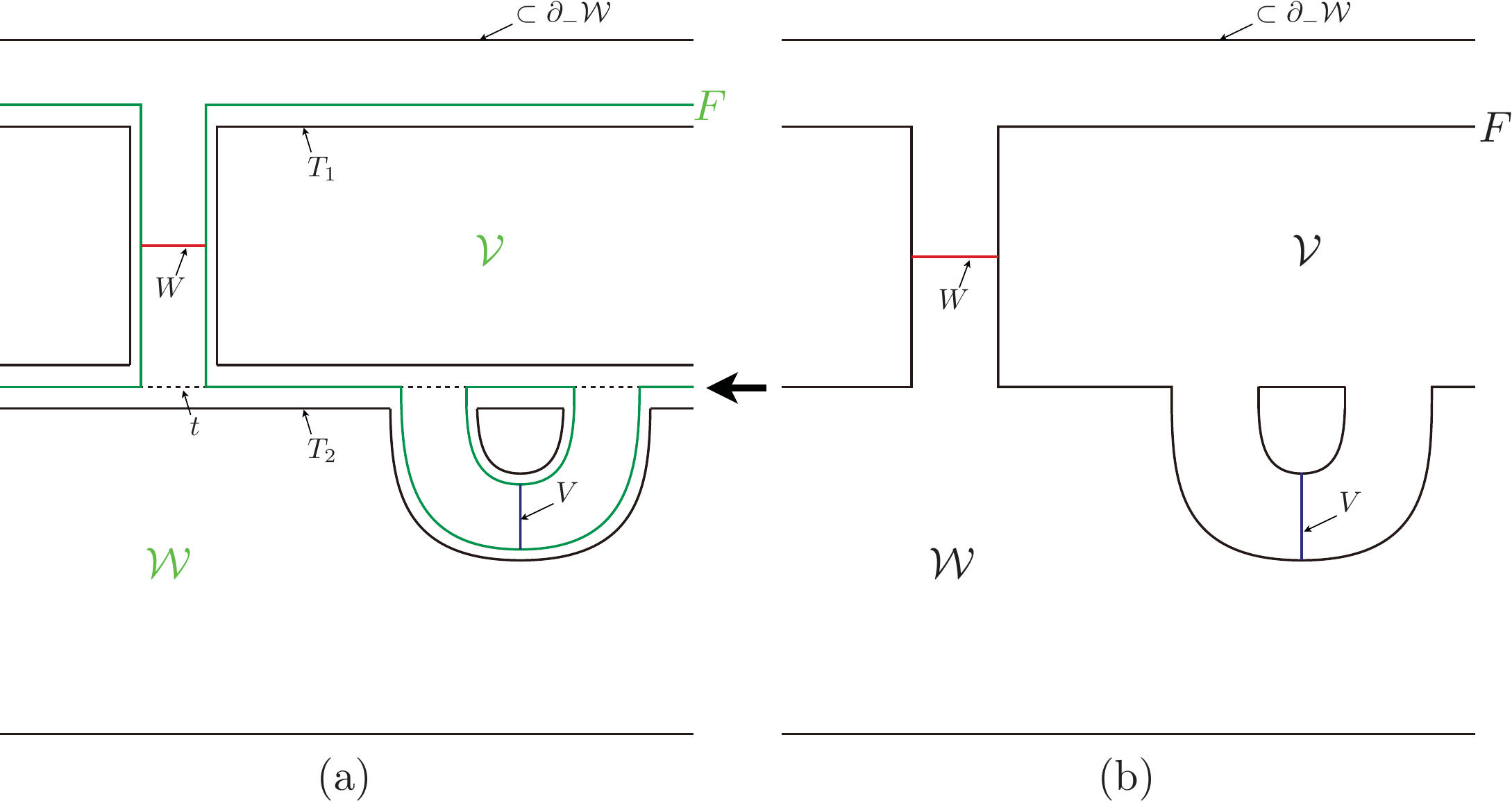}
\caption{(b)$\to$(a) :we can isotope $F$ so that $\mathbf{H}=\{T_1,t,T_2\}$ would be the GHS obtained by weak reductions from $(\V,\W;F)$.\label{fig-wr-amal-3}}
\end{figure}
The symmetric argument also holds for $F'$ and $\mathbf{H}'$ by using $(V',W')$.

After these isotopies of $F$ and $F'$, we can define the equivalent classes $(\mathbf{H})$ and $(\mathbf{H}')$ in $\overline{\mathcal{GHS}}_F$ and $\overline{\mathcal{GHS}}_{F'}$ respectively. 
From now on, we will use these embeddings of $F$ and $F'$.
Since $V$ and $W$ comes from the cocore disks of the $1$-handles, if any of them is separating in $\V$ or $\W$ after the amalgamation, then it cuts off $(\text{torus})\times I$ from $\V$ or $\W$ (recall the five cases of amalgamations in the proof of Lemma \ref{lemma-GHS-GHSS}).
This means that $(V,W)$ is the center of the component of $\DVW(F)$ which $(V,W)$ belongs to, say $\mathcal{B}$, by Lemma \ref{lemma-character-BB}.
Similarly, $(V',W')$ is the center of the component of $\DVPWP(F')$ which $(V',W')$ belongs to, say $\mathcal{B}'$.

By the assumption, $[f]([\mathbf{H}])=[\mathbf{H}']$ and $[g]([\mathbf{H}])=[\mathbf{H}']$ for $[f],[g]\in Mod(M)$.
Hence, there are representatives $f$ and $g$ of $[f]$ and $[g]$ respectively such that $f(\mathbf{H})=\mathbf{H}'$ and $g(\mathbf{H})=\mathbf{H}'$.
Therefore, we can isotope $f$ and $g$ so that (i) $f(F)=F'$ and $g(F)=F'$ and (ii) $f(\mathbf{H})=\mathbf{H}'$ and $g(\mathbf{H})=\mathbf{H}'$ by Theorem \ref{lemma-determine-GHSs}.

Recall that $\mathbf{H}$ and $\mathbf{H}'$ are generalized Heegaard splittings obtained by weak reductions from  $(\V,\W;F)$ and $(\V',\W';F')$ along the weak reducing pairs $(V,W)$ and $(V',W')$ respectively at this moment.
If we consider Claim A in Corollary \ref{corollary-GHS-HS}, then we can see $f^{-1}(\mathbf{H}')=\mathbf{H}$ is the generalized Heegaard splitting obtained by weak reduction from $(\V,\W;F)$ along the weak reducing pair determined by $\{f^{-1}(V'),f^{-1}(W')\}$, say $(\tilde{V},\tilde{W})$.
But Theorem \ref{theorem-structure} means that $(\tilde{V},\tilde{W})$ must belong to $\mathcal{B}$ since the embeddings of thick levels determined by $(\tilde{V},\tilde{W})$ are isotopic to those  determined by $(V,W)$ in the relevant compression bodies.
Indeed, $\{[\tilde{V}],[\tilde{W}]\}$ is $\{[V],[W]\}$ by Lemma \ref{lemma-center-center}.
That is, the induced map $f^{-1}_\ast:\D(F')\to \D(F)$ sends $\mathcal{B}'$ into $\mathcal{B}$ and $\{[V'],[W']\}$ into $\{[V],[W]\}$. 
Similarly,  if we consider $g(\mathbf{H})=\mathbf{H}'$, then the induced map $g_\ast:\D(F)\to \D(F')$ sends $\mathcal{B}$ into $\mathcal{B'}$ and $\{[V],[W]\}$ into $\{[V'],[W']\}$.

Let us consider the difference $[h]=[g]\cdot [f]^{-1}$.
Then $h=g\circ f^{-1}$ is a representative of $[h]$ such that $h(F')=F'$.
Moreover, the induced map $h_\ast=g_\ast\circ f^{-1}_\ast$ sends $\mathcal{B}'$ into $\mathcal{B}'$ and $\{[V'],[W']\}$ into $\{[V'],[W']\}$ itself by the previous observations.
This completes the proof of the existence of the representative $h$.
Since (i) the weak reducing pair $(V',W')$ is the center of $\mathcal{B}'$ which is unique in $\mathcal{B}'$ by definition and (ii) the component $\mathcal{B}'$ is uniquely determined by the equivalent class $(\mathbf{H}')\in\overline{\mathcal{GHS}}_{F'}$ as the preimage of the bijection $\Phi_{F'}$, the weak reducing pair $(V',W')$ is uniquely determined. 
This completes the proof of the first statement.

From now on, we will prove the last statement.

Consider an orientation-preserving automorphism $\tilde{h}$ of $M$ such that (i) $\tilde{h}$ takes $F'$ into itself and (ii) $\tilde{h}$ sends $(V',W')$ into itself up to isotopy.
This means that if we consider the embeddings of thick levels of the generalized Heegaard splitting $\tilde{\mathbf{H}}'$ obtained by weak reduction from $(\V',\W';F')$ along the weak reducing pair determined by $\{\tilde{h}(V'),\tilde{h}(W')\}$, then they are isotopic to those obtained by weak reduction from $(\V',\W';F')$ along $(V',W')$  in the relevant compression bodies, i.e. $(\tilde{\mathbf{H}}')=(\mathbf{H}')$ in $\overline{\mathcal{GHS}}_{F'}$.
Moreover, we can see that $(\tilde{h}(\mathbf{H}'))=(\tilde{\mathbf{H}}')$ in $\overline{\mathcal{GHS}}_{F'}$ by Claim A of Corollary \ref{corollary-GHS-HS}. i.e.  $(\tilde{h}(\mathbf{H}'))=(\mathbf{H}')$.
Therefore, we can isotope $\tilde{h}$ so that $\tilde{h}(\mathbf{H}')=\mathbf{H}'$.
Since there is at least one correspondence between $[\mathbf{H}]$ and $[\mathbf{H}']$ by an element $[f]\in Mod(M)$, choose a representative $f'$ of $[f]$ such that $f'(\mathbf{H})=\mathbf{H}'$.
Let $g'=\tilde{h}\circ f'$.
Then we can see that (i) $g'$ sends $\mathbf{H}$ into $\mathbf{H}'$ and (ii) $g'\circ f'^{-1}=\tilde{h}$.
Hence, $\tilde{h}$ is a representative of the difference $[g']\cdot[f']^{-1}$ between two elements $[g']$, $[f']\in Mod (M)$ giving the correspondence $[\mathbf{H}]\to[\mathbf{H}']$.
This completes the proof of the last statement.

This completes the proof.
\end{proof}

\section*{Acknowledgments}
This research was supported by BK21 PLUS SNU Mathematical Sciences Division.

\end{document}